\documentclass{article}[12pt]
\newcommand\independent{\protect\mathpalette{\protect\independenT}{\perp}}
\def\independenT#1#2{\mathrel{\rlap{$#1#2$}\mkern2mu{#1#2}}}
\makeatletter
\newcommand*{\rom}[1]{\expandafter\@slowromancap\romannumeral #1@}
\makeatother
\usepackage{amssymb}
\usepackage{amsmath}
\usepackage{natbib}
\usepackage{amsthm}
\usepackage{hyperref}
\usepackage{mathtools}
\usepackage[title]{appendix}
\newtheorem{theorem}{Theorem}
\newtheorem{corollary}{Corollary}[theorem]
\newtheorem{lemma}[theorem]{Lemma}
\newtheorem*{remark}{Remark}
\title{Statistical Independence and the Brockwell Transform---From an Integral Equation Perspective}
\author{Xingzhi Wang\footnote{\href{mailto:xingzhi-wang@uiowa.edu}{xingzhi-wang@uiowa.edu}}}

\begin{document}
    \maketitle

    \begin{abstract}
    Statistical independence is a notion ubiquitous in various fields such as in statistics, probability, number theory and physics.
    We establish the stability of independence for any pair of random variables by their corresponding Brockwell transforms \citep{brockwell2007universal}
    beyond the non-atomic condition that is naturally imposed on their distributions, 
    thereby generalizing the proposition originated by \cite{cai2022distribution}. 
    A central novelty in our work is to formulate the problem as a possibly new type of mathematical inverse problem, 
    which aims to claim that an integral equation has a unique solution in terms of its stochastic kernel---not under-determined contrary to the usual cases, 
    and also to design a recursive constructive scheme, combined with the property of the quantile function, 
    that solves the aforementioned integral equation in an iterative manner.
    \end{abstract}

\section{Introduction}

Testing statistical independence has been an important topic in statistical literature, 
and it may provide and suggest heuristics in numerical studies of different phenomenons in other areas as well.
Several novel measures of independence beyond the linear Pearson's correlation coefficient 
have been proposed in the twenty-first century, and may excel at finding out complex correlations in the data, 
for example, the distance correlation considered in \cite{szekely2007measuring}.
Corresponding new computable statistics to estimate the population-level construction may serve as tools for ranking variable in the problem of feature screening. 

However, various assumptions on the underlying distribution family of the tested pair of random variables 
are still necessary to validate the construction of the population-level correlation, 
where some of them are regarding finite moment conditions while others are restricting the specific type of distribution.
Thus, the case where the marginal distributions of the pair involve a mixture for which at least one of them does not have finite moments 
is still largely left unexplored, and such a situation may be prevalent across multiple areas like insurance claims in insurance companies and survival time in medical studies.

One good idea, built from a remark of \cite{szekely2007measuring}, is to transform the pair under the test into a corresponding bounded one, 
which at the same time preserves the independence property of its marginals. 
Such a transform is naturally related to the distribution function of the pair, 
which may also help to standardize the distribution.
Combined with the transform given in \cite{brockwell2007universal}, the main theoretical question under consideration for the unconditional case can be briefly stated as follows: 
Given a pair of real-valued random variables $\left(X, Y\right)$ with marginal distribution functions $\left(F, G\right),$ 
along with another pair $\left(U, V\right)$ that is independent of $\left(X, Y\right),$ let their Brockwell transforms be $Z_X \coloneqq \left(1 - U\right)F\left(X-\right) + UF\left(X\right)$ and
$Z_Y \coloneqq \left(1 - V\right)G\left(Y-\right) + VG\left(Y\right),$ then finally we ask for conditions upon which $Z_X \independent Z_Y$ is sufficient for $X \independent Y,$ 
where the symbol $\independent$ denotes statistical independence.

We enlarge the distribution family for the proposition put forward in (\cite{cai2022distribution}, Proposition $1$ and Theorem $8$) that claims the aforementioned sufficiency result holds if $X, Y$ are either both non-atomic or both categorical (and also $\left(U, V\right)$ is uniform on the unit square $\left(0, 1\right)^2$), 
thereby greatly extending the applicability of various independence measures that exist in the literature. 
We identify this problem as characterizing the uniqueness of solutions for an integral equation in its stochastic kernel,
and fulfill this characterization by a nontrivial ad-hoc construction, combined with some general (likely new) properties of the quantile function (the statistical counterpart for the non-increasing rearrangement of a distribution function) we shall derive. 
The proven result applies both conditionally and unconditionally, and it may be generalized to a pair of vectors in Euclidean spaces.

Throughout our discussion, denote the set of natural numbers by $\mathbb{N} \coloneqq \left\{ 0, 1, 2, \cdots \right\}$ 
and correspondingly the set of positive integers by $\mathbb{N}^\star \coloneqq \left\{ 1, 2, 3, \cdots \right\}.$ For any set $A,$ 
let $\mathrm{id.}_A$ be the identity mapping on $A,$ and $\sqcup$ is defined to be the disjoint union logical operation on sets. 
Given a measure space $\left(\Omega, \mathcal{A}, \mu\right)$ and a (binary-valued) logical proposition $P\left(\omega\right)$ for $\omega \in \Omega,$ 
we write $P\ \mathrm{is\ true}\ \left[\mu\right]$ if there exists a $\mu$-null set $N \in \mathcal{A}$ 
such that the set comprehension $\left\{ \omega \in \Omega : P\left(\omega\right) \mathrm{is\ false}\right\}$ is contained in $N.$

\subsection{Roadmap and Sketches of Proofs}

The major consequences of our development are presented in Corollary~\ref{cor::1} and Corollary~\ref{cor::2} below, 
which establish the equivalence of independence between tuples of real-valued random variables 
and certain standardized transformations, i.e., Brockwell transforms, related to their marginal (cumulative) distribution functions
under a well-ordering condition for their sets of atoms. On the other hand, the crux of the underlying issue is stated in Theorem~\ref{thm::4}, 
which claims that given a real-valued random variable $X : \Omega \rightarrow \mathbb{R}$ with distribution function $F,$ a Borel space valued random map $\widetilde{Y} : \Omega \rightarrow M$ (The Borel space $M$ is equipped with the class $\mathcal{F}$ as its $\sigma$-algebra.)
and another auxiliary real-valued random variable $U : \Omega \rightarrow \mathbb{R},$ supported on $\left[0, +\infty\right)$ but places some mass over $\left(0, 1\right]$ with distribution function $H$ ($H\left(0\right) = 0, H\left(1\right) > 0$), which is independent of $X$ as well as conditionally independent of $\widetilde{Y}$ given $X.$ Denote $Z = \left(1 - U\right)F\left(X-\right) + UF\left(X\right)$ by the Brockwell transfrom of $X.$ Then we have 
    \begin{eqnarray}
        Z \independent \widetilde{Y} \Rightarrow X \independent \widetilde{Y} \label{(33)}
    \end{eqnarray} if there exists a strictly increasing or a strictly decreasing transformation from the set of atoms for $X$ into the set of positive integers $\mathbb{N}^\star.$
Since Theorem~\ref{thm::4} is formulated in a one-sided fashion, we could invoke this sufficiency one-at-a-time to obtain Corollary~\ref{cor::1}, 
whereas for Corollary~\ref{cor::2}, $\widetilde{Y}$ to apply therein is a pair, and the statement (i) in the last item of Lemma~\ref{lma::2} is additionally required for the claim of the conditional independence.

Starting from an equation solving perspective, our idea to interpret the implication~(\ref{(33)}) is the following: The condition $Z \independent \widetilde{Y}$ is equivalent to 
an integral equation of the form 
    \begin{equation}
        T_{\mu_F} \kappa = \beta, \label{(34)}
    \end{equation} where 
    \begin{enumerate}
        \item $\mu_F \coloneqq PX^{-1}$ is the Borel probability measure of $X$ and $\mathcal{K}$ denotes the class of 
              maps $\kappa: \mathbb{R} \times \mathcal{F} \rightarrow \left[-\infty, +\infty\right]$ such that
              \begin{enumerate}
                 \item[(i)] $\kappa\left(\cdot, A\right)$ is $\sigma\left(\tau_\rho\right)$-measurable for all $A \in \mathcal{F},$ where $\rho$ denotes the ordinary ruler metric on the real line $\mathbb{R},$ and
                 \item[(ii)] for every real number $x \in \mathbb{R}, \kappa\left(x, \cdot\right)$ is a finite signed measure on the Borel space $\left(M, \mathcal{F}\right)$ with uniformly bounded (total) variations
                             $$\sup_{x \in \mathbb{R}} \left\lVert \kappa\left(x, \cdot\right) \right\rVert = \sup_{x \in \mathbb{R}} \left|\kappa\right|\left(x, M\right) < +\infty.$$  
              \end{enumerate}

              The map $T_{\mu_F} : \mathcal{K} \rightarrow \mathcal{K}$ sends a fixed $\kappa \in \mathcal{K}$ to $T_{\mu_F} \kappa$ as follows:
              \begin{align}
                  T_{\mu_F} \kappa: \mathbb{R} \times \mathcal{F} &\rightarrow \mathbb{R} \nonumber \\
                                                \left(z, A\right) &\mapsto \int_{\mathbb{R}} \left(I_1 + I_2\right)\left(x, z\right)\kappa\left(x, A\right) \mu_F\left(dx\right), \label{(35)}
              \end{align} where 
              \begin{eqnarray}
                  I_1\left(x, z\right) &\coloneqq& H\left(\frac{z - \mu_F\left(\left(-\infty, x\right)\right)}{\mu_F\left(\left\{ x \right\}\right)}\right)\mathbb{I}_{D_F}\left(x\right), \label{(48)} \\
                  I_2\left(x, z\right) &\coloneqq& \mathbb{I}_{\left(-\infty, \overleftarrow{F}\left(z\right)\right] - D_F}\left(x\right) \label{(49)}
              \end{eqnarray} and $D_F$ is the set of (jump) discontinuities for $F$ as well as $\overleftarrow{F}$ denotes the right-continuous generalized inverse of $F.$ ($\left(-\infty, \overleftarrow{F}\left(z\right)\right]$ is understood as the empty set $\varnothing$ if $z \leq 0$ while it is as the universe $\mathbb{R}$ if $z \geq 1.$)

              Notice that\footnote{We leave the proofs of these observations to Appendix~\ref{pfa} to~\ref{pfb}.}
              \begin{enumerate}
                  \item[(i)] $\left(\mathcal{K}, \left\lVert \cdot \right\rVert\right)$ is a Banach space over $\mathbb{R},$ where $\left\lVert \kappa \right\rVert \coloneqq \sup_{x \in \mathbb{R}} \left\lVert \kappa\left(x, \cdot\right) \right\rVert.$ 
                             We may identify two elements $\kappa_1 \sim_{\mu_F} \kappa_2$ in $\mathcal{K}$ by $\mu_F$ whenever $\kappa_1\left(\cdot, A\right) = \kappa_2\left(\cdot, A\right)$ $\mu_F$-a.e. for all $A \in \mathcal{F}.$ (Each $\mu_F$ exceptional set may depend on $A.$)

                  \item[(ii)] $T_{\mu_F}$ is a linear operator on $\mathcal{K}.$\footnote{Therefore, the notation $T_{\mu_F} \kappa$ is justified instead of $T_{\mu_F}\left(\kappa\right).$}
              \end{enumerate}
        \item The specific $\kappa$ that appears in this setting is of the form $\kappa\left(x, A\right) \coloneqq P\left(\widetilde{Y} \in A\ |\ X = x\right)$ while $\beta\left(x, A\right) \coloneqq PZ^{-1}\left(\left(-\infty, x\right]\right) \cdot P\widetilde{Y}^{-1}\left(A\right),$ 
              both of which lie in the subclass $\mathcal{K}_{\geq 0}$ of all stochastic kernels from $\left(\mathbb{R}, \sigma\left(\tau_\rho\right)\right)$ to $\left(M, \mathcal{F}\right).$
        \item The necessity of $Z \independent \widetilde{Y}$ for $X \independent \widetilde{Y}$ yields one particular solution for the solution space of~(\ref{(34)}), i.e., the degenerate kernel (a kernel that jumps according to a distribution independent of the starting point) $\kappa_0\left(x, A\right) = P\widetilde{Y}^{-1}\left(A\right).$ 
              This is paraphrased as $T_{\mu_F} \kappa_0 = \beta.$
\end{enumerate}

In this language, to prove that $Z \independent \widetilde{Y}$ is sufficient for $X \independent \widetilde{Y}$ is equivalent 
to show that $\kappa_0$ is the unique solution to the integral equation~(\ref{(34)}). 
Note that the problem of interest here may not be applicable in the realm of Fredholm-Riesz-Schauder theory 
since usually the integral kernel in Fredholm type equation may not be unique given an integrand and a right-hand side term.

A novel scheme we create to solve~(\ref{(34)}) is to take full advantage of the free variable $z$ 
which helps transform it into a differential equation problem, and also to solve it in an inductive manner:
\begin{enumerate}
    \item[\textbf{Step 1.}] 
        First, we utilize the particular solution $\kappa_0$ given to turn this equation into a homogeneous form 
        \begin{equation}
            T_{\mu_F} \phi = 0 \label{(38)}
        \end{equation} with $\phi = \kappa - \kappa_0.$ 
        It then remains to show that $\mathrm{ker}\left(T_{\mu_F}\right) = \left\{ 0 \right\}$ in the identified Banach space $\mathcal{K} / \sim_{\mu_F}.$

        In the proof of Theorem~\ref{thm::4} to follow, we proceed with the specific instance of $\phi \in \mathrm{ker}\left(T_{\mu_F}\right)$ in the set-up~(\ref{(33)}), but the argument works for all kernels $\phi \in \mathrm{ker}\left(T_{\mu_F}\right).$ 
        In fact, if we do not require $\phi$ to be a member of $\mathcal{K},$ then $\phi = 0$ holds $\mu_F$-a.e. whenever $T_{\mu_F} \phi = 0$ as long as $\phi\left(\cdot, A\right)$ is $\sigma\left(\tau_\rho\right)$-measurable and $\int_{\mathbb{R}} \left|\phi\left(\cdot, A\right)\right| d\mu_F < +\infty$ 
        for each $A \in \mathcal{F}.$\footnote{$\int_{\mathbb{R}} \left|\phi\left(\cdot, A\right)\right| d\mu_F \leq \sup_{x \in \mathbb{R}} \left|\phi\right|\left(x, A\right) \leq \left\lVert \phi \right\rVert < +\infty$ if $\phi \in \mathcal{K}.$}
    \item[\textbf{Step 2.}] 
        Secondly, the integrands $I_1$ and $I_2$ contribute separately to each resulting kernel $T_{\mu_F} \kappa$ in~(\ref{(35)}), 
        where $I_1$ is involved with the purely atomic part of $\mu_F,$ whereas $I_2$ with its non-atomic part.

        Because of this observation, we need a starting point for the free variable $z$ to vary at which the effects of $I_1$ and of $I_2$ are (already) isolated in its sum form. 
        This is exactly the rationale for the well-ordering condition we impose, in particular by the order $\leq$ of the real line $\mathbb{R}$ that is compatible with the form of those initial intervals that appear in $I_1$ and in $I_2,$ on the set of (jump) discontinuities $D_F$ for the distribution function $F.$  
        Such a well-ordering condition leads naturally to an inductive proof on the set of natural numbers $\mathbb{N}.$

        Lemma~\ref{lma::1} below documents a simple consequence of the well-ordering by $\leq$ relation for (at most) countable sets on the real line.
    \item[\textbf{Step 3.}] 
        Following the motivation in Step $2,$ we assume that the first atom of the probability measure $\mu_F$ is $x_1 \leq +\infty$ ($x_1 = +\infty$ denotes the case when it is non-atomic, i.e., $\mu_F$ does not have any atom at all), 
        then $\mu_F$ is non-atomic on $\left(-\infty, x_1\right)$ and the integrand $I_1$ vanishes if $z \in \left(0, \mu_F\left(\left(-\infty, x_1\right)\right)\right].$

        Now, we could focus on the integrand $I_2,$ which turns out to be the most difficult part of the proof. 
        It is natural to substitute $z = F\left(w\right) \in \left(0, F\left(x_1-\right)\right)$ for any $w \in \left(-\infty, x_1\right)$ in~(\ref{(49)}) to expect for cancellation 
        $\left(\overleftarrow{F} \circ F\right)\left(w\right) = w,$ so that we obtain a rich class of test functions, 
        i.e., indicators of initial intervals of the form $\mathbb{I}_{\left(-\infty, w\right]},$ in the integrand to vary over.

        Unless $F$ is strictly increasing over $\left(-\infty, x_1\right),$ the above hope can only be realized in a $\mu_F$-almost everywhere sense, 
        which we derive as a more general conclusion stronger than the well-known statement about the probability integral transform in (i) of $5$ in Lemma~\ref{lma::2}. 
        (The above identity $\left(\overleftarrow{F} \circ F\right)\left(w\right) = w$ only makes rigorous sense for $\mu_F$-almost all $w \in F^{-1}\left(\left(0, 1\right)\right) \cap \left(-\infty, x_1\right),$ and the statement (ii) of $5$ in Lemma~\ref{lma::2} is in place to take into account the members of $F^{-1}\left(\left\{ 0 \right\}\right)$ as well as of $F^{-1}\left(\left\{ 1 \right\}\right).$)

        Then we have
        \begin{equation}
            \int_{\mathbb{R}} \mathbb{I}_{\left(-\infty, w\right]}\left(x\right)\phi\left(x, A\right) \mu_F\left(dx\right) = 0 \label{(36)}
        \end{equation} for every $A \in \mathcal{F}$ and for $\mu_F$-almost every $w$ in $\left(-\infty, x_1\right).$ 
        Heuristically, we would differentiate with respect to the variable $w$ on both sides to infer that 
        \begin{equation}
            \phi\left(w, A\right) = 0. \label{(37)}
        \end{equation}

        However, no standard fundamental theorem of calculus type of result, for example, Lebesgue differentiation theorem in Section $11.52$ pp. $360$ of \cite{titchmarsh1939theory}, claims 
        that we could safely do so if the variable $w$ cannot vary everywhere even when $\mu_F$ is dominated by the Lebesgue measure $\lambda$ on the class of Borel subsets of $\left(-\infty, x_1\right).$

        On the one hand, in order to deal with the difficulty of differentiation with respect to a general Borel probability measure on the real line, 
        we require some more tools from measure theory (for instance, the notions of a (Vitali) covering relation in a metric space and of an (associated) derivate for a Borel regular outer measure that is finite on bounded sets). 
        Preliminary definitions and developments are prepared right before and within Lemma~\ref{lma::3}.
        The main goal of them is to show:
        \begin{enumerate}
            \item[(i)] 
                The set of pairs of the form $\left(x, \left[x - r, x + r\right]\right)$ for each point $x \in \mathbb{R}$ with any closed ball centered at $x$ of radius $r > 0$ is a valid Vitali relation on the real line 
                with respect to any Borel regular outer measure that is finite on bounded sets.
            \item[(ii)] 
                Starting from a valid distribution function $F$ on the real line $\mathbb{R},$ the outer measure constructed via the Carath\`eodory's method 
                is Borel regular, and it extends the Borel probability measure $\mu_F$ on the Borel $\sigma$-algebra we work with throughout this work.
        \end{enumerate}

        On the other hand, in order to address the complication arising from the variable $w,$ 
        we propose to construct a nested sequence of closed intervals in $\left(-\infty, x_1\right)$ 
        around every fixed $w$ outside a $\mu_F$-null subset of $\left(-\infty, x_1\right),$
        $\left[w_{n, 1}, w_{n, 2}\right] \ni w$ for all positive integers $n \in \mathbb{N}^\star,$ such that 
        \begin{eqnarray}
            \int_{\mathbb{R}} \mathbb{I}_{\left(-\infty, w_{n, i}\right]}\left(x\right) \phi\left(x, A\right) \mu\left(dx\right) = 0 \nonumber
        \end{eqnarray} for any index $i \in \left\{ 1, 2 \right\}$ and more importantly, at the same time $$\mu\left(\left[w_{n, 1}, w_{n, 2}\right]\right) > 0.$$

        Intuitively, the measure in the denominator of a measure ratio should be positive if we are able to differentiate equation~(\ref{(36)}) at $w$ at all.
        We will give such a construction that satisfies both conditions by (i) of $5$ in Lemma~\ref{lma::2} and by one implication of $3$ in Lemma~\ref{lma::3}. 
        Further, with $3$ of Lemma~\ref{lma::3}, we proceed to differentiate~(\ref{(36)}) with respect to the variable $w$ on both sides, which yields the claim~(\ref{(37)}).
        Note that, for every Borel probability measure $\mu_F$ that dominates the Lebesgue measure $\lambda,$ it always assumes positive values on nonempty intervals and hence the second condition is automatic.

    \item[\textbf{Step 4.}] 
        Having established that $\phi\left(\cdot, A\right) = 0$ holds $\mu_F$-almost everywhere on $\left(-\infty, x_1\right)$ for each $A \in \mathcal{F}$ in the previous step, we now turn to the integrand $I_1\left(x, \cdot\right)$ at $x = x_1.$ We distinguish three cases:
        \begin{itemize} 
            \item[\textbf{Case (i):}] $F\left(x_1\right) < 1$\textbf{.}
                
                As $H\left(0\right) = 0,$ setting $z = F\left(x_1\right)$ in~(\ref{(48)}) yields $I_1\left(x, z\right) = H\left(1\right)\mathbb{I}_{\left\{ x_1 \right\}}\left(x\right).$ 
                For the integrand $I_2,$ using $4$ of Lemma~\ref{lma::2} and $\overleftarrow{F}\left(F\left(x_1\right)\right) \geq x_1,$ we get
                $$\left(-\infty, \overleftarrow{F}\left(F\left(x_1\right)\right)\right] - D_F = \left(-\infty, x_1\right) \sqcup \left(\left(x_1, \overleftarrow{F}\left(F\left(x_1\right)\right)\right] - D_F\right),$$
                where there is a gap between $\left(-\infty, \overleftarrow{F}\left(F\left(x_1\right)\right)\right] - D_F$ and $\left(-\infty, x_1\right)$ in which $\phi$ vanishes ($\mu_F$-almost everywhere). 
                Equipped with the properties of $\overleftarrow{F}$ in $8$ of Lemma~\ref{lma::2}, we could close the gap $\left(x_1, \overleftarrow{F}\left(F\left(x_1\right)\right)\right] - D_F = \varnothing\ \left[\mu_F\right],$ 
                which implies $I_2\left(\cdot, F\left(x_1\right)\right)\phi\left(\cdot, A\right) = 0\ \mu_F$-almost everywhere over $\left(-\infty, \overleftarrow{F}\left(F\left(x_1\right)\right)\right]$ and hence
                $$0 = \left(T_{\mu_F} \phi\right)\left(x_1, A\right) = H\left(1\right)\phi\left(x_1, A\right)\mu_F\left(\left\{ x_1 \right\}\right).$$
                
                Due to the assumptions $H\left(1\right) > 0$ and $\mu_F\left(\left\{ x_1 \right\}\right) > 0$ ($x_1 < +\infty$), we have $\phi\left(x_1, A\right) = 0.$
            \item[\textbf{Case (ii):}] $F\left(x_1\right) = 1$ \textbf{and} $x_1 < +\infty$\textbf{.} 
                
                In this case, $x_1$ is the only (jump) discontinuity of $F,$ i.e., $D_F = \left\{ x_1 \right\}.$ Similarly, we obtain $I_1\left(x, z\right) = H\left(1\right)\mathbb{I}_{\left\{ x_1 \right\}}\left(x\right)$ 
                by letting $z = 1,$ and we can also prove $\left(x_1, +\infty\right) - D_F = \varnothing\ \left[\mu_F\right].$ Then $I_2\left(\cdot, 1\right)\phi\left(\cdot, A\right) = 0\ \left[\mu_F\right],$ so that $\phi\left(x_1, A\right) = 0$ holds again. 
            \item[\textbf{Case (iii):}] $x_1 = +\infty$\textbf{.}
                
                As mentioned earlier, this scenario is clear because there is no atom for $\mu_F.$
        \end{itemize}
    \item[\textbf{Step 5.}] 
        Combining Step $3$ and Step $4,$ we have shown that, for each $A \in \mathcal{F}, \phi\left(\cdot, A\right) = 0\ \mu_F$-almost everywhere on $\left(-\infty, x_1\right),$ 
        and $\phi\left(x_1, A\right) = 0$ so long as $x_1 < +\infty.$ The remainder of the proof follows in an inductive manner as alluded to.
\end{enumerate}

Perhaps a new takeaway from this work is that a form of statistical independence corresponds to a certain integral equation in terms of its kernel, 
and the structure of its solution space determines how we may characterize such an independence property, 
for instance, any condition that necessitate the stated independence may define one particular solution of that equation.

\section*{Acknowledgments}

The author would like to genuinely thank Professor Kung-Sik Chan 
who introduced the problem of feature screening and its recent research methodology developed by Professor Runze Li's group to the author.
He also read part of this manuscript, and provided several useful suggestions to improve the use of plain language and of mathematical symbols 
as well as to maintain a proper balance between the two. 

The core of this work was completed during the Fall $2022$ semester to the Spring $2023$ semester when the author was receiving a research assistantship for one academic year as a graduate student at the University of Iowa.

\section{Preliminary Lemmas}

We defer the proofs of Lemma~\ref{lma::1} to~\ref{lma::3} to Appendix~\ref{pflma::1} to~\ref{pflma::3}. Among these lemmas, Claim $5$ and the statement (ii) of Claim $9$ in Lemma~\ref{lma::2} may be new so far as we know.

\begin{lemma}
    \label{lma::1}
    Given an at most countable subset of the real line $A\subseteq \mathbb{R}$ such that it is well-ordered by $\leq.$ 
    \begin{enumerate}
        \item $A$ is a finite subset.

              Let $m = \# A,$ then we may obtain a permutation $A = \left\{ x_n \right\}_{n = 1}^m$ such that 
              $$-\infty = x_0 < x_1 < x_2 < \cdots < x_m < x_{m + 1} = +\infty.$$
        \item $A$ is a countably infinite subset.

              In this case, we can also arrange $A = \left\{ x_n \right\}_{n = 1}^\infty$ such that 
              $$-\infty = x_0 < x_1 < x_2 < \cdots < x_n < x_{n + 1} < \cdots < x_{\infty} = +\infty.$$ 
    \end{enumerate}
\end{lemma}

\begin{remark}
    \label{rmk::1} 
    Similar result also holds in an antitone manner: When $A$ is well-ordered by $\geq$ on $\mathbb{R},$ then we may apply the above statement by considering its reflection $-A.$
\end{remark}

\begin{lemma}
    \label{lma::2} 
    Given a distribution function on the real line $F : \mathbb{R} \rightarrow \left[0, 1\right],$ along with its left-limit $F\left(x-\right) = \lim_{h \rightarrow 0-} F\left(x + h\right)$ for every real number $x,$ define the left-continuous, and right-continuous, version of the quantile function for $F$ as 
    \begin{align}
        \overrightarrow{F} : \left(0, 1\right) &\rightarrow \mathbb{R} \nonumber \\
                                      y &\mapsto \inf \left\{ x \in \mathbb{R} : F\left(x\right) \geq y \right\} \nonumber
    \end{align} and
    \begin{align}
        \overleftarrow{F} : \left(0, 1\right) &\rightarrow \mathbb{R} \nonumber \\
                                      y &\mapsto \sup \left\{ x \in \mathbb{R} : F\left(x\right) \leq y \right\}, \nonumber
    \end{align} respectively.   

    Then the following set of claims hold:
    \begin{enumerate}
        \item
            $\overrightarrow{F}, \overleftarrow{F}$ are well-defined, and non-increasing.
        \item
            $\overrightarrow{F} \leq \overleftarrow{F}$ and 
            \begin{eqnarray}
                E \coloneqq \left\{ y \in \left(0, 1\right) : \overrightarrow{F}\left(y\right) < \overleftarrow{F}\left(y\right) \right\} &=& \left\{ y \in \left(0, 1\right) : \mathrm{int}\left(F^{-1}\left(\left\{ y \right\}\right)\right) \neq \varnothing \right\} \nonumber \\
                                                                                                                                  &=& \left\{ y \in \left(0, 1\right) : \mathrm{int}\left(F^{-1}\left(\left\{ y \right\}\right)\right) = \left(\overrightarrow{F}\left(y\right), \overleftarrow{F}\left(y\right)\right) \right\} \nonumber
            \end{eqnarray} for which $E$ is at most countable.
        \item
            For all $y \in E,$ we have $F\left(\overrightarrow{F}\left(y\right)\right) = F\left(\overleftarrow{F}\left(y\right)-\right) = y.$
        \item
            Let $G \coloneqq \left\{ x \in \mathbb{R} : 0 < F\left(x\right) < 1 \right\},$ then $\left(\overrightarrow{F} \circ F|_G\right) \leq \mathrm{id.}_G, \left(\overleftarrow{F} \circ F|_G\right) \geq \mathrm{id.}_G$ and 
            \begin{eqnarray}
                \left\{ x \in G : \left(\overrightarrow{F} \circ F|_G\right)\left(x\right) < x \right\} = \sqcup_{y \in E} E_{1, y}, \nonumber \\
                \left\{ x \in G : \left(\overleftarrow{F} \circ F|_G\right)\left(x\right) > x \right\} = \sqcup_{y \in E} E_{2, y}, \nonumber 
            \end{eqnarray} where 
            \begin{eqnarray}
                E_{1, y} &\coloneqq& \left\{\begin{array}{ll}
                                                \left(\overrightarrow{F}\left(y\right), \overleftarrow{F}\left(y\right)\right], & \mathit{if}\ F\ \mathit{is\ \left(left\right)\ continuous\ at}\ \overleftarrow{F}\left(y\right), \nonumber \\
                                                \left(\overrightarrow{F}\left(y\right), \overleftarrow{F}\left(y\right)\right), & \mathit{if}\ F\ \mathit{jumps\ at}\ \overleftarrow{F}\left(y\right), \nonumber
                                            \end{array}\right. \nonumber \\
                E_{2, y} &\coloneqq& \left[\overrightarrow{F}\left(y\right), \overleftarrow{F}\left(y\right)\right).
            \end{eqnarray}
        \item Let $D_F \coloneqq \left\{ x \in \mathbb{R} : F\left(x-\right) < F\left(x\right) \right\}$ be the set of (jump) discontinuities of $F,$ and denote the Lebesgue-Stieljies probability measure on $\mathbb{R}$ corresponding to $F$ by $\mu_F.$ Then,
            \begin{itemize}
                \item[(i)]
                    \begin{eqnarray}
                        \left\{ x \in G : \left(\overrightarrow{F} \circ F|_G\right)\left(x\right) < x \right\} &\subseteq& D_F^\mathsf{c}, \nonumber \\
                        \mu_F\left(\left\{ x \in G : \left(\overrightarrow{F} \circ F|_G\right)\left(x\right) < x \right\}\right) &=& 0 \nonumber
                    \end{eqnarray} and
                    \begin{eqnarray}
                        \mu_F\left(\left\{ x \in G : \left(\overleftarrow{F} \circ F|_G\right)\left(x\right) > x \right\} \cap \left(a, b\right)\right) &=& 0 \nonumber
                    \end{eqnarray} for each open interval $\left(a, b\right) \subseteq D_F^\mathsf{c}.$
                \item[(ii)]
                    $F^{-1}\left(\left\{ 0 \right\}\right) - D_F = F^{-1}\left(\left\{ 0 \right\}\right), \mu_F\left(F^{-1}\left(\left\{ 0 \right\}\right)\right) = \mu_F\left(F^{-1}\left(\left\{ 1 \right\}\right) - D_F\right) = 0.$
            \end{itemize}
            In particular, for any real-valued random variable $X$ with a continuous distribution function $F,$ we have $\overrightarrow{F}\left(F\left(X\right)\right) = \overleftarrow{F}\left(F\left(X\right)\right) = X\ \left[P\right].$
        \item Given a pair of real numbers $x \in \mathbb{R}$ and $y \in \left(0, 1\right).$ The following sequence of implications hold:
            \begin{eqnarray}
                \left(F\left(x\right) \geq y\right) &\Leftrightarrow& \left(\overrightarrow{F}\left(y\right) \leq x\right), \label{(42)}{\iffalse \nonumber \fi}\\
                \left(F\left(x\right) \leq y\right) &\Rightarrow& \left(\overleftarrow{F}\left(y\right) \geq x\right), \label{(1)} \\
                \left(\overleftarrow{F}\left(y\right) > x\right) &\Rightarrow& \left(F\left(x\right)\leq y\right), \label{(2)} \\
                \left(\overleftarrow{F}\left(y\right) = x\right) &\Rightarrow& \left(F\left(x-\right)\leq y\right). \label{(3)}
            \end{eqnarray}
        \item Given $x \in D_F,$ then for each $y \in \left(F\left(x-\right), F\left(x\right)\right] \cap \left(0, 1\right), \overrightarrow{F}\left(y\right) = x,$
              whereas for every $y \in \left[F\left(x-\right), F\left(x\right)\right) \cap \left(0, 1\right), \overleftarrow{F}\left(y\right) = x.$
        \item
              For all $y \in \left(0, 1\right), \epsilon > 0, F\left(\overrightarrow{F}\left(y\right) - \epsilon\right) < y \leq F\left(\overrightarrow{F}\left(y\right)\right), F\left(\overleftarrow{F}\left(y\right) - \epsilon\right) \leq y.$

            In particular, we have $F\left(\overrightarrow{F}\left(y\right)-\right) \leq y \leq F\left(\overrightarrow{F}\left(y\right)\right), F\left(\overleftarrow{F}\left(y\right)-\right)\leq y$ for all $y \in \left(0, 1\right).$
        \item
            Let $X$ denote a real-valued random variable with distribution function $F,$ and another $U$ that is independent of $X.$ Let its Brockwell transform be $Z \coloneqq \left(1 - U\right)F\left(X-\right) + UF\left(X\right).$
            \begin{itemize}
                \item[(i)] If $U$ is uniformly distributed over $\left(0, 1\right),$ then so does $Z.$ 

                        In particular, when $F$ is continuous, or equivalently when $X$ is non-atomic, 
                        the probability integral transform $Z = F\left(X\right)$ is uniformly distributed over $\left(0, 1\right).$
                \item[(ii)] Let $H$ be the distribution function of $U$ such that $H\left(0\right) = 0$ and $H\left(1\right) > 0,$ then the set of discontinuities for the distribution function $F_Z$ of $Z$ satisfies the following equation:
                                \begin{equation}
                                    D_{F_Z} = \cup_{x \in D_F} \left(\mu_F\left(\left(-\infty, x\right)\right) + \mu_F\left(\left\{ x \right\}\right) D_H\right). \label{(47)}
                                \end{equation} 
            \end{itemize}
    \end{enumerate}
\end{lemma}

Before stating the next lemma, we need some more, possibly familiar, definitions to facilitate its discussion, 
among which Definitions $1$ to $6$ are adapted from Section $2.1$ to $2.4$ and Section $2.8$ of \cite{federer2014geometric},
Definition $7$ is from Section $12$ on pp. $176$--$177$ of \cite{billingsley1995probability}
and finally Definition $8$ from Definition $8.34,$ $8.35$ as well as Theorem $8.36$ on pp. $208$--$209$ of \cite{klenke2013probability}.

\begin{enumerate}
    \item
        Given an arbitrary set $X$ and an outer measure $\mu^\star$ on its power set $\mathcal{P}\left(X\right).$ For each $A \in \mathcal{P}\left(X\right),$
        say $A$ is $\mu^\star$-\textit{measurable} if and only if $A$ satisfies the $\mu^\star$-Carath\'eodory's condition, in other words, for every subset $T \subseteq X,$ we have

        $$\mu^\star\left(T\right) \geq \mu^\star\left(T \cap A\right) + \mu^\star\left(T - A\right).$$

        Denote the class of all $\mu^\star$-measurable subsets by $\mathcal{L}_{\mu^\star},$ which is a $\sigma$-algebra on which 
        $\mu^\star|_{\mathcal{L}_{\mu^\star}}$ is a measure.

    \item
        If the set $X$ is further endowed with a topology $\tau \subseteq \mathcal{P}\left(X\right),$ then let 
        $\sigma\left(\tau\right) \subseteq \mathcal{P}\left(X\right)$ be the $\sigma$-algebra generated by the topology $\tau,$
        whose members we call \textit{Borel sets associated with $\tau.$}

    \item
        An outer measure $\mu^\star$ on a topological space $\left(X, \tau\right)$ is \textit{Borel regular} if and only if
        the following two conditions hold:
        \begin{enumerate}
            \item[(i)] $\tau$ is contained in the class of $\mu^\star$-measurable subsets $\mathcal{L}_{\mu^\star}.$
            \item[(ii)] For any subset $A \in \mathcal{P}\left(X\right),$ there exists a Borel set $B \in \sigma\left(\tau\right)$ such that $B$ contains $A,$ meanwhile $\mu^\star\left(B\right) = \mu^\star\left(A\right).$
        \end{enumerate}

        Observe that given any outer measure $\mu^\star$ on a topological space $\left(X, \tau\right)$ that satisfies (i) above, then for each subset $A \subseteq X,$
        let $$\nu^\star\left(A\right) \coloneqq \inf \left\{ \mu^\star\left(B\right) : A \subseteq B, B \in \sigma\left(\tau\right) \right\},$$
        then $\nu^\star$ is a Borel regular outer measure, and for each Borel set $B \in \sigma\left(\tau\right),$ we have $\nu^\star\left(B\right) = \mu^\star\left(B\right).$
        Henceforth, in the presence of (i), an equivalent condition for $\mu^\star$ to satisfy (ii) is $\nu^\star = \mu^\star.$ 

    \item
        Let $X$ be a set on which an outer-measure $\mu^\star$ is defined, $Z$ be a subset of $X$ and $\left(Y, \tau_Y\right)$ be a topological space 
        as well as $f : Z \rightarrow Y$ be a map from $Z$ into $Y.$ We say that $f$ is $\mu^\star$-\textit{measurable} if and only if
        $\mu^\star\left(X - Z\right) = 0$ and also for any open set $G \in \tau_Y, f^{-1}\left(G\right)$ is a $\mu^\star$-measurable subset of $X.$ 
        In addition, if $X$ is a topological space on which a map $f : \left(X, \tau_X\right) \rightarrow \left(Y, \tau_Y\right)$ from $X$ into $Y$ is specified, then $f$ is \textit{Borel measurable}
        if, for any open subset $G$ of $Y, f^{-1}\left(G\right)$ is a Borel subset of $X.$

        The way to define the $\mu^\star$-\textit{integral}, $\int f d\mu^\star,$ of a $\mu^\star$-measurable (extended) real-valued function $f : X \rightarrow \overline{\mathbb{R}}$ is same as that of $\int f d\mu^\star|_{\mathcal{L}_{\mu^\star}}$--the integral of $f$ with respect to the measure $\mu^\star|_{\mathcal{L}_{\mu^\star}}.$
    
    \item
        Given a metric space $\left(X, \rho\right),$ the open and the closed ball at center $x \in X$ with radius $r > 0$ are denoted by 
        \begin{eqnarray}
            B_\rho\left(x, r\right) &\coloneqq& \left\{ y \in X : \rho\left(y, x\right) < r \right\} \nonumber
        \end{eqnarray} and
        \begin{eqnarray}
            B_\rho\left[x, r\right] &\coloneqq& \left\{ y \in X : \rho\left(y, x\right) \leq r \right\}, \nonumber
        \end{eqnarray} respectively. Additionally, let $\tau_\rho$ denote the topology generated by the metric $\rho,$ 
        or more specifically, by the topological base $\left\{ B_\rho\left(x, r\right) : \left(x, r\right) \in X \times \left(0, +\infty\right) \right\}.$

        The following series of definitions are applied under an outer measure $\mu^\star$ on $\left(X, \rho\right)$ such that the class of $\mu^\star$-measurable subsets $\mathcal{L}_{\mu^\star}$ contains the metric topology $\tau_\rho,$ 
        and that any subset $E$ of finite $\rho$-diameter must have finite $\mu^\star$ outer measure. 

        Given a set $A \in \mathcal{P}\left(X\right),$ say a countable family $\mathcal{A}$ of closed sets of $X$ \textit{covers $\mu^\star$ almost all of $A$} if $\mu^\star\left(A - \cup \mathcal{A}\right) = 0.$

        By a \textit{covering relation}, we mean a subset of the pairs $\left\{ \left(x, A\right) : x \in A \right\} \subseteq X \times \mathcal{P}\left(X\right).$ 
        Whenever $\mathfrak{C}$ is a covering relation and $B$ is a subset of $X,$ we let
        $$\mathfrak{C}\left(B\right) \coloneqq \pi\left(\mathfrak{C} \cap \left[B \times \mathcal{P}\left(X\right)\right]\right) = \left\{A : \mathrm{There\ exists}\ x \in B\ \mathrm{such\ that}\ \left(x, A\right) \in \mathfrak{C} \right\},$$
        where $\pi$ is the projection onto the second coordinate
        \begin{align}
            \pi : X \times \mathcal{P}\left(X\right) &\rightarrow \mathcal{P}\left(X\right) \nonumber \\
                                   \left(x, A\right) &\mapsto A. \nonumber \end{align}

        Say $\mathfrak{C}$ is \textit{fine at} $x \in X$ if and only if $\inf \left\{ \mathrm{diam}_\rho\left(A\right) : \left(x, A\right) \in \mathfrak{C} \right\} = 0.$ 
        Whenever $\mathfrak{C}$ is fine at some $x \in X,$ we shall use the following notations to designate        
        $$\left(\mathfrak{C}\right) \overline{\lim}_{x} f = \left(\mathfrak{C}\right) \overline{\lim}_{S \rightarrow x} f\left(S\right) \coloneqq \lim_{\epsilon \rightarrow 0+} \sup \left\{ f\left(A\right) : \left(x, A\right) \in \mathfrak{C}, \mathrm{diam}_\rho\left(A\right) < \epsilon \right\}$$
        for each extended real-valued function $f : X \rightarrow \mathbb{\overline{R}}.$ 
        Similar definitions carry over to $\left(\mathfrak{C}\right) \underline{\lim}$ as well as $\left(\mathfrak{C}\right) \lim.$

        A $\mu^\star$-\textit{Vitali relation} is a covering relation $\mathfrak{V}$ such that 
        \begin{itemize}
            \item[(i)] $\mathfrak{V}\left(X\right)$ is contained in the class of Borel sets $\sigma\left(\tau_\rho\right),$
            \item[(ii)] for every $x \in X, \mathfrak{V}$ is fine at $x,$ and
            \item[(iii)] as long as $\mathfrak{C}$ is a covering relation contained in $\mathfrak{V}$ for which it is fine at $x$ for all $x$ belonging to a subset $B \subseteq X,$ 
                         there exists a countable, pairwise disjoint family $\mathcal{A} \subseteq \mathfrak{C}\left(B\right)$ covering $\mu^\star$ almost all of $B.$ 
        \end{itemize}

        \item
            We need another technical geometric condition on a metric space: Given a metric space $\left(X, \rho\right)$ and a subset $A \in \mathcal{P}\left(X\right),$ say $\rho$ is \textit{directionally $\left(\xi, \eta, \zeta\right)$-limited at $A$} for some 
            $\xi > 0, \eta \in \left(0, \frac{1}{3}\right], \zeta \in \mathbb{N}^\star$ if, for every
            $a \in A, B \subseteq A \cap \left[B_\rho\left(a, \xi\right) - \left\{ a \right\}\right]$ and 
            every $x \in X$ such that $\rho\left(x, c\right) \geq \eta \rho\left(a, c\right),$ we have the following implication: That $b, c \in B, b \neq c, \rho\left(a, b\right) \geq \rho\left(a, c\right)$ with 
            $\rho\left(a, x\right) = \rho\left(a, c\right), \rho\left(x, b\right) = \rho\left(a, b\right) - \rho\left(a, c\right)$ shall imply $\# B \leq \zeta.$

            An important special case where a metric space enjoys the directional limited property is that $X$ is a finite-dimensional normed linear space over $F$ ($F$ being either the field of real numbers $\mathbb{R}$ or that of complex numbers $\mathbb{C}$) equipped with the norm $\left\lVert \cdot \right\rVert,$ 
            and then the metric $\rho$ induced by $\left\lVert \cdot \right\rVert$ is $\left(+\infty, \eta, \zeta\right)$-directional limited at $X$ for each $\eta > 0$ with a corresponding $\zeta\left(\eta\right) \in \mathbb{N}^\star.$ 

        \item
            For each $n \in \mathbb{N}^\star,$ equip the Euclidean space $\mathbb{R}^n$ with the standard metric topology $\tau$ 
            induced by any $p$-norm $\left\lVert \cdot \right\rVert_p$ for $p \in \left[1, +\infty\right].$

            Let the class of bounded rectangles in $\mathbb{R}^n$ be 
            $$\mathcal{R}_n \coloneqq \left\{ \times_{i = 1}^n \left(a_i, b_i\right] : -\infty < a_i \leq b_i < +\infty\ \mathrm{for\ all}\ i = 1, 2, \cdots, n \right\}.$$

            Given a real-valued function $F : \mathbb{R}^n \rightarrow \mathbb{R}.$ For each bounded rectangle $R \in \mathcal{R}_n,$ 
            define the difference of $F$ around the vertices of $R$ by 
            $$\Delta_R F \coloneqq \sum_{\mathbf{x} \in \times_{1 \leq i \leq k} \left\{ a_i, b_i \right\}} \left(-1\right)^{\# \left\{i : x_i = a_i\right\}} F\left(\mathbf{x}\right).$$

            $F$ is \textit{continuous from above}, if, for any sequence $\left\{ \mathbf{x}^{\left(k\right)} \right\}_{k = 1}^\infty \subseteq \mathbb{R}^n$ and any point $\mathbf{x} \in \mathbb{R}^n$ such that 
            $x^{\left(k\right)}_i \downarrow x_i$ for all $i = 1, 2, \cdots, n,$ then we have $F\left(\mathbf{x}^{\left(k\right)}\right) \rightarrow F\left(\mathbf{x}\right)$ as $k \rightarrow \infty.$
        
        \item
            Given two measurable spaces $\left(X, \mathcal{F}\right)$ and $\left(Y, \mathcal{G}\right).$ $X$ and $Y$ are called \textit{isomorphic} if there exists a bijection $\varphi : X \rightarrow Y$ such that
            for all $B \in \mathcal{G},$ we have $\varphi^{-1}\left(B\right) \in \mathcal{F},$ and conversely, that for all $A \in \mathcal{F},$ we also have $\varphi\left(A\right) \in \mathcal{G}.$

            A measurable space $\left(X, \mathcal{F}\right)$ is called a \textit{Borel space} if there exists a Borel subset of the real line $B \in \sigma\left(\tau\right)$ such that 
            $\left(X, \mathcal{F}\right)$ is isomorphic to $\left(B, \sigma\left(\tau\right) \cap B\right),$ where $\tau$ is the standard ruler metric topology on $\mathbb{R}$ 
            and hence $\sigma\left(\tau\right) \cap B$ is the sub Borel $\sigma$-algebra on $B.$

            Note that an important example of a Borel space is given by any metric space $\left(X, \rho\right)$ for which $\rho$ is complete and $X$ is separable under $\rho,$ in other words, $\left(X, \rho\right)$ is Polish.
            The corresponding measurable space $\left(X, \sigma\left(\tau_\rho\right)\right)$ is a Borel space.
\end{enumerate}

Now we are ready for the upcoming Lemma~\ref{lma::3}.

\begin{lemma}
    \label{lma::3}
    \begin{enumerate}
        \item 
              Let $X$ be a finite-dimensional normed linear space over $F$ ($F \in \left\{ \mathbb{R}, \mathbb{C} \right\}$) equipped with the norm $\left\lVert \cdot \right\rVert$ and let $\rho$ be the metric induced by $\left\lVert \cdot \right\rVert.$ 
              Given an outer measure $\mu^\star$ on $\left(X, \rho\right)$ such that the class of $\mu^\star$-measurable subsets $\mathcal{L}_{\mu^\star}$ contains the metric topology $\tau_\rho,$ and that
              the outer measure $\mu^\star$ is finite on any subset $E$ of finite $\rho$-diameter.

              Then $\mathfrak{V} = \left\{ \left(x, B_\rho\left[x, r\right]\right) : \left(x, r\right) \in X \times \left(0, +\infty\right) \right\}$ is a $\mu^\star$-Vitali relation.
        \item
              For a real-valued function $F : \mathbb{R}^n \rightarrow \mathbb{R},$ if 
            \begin{itemize}
                \item[(i)] $F$ is continuous from above, and
                \item[(ii)] $\Delta_R F$ is non-negative for all bounded rectangles $R \in \mathcal{R}_n,$
            \end{itemize} then the induced outer measure 
            \begin{align} 
                \mu_F^\star : \mathcal{P}\left(\mathbb{R}^n\right) &\rightarrow \left[0, +\infty\right] \nonumber \\
                                                                 E &\mapsto \inf \left\{ \sum_{n = 1}^\infty \mu_F\left(E_n\right) : \left\{ E_n \right\}_{n = 1}^\infty \subseteq \mathcal{R}_n, E \subseteq \cup_{n = 1}^\infty E_n \right\} \nonumber 
            \end{align} is Borel regular on $\left(\mathbb{R}^n, \tau\right),$ where 
            \begin{align} 
                \mu_F : \mathcal{R}_n &\rightarrow \left[0, +\infty\right) \nonumber \\
                                    R &\mapsto \Delta_R F. \nonumber 
            \end{align}

            In particular, for any finite measure $\nu : \sigma\left(\tau\right) \rightarrow \left[0, +\infty\right],$ 
            there exists an $F_\nu : \mathbb{R}^n \rightarrow \mathbb{R}$ that satisfies (i) and (ii)
            such that the Borel regular outer measure $\mu^\star_{F_\nu} : \mathcal{P}\left(\mathbb{R}^n\right) \rightarrow \left[0, +\infty\right]$ extends $\nu$: $\mu^\star_{F_\nu}|_{\sigma\left(\tau\right)} = \nu.$
        \item 
            Given any $F : \mathbb{R}^n \rightarrow \mathbb{R}$ that satisfies Conditions (i) and (ii) stated in Definition $2$ and also a $\mu^\star_F$-measurable $f : \mathbb{R}^n \rightarrow \mathbb{\overline{R}}$ 
            such that for each bounded $\mu^\star_F$-measurable subset $E \subseteq \mathbb{R}^n,$ we have $\int_E \left|f\right| d\mu^\star_F < +\infty,$ then 
            $$\mu^\star_F \left(\left\{ \mathbf{x} \in \mathbb{R}^n : \left(\mathfrak{V}_p\right) \lim_{B \rightarrow \mathbf{x}} \frac{\int_B f d\mu^\star_F}{\mu^\star_F\left(B\right)} = f\left(\mathbf{x}\right) \right\}^\mathsf{c}_{\mathbb{R}^n}\right) = 0,$$
            where $\mathfrak{V}_p = \left\{ \left(x, B_{\rho_p}\left[x, r\right]\right) : \left(x, r\right) \in \mathbb{R}^n \times \left(0, +\infty\right) \right\}$ with $\rho_p$ being the metric induced by $p$-norm $\left\lVert \cdot \right\rVert_p$ on $\mathbb{R}^n$ for any $p \in \left[1, +\infty\right].$
    \end{enumerate}
\end{lemma}

\section{Main Result and its Corollaries}

We say that a distribution function $F : \mathbb{R} \rightarrow \left[0, 1\right],$ or a real-valued random variable $X$ with $F$ as its distribution function, satisfies the $\leq$ well-ordering condition if
the set of (jump) discontinuities for $F,$ i.e., $D_F,$ is well-ordered by $\leq$ on the real line. Similarly, 
we also define distribution functions/random variables that satisfy the $\geq$ well-ordering condition.

\begin{theorem}
    \label{thm::4} 
    Denote $X$ by a real-valued random variable with distribution function $F,$ and $U$ by another random variable that is independent of $X$ and has its distribution function $H : \mathbb{R} \rightarrow \left[0, 1\right]$ satisfying the conditions $H\left(0\right) = 0, H\left(1\right) > 0.$ 
    Let $\left(M, \mathcal{F}\right)$ be a Borel space and $\widetilde{Y} : \Omega \rightarrow M$ be a random map which is conditionally independent of $U$ given $X.$ 
    Additionally, define the Brockwell transform of $X$ to be $Z = \left(1 - U\right)F\left(X-\right) + UF\left(X\right).$

    If $F$ satisfies the $\leq$ well-ordering condition, then 
    $$Z\ \mathrm{is\ independent\ of}\ \widetilde{Y} \Rightarrow X\ \mathrm{is\ independent\ of}\ \widetilde{Y}.$$
\end{theorem}

\begin{remark}
    \label{rmk::4} 
    Recall Remark~\ref{rmk::1} in Lemma~\ref{lma::1} and the fact that the negation 
    \begin{align}
        - : \mathbb{R} &\rightarrow \mathbb{R} \nonumber \\
                     x &\mapsto -x \nonumber
    \end{align} is a bijection which preserves the independence property. If a random variable $X$ satisfies the $\geq$ well-ordering condition, 
    then we may apply Theorem~\ref{thm::4} to $-X$ whose set of atoms satisfies $D_{F_{-X}} = -D_{F_X},$ which is well-ordered by $\leq$ on $\mathbb{R}.$
\end{remark}

\begin{proof}
    Denote $D_F$ by the set of atoms for $X.$ 
    Due to Lemma~\ref{lma::1}, if we let $m \coloneqq \# D_F$ for which $D_F$ is at most countable due to finiteness of $PX^{-1},$ in each of the following scenarios:
    \begin{itemize}
        \item[\textbf{Case (i):}] $D_F$ \textbf{is countably infinite.}

            Arrange $D_F = \left\{ x_n \right\}_{n = 1}^\infty$ 
            such that 
            \begin{equation}
                -\infty = x_0 < x_1 < x_2 < \cdots < x_n < x_{n + 1} < \cdots < x_{\infty} = +\infty, \label{(10)}
            \end{equation} which implies $\forall n \in \mathbb{N} \Rightarrow F\left(x_n\right) \leq F\left(x_{n + 1}-\right) < F\left(x_{n + 1}\right)$ with $F\left(x_0\right) = F\left(x_0+\right) = 0, F\left(x_\infty\right) = F\left(x_\infty-\right) = 1.$
        \item[\textbf{Case (ii):}] $D_F$ \textbf{is finite.}

            In this case, also permute $D_F = \left\{ x_n \right\}_{n = 1}^m$ satisfying
            \begin{equation}
                -\infty = x_0 < x_1 < x_2 < \cdots < x_m < x_{m + 1} = +\infty, \label{(11)}
            \end{equation} where $\forall n \in \left\{ 0, 1, \cdots, m \right\} \Rightarrow F\left(x_n\right) \leq F\left(x_{n + 1}-\right) < F\left(x_{n + 1}\right)$ with $F\left(x_0\right) = F\left(x_0+\right) = 0, F\left(x_{m + 1}\right) = F\left(x_{m + 1}-\right) = 1.$
    \end{itemize}

    Up follows let $\Lambda_F \coloneqq \left\{\begin{array}{cc}
                                                   \mathbb{N}, & \mathrm{if}\ m = \aleph_0, \nonumber \\
                                                   \left\{ 0, 1, \cdots, m \right\}, & \mathrm{if}\ m < \aleph_0, \nonumber
                                               \end{array}\right.$ (as usual $\Lambda_F^\star = \Lambda_F - \left\{ 0 \right\}$) 
    \quad $\mu_F \coloneqq PX^{-1}$ be the law of $X$ on $\sigma\left(\tau_\rho\right)$ with $\rho$ being the standard ruler metric on the real line $\mathbb{R},$ 
    whereas $\nu_{\widetilde{Y}} \coloneqq P\widetilde{Y}^{-1}$ be that of $\widetilde{Y}$ on $\mathcal{F}.$ 

    With $Z \independent \widetilde{Y}$ we have for each $\left(z, A\right) \in \mathbb{R} \times \mathcal{F},$
    \begin{eqnarray}
        P\left(\left\{ Z \leq z \right\} \cap \left\{ \widetilde{Y} \in A \right\}\right) &=& P\left(\left\{ Z \leq z \right\}\right)P\left(\left\{ \widetilde{Y} \in A \right\}\right) = P\left(\left\{ Z \leq z \right\}\right)\nu_{\widetilde{Y}}\left(A\right). \nonumber\\ \label{(12)}
    \end{eqnarray}

    To simplify the left-hand side of the above, given $\left(z, A\right) \in \left(0, 1\right) \times \mathcal{F},$
    \begin{eqnarray}
        P\left(\left\{ Z \leq z\right\} \cap \left\{ \widetilde{Y} \in A \right\}\right) &=& E\left(P\left(\left\{ Z \leq z \right\} \cap \left\{ \widetilde{Y} \in A \right\}\ |\ X\right)\right) \nonumber \\
                                                                                         &=& E\left(P\left(\left\{ \mu_F\left(\left(-\infty, X\right)\right) + \mu_F\left(\left\{ X \right\}\right)U \leq z \right\} \cap \left\{ \widetilde{Y} \in A \right\}\ |\ X\right)\right) \nonumber \\
                                                                                         &=& E\left(P\left(\left\{ \mu_F\left(\left(-\infty, X\right)\right) + \mu_F\left(\left\{ X \right\}\right)U \leq z \right\} \cap \left\{ \widetilde{Y} \in A \right\} \cap \left\{ X \in D_F \right\}\ |\ X\right)\right) \nonumber \\
                                                                                         &+& E\left(P\left(\left\{ \mu_F\left(\left(-\infty, X\right]\right) \leq z \right\} \cap \left\{ \widetilde{Y} \in A \right\} \cap \left\{ X \not\in D_F \right\}\ |\ X\right)\right) \nonumber
    \end{eqnarray}
    \begin{eqnarray}
        \left((\ref{(1)})\ \mathrm{to}\ (\ref{(3)})\ \mathrm{of}\ 6\ \mathrm{in\ Lemma}~\ref{lma::2}\right)\Rightarrow &=& E\left(P\left(\left\{ U \leq \frac{z - \mu_F\left(\left(-\infty, X\right)\right)}{\mu_F\left(\left\{ X \right\}\right)}\right\} \cap \left\{ \widetilde{Y} \in A \right\} \cap \left\{ X \in D_F \right\}\ |\ X\right)\right) \nonumber \\
                                                                                         &+& E\left(P\left(\left\{ X \leq \overleftarrow{F}\left(z\right) \right\} \cap \left\{ \widetilde{Y} \in A \right\} \cap \left\{ X \not\in D_F \right\}\ |\ X\right)\right) \nonumber \\
        \left(\left(U, \widetilde{Y}\right) \big| X \sim U \times \widetilde{Y} | X\right) \Rightarrow &=& E\left(H\left(\frac{z - \mu_F\left(\left(-\infty, X\right)\right)}{\mu_F\left(\left\{ X \right\}\right)}\right)P\left(\widetilde{Y} \in A\ |\ X\right)\mathbb{I}_{\left\{ X \in D_F \right\}}\right) \nonumber \\
                                                                                         &+& E\left(P\left(\widetilde{Y} \in A\ |\ X\right)\mathbb{I}_{\left\{ X \in \left(-\infty, \overleftarrow{F}\left(z\right)\right] - D_F \right\}}\right) \nonumber \\
                                                                                         &=& \int_{\mathbb{R}} \left(I_1 + I_2\right)\left(x, z\right)P\left(\widetilde{Y} \in A\ |\ X = x\right) \mu_F\left(dx\right), \label{(13)} 
    \end{eqnarray} for which 
    \begin{eqnarray}
        I_1\left(x, z\right) &\coloneqq& H\left(\frac{z - \mu_F\left(\left(-\infty, x\right)\right)}{\mu_F\left(\left\{ x \right\}\right)}\right)\mathbb{I}_{D_F}\left(x\right), \label{(39)} \\
        I_2\left(x, z\right) &\coloneqq& \mathbb{I}_{\left(-\infty, \overleftarrow{F}\left(z\right)\right] - D_F}\left(x\right). \nonumber
    \end{eqnarray}

    Take $A = M$ in~(\ref{(13)}) above so that $\left\{ \widetilde{Y} \in M \right\} = \Omega,$ we actually know $$P\left(\left\{ Z \leq z \right\}\right) = \int_{\mathbb{R}} \left(I_1 + I_2\right)\left(x, z\right) \mu_F\left(dx\right).$$

    By substituting the above identity with~(\ref{(13)}) into~(\ref{(12)}) we obtain for each $z \in \left(0, 1\right),$ 
    \begin{equation}
        \int_{\mathbb{R}} \left(I_1 + I_2\right)\left(x, z\right)\phi_A\left(x\right) \mu_F\left(dx\right) = 0, \label{(14)}
    \end{equation}
    where \begin{align}
              \phi_A : \mathbb{R} &\rightarrow \mathbb{R} \nonumber \\
                                x &\mapsto P\left(\widetilde{Y} \in A\ \bigg|\ X = x\right) - \nu_{\widetilde{Y}}\left(A\right). \nonumber 
          \end{align}

    For induction index $n = 0,$ if $\mu_F\left(\left(-\infty, x_1\right)\right) = F\left(x_1-\right) > 0, \forall z \in \left(0, F\left(x_1-\right)\right]$ through the well-ordering condition by $\leq$ of $D_F$ in~(\ref{(10)}) and~(\ref{(11)})
    we actually have 
    $$\forall x \in \mathbb{R} \Rightarrow I_1\left(x, z\right) = \sum_{n \in \Lambda_F^\star} H\left(\frac{z - \mu_F\left(\left(-\infty, x_n\right)\right)}{\mu_F\left(\left\{ x_n \right\}\right)}\right)\mathbb{I}_{\left\{ x_n \right\}}\left(x\right) = 0.$$

    Then identity~(\ref{(14)}) when $\forall w \in \left(-\infty, x_1\right) \ni z = F\left(w\right) \in \left(0, F\left(x_1-\right)\right) \subseteq \left(0, 1\right)$ (since $F\left(x_1-\right) \leq F\left(x_1\right) \leq 1$ and also such $w$ must exist due to the definition for left limit of $F$ at $x_1$) reduces to 
    \begin{equation}
        \int_{\mathbb{R}} I_2\left(x, F\left(w\right)\right)\phi_A\left(x\right) \mu_F\left(dx\right) = 0. \nonumber
    \end{equation}

    By the first item of $5$ in Lemma~\ref{lma::2}, as $\left(-\infty, x_1\right) \subseteq D_F^\mathsf{c},$ 
    $$\mu_F\left(\left\{ w \in G : \left(\overleftarrow{F} \circ F|_G\right)\left(w\right) > w \right\} \cap \left(-\infty, x_1\right)\right) = 0,$$ 
    for which $\forall w \in \left\{ w \in G : \left(\overleftarrow{F} \circ F|_G\right)\left(w\right) = w \right\} \cap \left(-\infty, x_1\right)$ we have 
    \begin{equation} 
        \int_{\mathbb{R}} \mathbb{I}_{\left(-\infty, w\right]}\left(x\right)\phi_A\left(x\right) \mu_F\left(dx\right) = 0. \label{(15)}
    \end{equation}

    For each of the above fixed $w$ and $B_\rho\left[v, r\right] = \left[v - r, v + r\right]$ such that $w \in B_\rho\left[v, r\right]$ and $\mu_F\left(B_\rho\left[v, r\right]\right) > 0,$ 
    we must have, since $\mu_F\left(\left\{ w \right\}\right) = 0$ due to $\left(-\infty, x_1\right) \subseteq D_F^\mathsf{c}$ so that $\mu_F\left(B_\rho\left[v, r\right] - \left\{ w \right\}\right) > 0:$
    \begin{eqnarray}
        && \mathrm{If}\ \mu_F\left(B_\rho\left[v, r\right] \cap \left(-\infty, w\right)\right) > 0, \nonumber \\
        &&\left(B_\rho\left[v, r\right] \cap \left(-\infty, w\right)\right) \cap \left(\left\{ w \in G : \left(\overleftarrow{F} \circ F|_G\right)\left(w\right) = w \right\} \cap \left(-\infty, x_1\right)\right) \neq \varnothing. \nonumber \\
        && \mathrm{If}\ \mu_F\left(B_\rho\left[v, r\right] \cap \left(w, +\infty\right)\right) > 0, \nonumber \\
        && \left(B_\rho\left[v, r\right] \cap \left(w, +\infty\right)\right) \cap \left(\left\{ w \in G : \left(\overleftarrow{F} \circ F|_G\right)\left(w\right) = w \right\} \cap \left(-\infty, x_1\right)\right) \neq \varnothing. \nonumber \\ \label{(16)}
    \end{eqnarray}

    Otherwise, there would exist a $B_\rho\left[v_0, r_0\right] \ni w$ with $\mu_F\left(B_\rho\left[v_0, r_0\right]\right) > 0$ such that, for instance, if $\mu_F\left(B_\rho\left[v_0, r_0\right] \cap \left(-\infty, w\right)\right) > 0,$ 
    \begin{eqnarray}
        B_\rho\left[v_0, r_0\right] \cap \left(-\infty, w\right) &\subseteq& \left\{ w \in G : \left(\overleftarrow{F} \circ F|_G\right)\left(w\right) > w \right\} \cap \left(-\infty, x_1\right) \nonumber \\
        \Rightarrow 0 < \mu_F\left(B_\rho\left[v_0, r_0\right] \cap \left(-\infty, w\right)\right) &\leq& \mu_F\left(\left\{ w \in G : \left(\overleftarrow{F} \circ F|_G\right)\left(w\right) > w \right\} \cap \left(-\infty, x_1\right)\right) = 0, \nonumber
    \end{eqnarray} due to the monotonicity of $\mu_F$ which is an absurdity.

    Note that
    \begin{enumerate}
        \item[(i)]
            $F : \mathbb{R} \rightarrow \left[0, 1\right]$ is continuous from above, due to its right continuity, and also $\forall \left(a, b\right] \in \mathcal{R}_1 \Rightarrow \Delta_{\left(a, b\right]} F = F\left(b\right) - F\left(a\right) \geq 0$ by monotonicity of $F,$
        \item[(ii)]
            By the existence of regular conditional distribution of $\widetilde{Y}$ given $\sigma\left(\left\{ X \right\}\right)$ as a stochastic kernel 
            (for example Definition $8.25, 8.28$ along with Theorem $8.37$ on pp. $204$--$205$ of \cite{klenke2013probability}), 
            for each $A \in \mathcal{F},$ 
            $x \mapsto P\left(\widetilde{Y} \in A\ |\ X = x\right)$ is $\sigma\left(\left\{ X \right\}\right)$-measurable, hence further $\mu_F^\star$-measurable as 
            $\sigma\left(\left\{ X \right\}\right) \subseteq \sigma\left(\tau_\rho\right) \subseteq \mathcal{L}_{\mu_F^\star},$ then so is $\phi_A.$
        \item[(iii)]
            Thirdly, 
            \begin{eqnarray}
                \int_{\mathbb{R}} \left|\phi_A\left(x\right)\right| \mu_F^\star\left(dx\right) &=& \int_{\mathbb{R}} \left|\phi_A\left(x\right)\right| \mu_F\left(dx\right) \nonumber \\
                                                                                               &\leq& 2\mu_F\left(\mathbb{R}\right) = 2 < +\infty, \nonumber
            \end{eqnarray} where the first equality is for $\mu_F^\star|_{\sigma\left(\tau_\rho\right)} = \mu_F$ 
            from the ``In particular'' part of $2$ in Lemma~\ref{lma::3} whereas the second $\forall x \in \mathbb{R}, 0 \leq P\left(\widetilde{Y} \in A\ |\ X = x\right) \leq 1$ along with $0 \leq \nu_{\widetilde{Y}}\left(A\right) \leq 1.$
    \end{enumerate}

    Then we appeal to the third item of Lemma~\ref{lma::3} to conclude that $\mu_F^\star\left(\mathbb{R} - E\right) = 0,$ where
    \begin{eqnarray}
        E &\coloneqq& \left\{ w \in \mathbb{R} : \left(\mathfrak{V}\right) \lim_{B_\rho\left[v, r\right] \rightarrow w} \frac{\int_{B_\rho\left[v, r\right]} \phi_A d\mu_F}{\mu_F\left(B_\rho\left[v, r\right]\right)} = \phi_A\left(w\right) \right\}, \nonumber \\
        \mathfrak{V} &\coloneqq& \left\{ \left(v, B_\rho\left[v, r\right]\right) : \left(v, r\right) \in \mathbb{R} \times \left(0, +\infty\right) \right\}, \nonumber
    \end{eqnarray} and $\mathfrak{V}\left(\mathbb{R}^n\right) \subseteq \sigma\left(\tau_\rho\right)$ along with the $\sigma\left(\tau_\rho\right)$-measurability of each $\phi_A$ allows us to interchange 
    $\mu_F$ and $\mu^\star_F$ in the numerator and denominator of $\mathfrak{V}$-limit.

    Additionally, notice that 
    \begin{equation}
        E \subseteq \left\{ w \in \mathbb{R} : \mathfrak{V} \cap \left\{ \left(w, B_\rho\left[v, r\right]\right) : \mu_F\left(B_\rho\left[v, r\right]\right) = 0 \right\} \mathrm{is\ not\ fine\ at}\ w \right\}. \label{(17)}
    \end{equation}

    Now for each $w \in \left\{ w \in G : \left(\overleftarrow{F} \circ F|_G\right)\left(w\right) = w \right\} \cap \left(-\infty, x_1\right) \cap E,$ first of all,

    \begin{eqnarray}
        s = 2\inf \left\{ r : \left(w, B_\rho\left[v. r\right]\right) \in \mathfrak{V}, \mu_F\left(B_\rho\left[v, r\right]\right) = 0 \right\} > 0, \nonumber
    \end{eqnarray} due to inequality~(\ref{(17)}) and then for $n = 1, \epsilon_1 \coloneqq \min\left\{ 1, s \right\} > 0,$ as 
    $$\mathfrak{V} \cap \left\{ \left(v, B_\rho\left[v, r\right]\right) : \mu_F\left(B_\rho\left[v, r\right]\right) = 0 \right\} \mathrm{is\ not\ fine\ at}\ w$$ but $\mathfrak{V}$ is fine at $w$ we know there exists 
    $B_\rho\left[v_1, r_1\right] \ni w$ such that $r_1 < \frac{\epsilon_1}{2}$ and also $\mu_F\left(B_\rho\left[v_1, r_1\right]\right) > 0.$ 
    By~(\ref{(16)}), we further know
    \begin{itemize}
        \item[(i)] when $\mu_F\left(B_\rho\left[v_1, r_1\right] \cap \left(-\infty, w\right)\right) > 0,$ 

            we may pick a
            $w_{1, 1} \in \left(B_\rho\left[v_1, r_1\right] \cap \left(-\infty, w\right)\right) \cap \left\{ w \in G : \left(\overleftarrow{F} \circ F|_G\right)\left(w\right) = w \right\} \cap \left(-\infty, x_1\right),$ and 
        \item[(ii)] when $\mu_F\left(B_\rho\left[v_1, r_1\right] \cap \left(w, +\infty\right)\right) > 0,$

            we may select another
            $w_{1, 2} \in \left(B_\rho\left[v_1, r_1\right] \cap \left(w, +\infty\right)\right) \cap \left\{ w \in G : \left(\overleftarrow{F} \circ F|_G\right)\left(w\right) = w \right\} \cap \left(-\infty, x_1\right).$ 
    \end{itemize} If either one of them vanishes then we set the corresponding $w_{1, k} \coloneqq w$ so at least one of $i \in \left\{ 1, 2 \right\}$ satisfies $w_{1, i} \neq w$ and hence $\delta w_1 \coloneqq w_{1, 2} - w_{1, 1} \in \left(0, 2r_1\right].$ 
    We also denote $\overline{w}_1 \coloneqq \frac{w_{1, 1} + w_{1, 2}}{2},$
    $$h_1 \coloneqq \left\{\begin{array}{ll}
                               \frac{1}{2}\min \left\{ w - w_{1, 1}, w_{1, 2} - w \right\}, & \left(w_{1, 1} \neq w\right) \mathrm{and} \left(w_{1, 2} \neq w\right), \\
                               \frac{\delta w_1}{2}, & \left(w_{1, 1} \neq w\right) \mathrm{xor} \left(w_{1, 2} \neq w\right),
                           \end{array}\right.$$ where xor means exclusive or logical binary operation on two propositions.

    For $n = 2, \epsilon_2 \coloneqq \min \left\{ \frac{1}{2}, s, h_1 \right\} > 0,$ same rationale yields the existence of 
    $B_\rho\left[v_2, r_2\right] \ni w$ such that $r_2 < \frac{\epsilon_2}{2}$ and also $\mu_F\left(B_\rho\left[v_2, r_2\right]\right) > 0$ 
    but due to the choice of $\epsilon_2$ and $\left\{ w_{1, i} \right\}_{i = 1}^2$ we know $B_\rho\left[v_2, r_2\right] \subseteq \left[w_{1, 1}, w_{1, 2}\right] = B_\rho\left[\overline{w}_1, \frac{\delta w_1}{2}\right]$ 
    so that $\mu_F\left(B_\rho\left[\overline{w}_1, \frac{\delta w_1}{2}\right]\right) \geq \mu_F\left(B_\rho\left[v_2, r_2\right]\right) > 0.$

    Suppose that $\forall i \leq n, n \in \mathbb{N}^\star$ we have constructed two sequence of $\rho$-balls $\left\{ B_\rho\left[\overline{w}_i, \frac{\delta w_i}{2}\right] \right\}_{i = 1}^n \cup \left\{ B_\rho\left[v_i, r_i\right] \right\}_{i = 1}^{n + 1}$ 
    with $\forall i = 1, 2, \cdots, n + 1, r_i < \frac{\epsilon_i}{2}, \mu_F\left(B_\rho\left[v_i, r_i\right]\right) > 0$ where $\epsilon_i \coloneqq \min \left\{ \frac{1}{i}, s, h_{i - 1} \right\}$ 
    such that $\forall i = 1, 2, \cdots, n, B_\rho\left[v_{i + 1}, r_{i + 1}\right] \subseteq B_\rho\left[\overline{w}_i, \frac{\delta w_i}{2}\right]$ along with a sequence of 
    \begin{itemize}
        \item[(i)] $\mu_F\left(B_\rho\left[v_i, r_i\right] \cap \left(-\infty, w\right)\right) > 0,$ 
    
            $w_{i, 1} \in \left(B_\rho\left[v_i, r_i\right] \cap \left(-\infty, w\right)\right) \cap \left\{ w \in G : \left(\overleftarrow{F} \circ F|_G\right)\left(w\right) = w \right\} \cap \left(-\infty, x_1\right).$ 
        \item[(ii)] $\mu_F\left(B_\rho\left[v_i, r_i\right] \cap \left(w, +\infty\right)\right) > 0,$
    
            $w_{i, 2} \in \left(B_\rho\left[v_i, r_i\right] \cap \left(w, +\infty\right)\right) \cap \left\{ w \in G : \left(\overleftarrow{F} \circ F|_G\right)\left(w\right) = w \right\} \cap \left(-\infty, x_1\right).$ 
    \end{itemize} If either one of them vanishes, for which at most one may do, then we set the corresponding $w_{i, k} \coloneqq w$ and $\overline{w}_i \coloneqq \frac{w_{i, 1} + w_{i, 2}}{2}, \delta w_i \coloneqq w_{i, 2} - w_{i, 1} \in \left(0, 2r_i\right],$
    $$h_i \coloneqq \left\{\begin{array}{ll}
                               \frac{1}{2}\min \left\{ w - w_{i, 1}, w_{i, 2} - w \right\}, & \left(w_{i, 1} \neq w\right) \mathrm{and} \left(w_{i, 2} \neq w\right), \\
                               \frac{\delta w_i}{2}, & \left(w_{i, 1} \neq w\right) \mathrm{xor} \left(w_{i, 2} \neq w\right).
                           \end{array}\right.$$ ($\delta w_0 = 2$)

    Then for $i = n + 1,$ as $\mu_F\left(\left\{ w \right\}\right) = 0,$ we again separately discuss case by case: 
    \begin{itemize}
        \item[(i)] $\mu_F\left(B_\rho\left[v_{n + 1}, r_{n + 1}\right] \cap \left(-\infty, w\right)\right) > 0,$ 

            select $w_{n + 1, 1} \in \left(B_\rho\left[v_{n + 1}, r_{n + 1}\right] \cap \left(-\infty, w\right)\right) \cap \left\{ w \in G : \left(\overleftarrow{F} \circ F|_G\right)\left(w\right) = w \right\} \cap \left(-\infty, x_1\right).$ 
        \item[(ii)] $\mu_F\left(B_\rho\left[v_{n + 1}, r_{n + 1}\right] \cap \left(w, +\infty\right)\right) > 0,$

            pick $w_{n + 1, 2} \in \left(B_\rho\left[v_{n + 1}, r_{n + 1}\right] \cap \left(w, +\infty\right)\right) \cap \left\{ w \in G : \left(\overleftarrow{F} \circ F|_G\right)\left(w\right) = w \right\} \cap \left(-\infty, x_1\right).$ 
    \end{itemize} If either one of them vanishes, for which at most one will do, then we set the corresponding $w_{n + 1, k} \coloneqq w, \overline{w}_{n + 1} \coloneqq \frac{w_{n + 1, 1} + w_{n + 1, 2}}{2}, \delta w_{n + 1} \coloneqq w_{n + 1, 2} - w_{n + 1, 1} \in \left(0, 2r_{n + 1}\right]$
    as well as $$h_{n + 1} \coloneqq \left\{\begin{array}{ll}
                                        \frac{1}{2} \min \left\{ w - w_{n + 1, 1}, w_{n + 1, 2} - w \right\}, & \left(w_{n + 1, 1} \neq w\right) \mathrm{and} \left(w_{n + 1, 2} \neq w\right), \\
                                        \frac{\delta w_{n + 1}}{2}, & \left(w_{n + 1, 1} \neq w\right) \mathrm{xor} \left(w_{n + 1, 2} \neq w\right).
                                    \end{array}\right.$$

    Take $\epsilon_{n + 2} \coloneqq \min \left\{ \frac{1}{n + 2}, s, h_{n + 1} \right\} > 0.$ Again as 
    $$\mathfrak{V} \cap \left\{ \left(v, B_\rho\left[v, r\right]\right) : \mu_F\left(B_\rho\left[v, r\right]\right) = 0 \right\} \mathrm{is\ not\ fine\ at}\ w$$ but $\mathfrak{V}$ is fine at $w$
    we can find out a $B_\rho\left[v_{n + 2}, r_{n + 2}\right] \ni w$ such that $r_{n + 2} < \frac{\epsilon_{n + 2}}{2}$ and also $\mu_F\left(B_\rho\left[v_{n + 2}, r_{n + 2}\right]\right) > 0.$
    However, the choice of $\epsilon_{n + 2}$ and $\left\{ w_{n + 1, 1}, w_{n + 1, 2} \right\}$ enforces $B_\rho\left[v_{n + 2}, r_{n + 2}\right] \subseteq \left[w_{n + 1, 1}, w_{n + 1, 2}\right] = B_\rho\left[\overline{w}_{n + 1}, \frac{\delta w_{n + 1}}{2}\right]$ 
    so that $\mu_F\left(B_\rho\left[\overline{w}_{n + 1}, \frac{\delta w_{n + 1}}{2}\right]\right) \geq \mu_F\left(B_\rho\left[v_{n + 2}, r_{n + 2}\right]\right) > 0$ which completes the deductive step for mathematical induction.

    Observing the sequence of $\rho$-balls we constructed $\left\{ B_\rho\left[\overline{w}_n, \frac{\delta w_n}{2}\right] \right\}_{n = 1}^\infty$ 
    from the above procedure:
    \begin{itemize}
        \item[(i)] $0 < \delta w_n \leq 2r_n < \epsilon_n \leq \frac{1}{n} \rightarrow 0$ as $n \rightarrow \infty,$
        \item[(ii)] $\forall n \in \mathbb{N}^\star \Rightarrow \mu_F\left(B_\rho\left[\overline{w}_n, \frac{\delta w_n}{2}\right]\right) > 0,$
        \item[(iii)] $\forall n \in \mathbb{N}^\star, B_\rho\left[\overline{w}_n, \frac{\delta w_n}{2}\right] = \left[w_{n, 1}, w_{n, 2}\right] \ni w$ with 
            $$\left\{ w_{n, 1}, w_{n, 2} \right\} \subseteq \left\{ w \in G : \left(\overleftarrow{F} \circ F|_G\right)\left(w\right) = w \right\} \cap \left(-\infty, x_1\right),$$ 
            due to the rule in our selection procedure and also $w \in \left\{ w \in G : \left(\overleftarrow{F} \circ F|_G\right)\left(w\right) = w \right\} \cap \left(-\infty, x_1\right)$ at onset.
    \end{itemize}

    Now applying~(\ref{(15)}) above separately towards each $w_{n, k}, k = 1, 2$, as $\int_{\mathbb{R}} \left|\phi_A\right| d\mu_F < +\infty,$
    \begin{eqnarray}
        \int_{B_\rho\left[\overline{w}_n, \frac{\delta w_n}{2}\right]} \phi_A d\mu_F &=& \left(\int_{\left(-\infty, w_{n, 2}\right]} - \int_{\left(-\infty, w_{n, 2}\right)}\right) \phi_A d\mu_F \nonumber \\
                                                                                     &=& \left(\int_{\left(-\infty, w_{n, 2}\right]} - \int_{\left(-\infty, w_{n, 2}\right]}\right) \phi_A d\mu_F \nonumber \\
                                                                                     &=& 0, \nonumber 
    \end{eqnarray} where the second equality is again due to $\mu_F\left(\left\{ w_{n, 2} \right\}\right) = 0$ by $\left(-\infty, x_1\right) \subseteq D_F^\mathsf{c}.$

    Then we sandwich
    \begin{eqnarray}
        && \inf \left\{ \frac{\int_{B_\rho\left[v, r\right]} \phi_A d\mu_F}{\mu_F\left(B_\rho\left[v, r\right]\right)} : \left(w, B_\rho\left[v, r\right]\right) \in \mathfrak{V}, r < \frac{\epsilon_n}{2} \right\} \nonumber \\
        &\leq& \frac{\int_{B_\rho\left[\overline{w}_n, \frac{\delta w_n}{2}\right]} \phi_A d\mu_F}{\mu_F\left(B_\rho\left[\overline{w}_n, \frac{\delta w_n}{2}\right]\right)} \nonumber \\
        &\leq& \sup \left\{ \frac{\int_{B_\rho\left[v, r\right]} \phi_A d\mu_F}{\mu_F\left(B_\rho\left[v, r\right]\right)} : \left(w, B_\rho\left[v, r\right]\right) \in \mathfrak{V}, r < \frac{\epsilon_n}{2} \right\}, \nonumber
    \end{eqnarray} and let $n \rightarrow \infty$ and note that $w \in E$ we get $\phi_A\left(w\right) = 0,$ i.e., 
    $$\left\{ w \in G : \left(\overleftarrow{F} \circ F|_G\right)\left(w\right) = w \right\} \cap \left(-\infty, x_1\right) \cap E \subseteq \phi_A^{-1}\left(\left\{ 0 \right\}\right),$$
    for which again as $\mu_F^\star|_{\sigma\left(\tau_\rho\right)} = \mu_F$ and $\phi_A$ is $\sigma\left(\tau_\rho\right)$-measurable,
    \begin{eqnarray}
        0 &\leq& \mu_F\left(\left(-\infty, x_1\right) - \phi_A^{-1}\left(\left\{ 0 \right\}\right)\right) = \mu_F^\star\left(\left(-\infty, x_1\right) - \phi_A^{-1}\left(\left\{ 0 \right\}\right)\right) \nonumber \\
          &\leq& \mu_F\left(\left\{ w \in G : \left(\overleftarrow{F} \circ F|_G\right)\left(w\right) > w \right\} \cap \left(-\infty, x_1\right)\right) + \mu_F\left(\left(-\infty, x_1\right) - G\right) + \mu^\star_F\left(\mathbb{R} - E\right) \nonumber \\
          &\leq& 0, \nonumber
    \end{eqnarray} where the last inequality is due to 
    \begin{eqnarray}
        0 \leq \mu_F\left(\left(-\infty, x_1\right) - G\right) &=& \mu_F\left(\left(-\infty, x_1\right) \cap F^{-1}\left(\left\{ 0 \right\}\right)\right) + \mu_F\left(\left(-\infty, x_1\right) \cap F^{-1}\left(\left\{ 1 \right\}\right)\right) \nonumber \\
                                                               &\leq& \mu_F\left(F^{-1}\left(\left\{ 0 \right\}\right)\right) + \mu_F\left(\varnothing\right) = 0, \nonumber
    \end{eqnarray} by appealing to the (ii) of $5$ in Lemma~\ref{lma::2} and also that $F\left(x_1-\right) < F\left(x_1\right) \leq 1.$

    We have thus proved if $\mu_F\left(\left(-\infty, x_1\right)\right) = F\left(x_1-\right) > 0,$ then 
    $$\phi_A^{-1}\left(\left\{ 0 \right\}\right) \cap \left(-\infty, x_1\right) = \left(-\infty, x_1\right)\ \left[\mu_F\right],$$  
    for which it also holds if $\mu_F\left(\left(-\infty, x_1\right)\right) = F\left(x_1-\right) = 0$ as two sets on both sides are $\mu_F$-null in this case. With this we know,
    $\int_{\mathbb{R}} \mathbb{I}_{\left(-\infty, x_1\right)}\phi_A d\mu_F = 0.$

    Now,
    \begin{itemize}
        \item[\textbf{Case (i):}] $F\left(x_1\right) < 1$\textbf{.}
        
            Take $w = x_1, z = F\left(x_1\right) \in \left(F\left(x_1-\right), 1\right) \subseteq \left(0, 1\right)$ (since $F\left(x_1-\right) \geq 0$) in~(\ref{(14)}). In this case because $\forall n \in \Lambda_F^\star \ni n \geq 2 \Rightarrow F\left(x_1\right) - \mu_F\left(\left(-\infty, x_n\right)\right) = F\left(x_1\right) - F\left(x_n-\right) \leq 0$ 
            and also $F\left(x_1\right) - \mu_F\left(-\infty, x_1\right) = F\left(x_1\right) - F\left(x_1-\right) = \mu_F\left(\left\{ x_1 \right\}\right) > 0,$ then
            \begin{eqnarray}
                \forall x \in \mathbb{R} \Rightarrow I_1\left(x, F\left(x_1\right)\right) &=& H\left(\frac{F\left(x_1\right) - \mu_F\left(\left(-\infty, x_1\right)\right)}{\mu_F\left(\left\{ x_1 \right\}\right)}\right)\mathbb{I}_{\left\{ x_1 \right\}}\left(x\right) \nonumber \\
                                                                                          &=& H\left(1\right)\mathbb{I}_{\left\{ x_1 \right\}}\left(x\right). \label{(18)}
            \end{eqnarray}

            On the other hand, notice that as $\overleftarrow{F}\left(F\left(x_1\right)\right) \geq x_1$ by fourth bullet of Lemma~\ref{lma::2}, $\left(-\infty, \overleftarrow{F}\left(F\left(x_1\right)\right)\right] - D_F = \left(-\infty, x_1\right) \sqcup \left(\left(x_1, \overleftarrow{F}\left(F\left(x_1\right)\right)\right] - D_F\right).$
            With this it suffices to focus upon the case where $x_1 < \overleftarrow{F}\left(F\left(x_1\right)\right),$ in other words $\left(x_1, \overleftarrow{F}\left(F\left(x_1\right)\right)\right] \neq \varnothing.$

            If $x_2 < +\infty,$ then from~(\ref{(10)}) and~(\ref{(11)}) $x_2 \in \Lambda_F^\star$ and hence $F\left(x_2\right) > F\left(x_1\right)$ with $x_2 > x_1,$ 
            with the monotonicity of $F$ we know $x_2$ is an upper bound for $F^{-1}\left(\left(-\infty, F\left(x_1\right)\right]\right)$
            so that $\overleftarrow{F}\left(F\left(x_1\right)\right) \leq x_2.$ 
            Otherwise, $x_2 = +\infty, \overleftarrow{F}\left(F\left(x_1\right)\right) \leq x_2$ holds all the same.

            \begin{itemize}
                \item[\textbf{Subcase (a):}] $\overleftarrow{F}\left(F\left(x_1\right)\right) < x_2$\textbf{.} 

                    Then $\overleftarrow{F}\left(F\left(x_1\right)\right) \in \left(x_1, x_2\right) \subseteq D_F^\mathsf{c}$ 
                    and this implies $\left(x_1, \overleftarrow{F}\left(F\left(x_1\right)\right)\right] \subseteq \left(x_1, x_2\right) \subseteq D_F^\mathsf{c}$ so that 
                    \begin{eqnarray}
                        0 \leq \mu_F\left(\left(x_1, \overleftarrow{F}\left(F\left(x_1\right)\right)\right] - D_F\right) &=& \mu_F\left(\left(x_1, \overleftarrow{F}\left(F\left(x_1\right)\right)\right]\right) \nonumber \\
                                                                                                                         &=& F\left(\overleftarrow{F}\left(F\left(x_1\right)\right)\right) - F\left(x_1\right) \nonumber \\
                                                                                                                         &=& F\left(\overleftarrow{F}\left(F\left(x_1\right)\right)-\right) - F\left(x_1\right) \nonumber \\
                                                                                                                         &\leq& 0, \nonumber
                    \end{eqnarray} where the last inequality is due to the eighth item of Lemma~\ref{lma::2}.

                \item[\textbf{Subcase (b):}] $\overleftarrow{F}\left(F\left(x_1\right)\right) = x_2$\textbf{.} 

                    In this case necessarily $x_2 < +\infty$ so $F$ jumps at $x_2$ and hence, 
                    \begin{eqnarray}
                        \left(x_1, \overleftarrow{F}\left(F\left(x_1\right)\right)\right] - D_F &=& \left(x_1, \overleftarrow{F}\left(F\left(x_1\right)\right)\right) \nonumber \\
                        \Rightarrow 0 \leq \mu_F\left(\left(x_1, \overleftarrow{F}\left(F\left(x_1\right)\right)\right] - D_F\right) &=& \mu_F\left(\left(x_1, \overleftarrow{F}\left(F\left(x_1\right)\right)\right)\right) \nonumber \\
                                                                                                &=& F\left(\overleftarrow{F}\left(F\left(x_1\right)\right)-\right) - F\left(x_1\right) \nonumber \\
                                                                                                &\leq& 0. \nonumber
                    \end{eqnarray}
            \end{itemize}

            We thus know $\mu_F\left(\left(x_1, \overleftarrow{F}\left(F\left(x_1\right)\right)\right] - D_F\right) = 0,$ so that $\mathbb{I}_{\left(x_1, \overleftarrow{F}\left(F\left(x_1\right)\right)\right] - D_F} = 0\ \left[\mu_F\right].$
            In conjunction with $\int_{\mathbb{R}} \mathbb{I}_{\left(-\infty, x_1\right)} \phi_A d\mu_F = 0$ by linearity we obtain
            \begin{eqnarray}
                \int_{\mathbb{R}} I_2\left(\cdot, F\left(x_1\right)\right) \phi_A d\mu_F &=& \int_{\mathbb{R}} \mathbb{I}_{\left(-\infty, x_1\right)} \phi_A d\mu_F + \int_{\mathbb{R}} \mathbb{I}_{\left(x_1, \overleftarrow{F}\left(F\left(x_1\right)\right)\right] - D_F} \phi_A d\mu_F \nonumber \\
                                                                                         &=& 0. \nonumber
            \end{eqnarray}

            Substituting identity~(\ref{(18)}) and above back to~(\ref{(14)}) simply gives $H\left(1\right)\mu_F\left(\left\{ x_1 \right\}\right)\phi_A\left(x_1\right) = 0,$ 
            which implies $\phi_A\left(x_1\right) = 0$ as $H\left(1\right) > 0$ by assumption and $\mu_F\left(\left\{ x_1 \right\}\right) > 0$ in the current case where $F\left(x_1\right) < 1.$

        \item[\textbf{Case (ii):}] $F\left(x_1\right) = 1$ \textbf{and} $x_1 < +\infty$\textbf{.} 

            Then we can let $z = F\left(x_1\right) = 1$ in the derivations that lead to~(\ref{(13)}) 
            and make modifications correspondingly by noting that $\left\{ \mu_F\left(\left(-\infty, X\right]\right) \leq 1 \right\} = \Omega,$ 
            finally obtaining the counterpart of identity~(\ref{(14)}) in this case as 
            \begin{equation}
                \int_{\mathbb{R}} \left[I_1\left(x, F\left(x_1\right)\right) + I_2\left(x\right)\right]\phi_A\left(x\right) \mu_F\left(dx\right) = 0, \label{(19)}
            \end{equation}
            with $I_1$ defined as in~(\ref{(39)}) but $I_2 = \mathbb{I}_{D_F^\mathsf{c}}.$

            Further, observe that $\forall x \geq x_1 \Rightarrow F\left(x\right) = 1$ by monotonicity of $F,$ which implies through~(\ref{(11)}) that $D_F = \left\{ x_1 \right\}.$ Additionally, 
            $D_F^\mathsf{c} = \left[\left(-\infty, x_1\right) \cap D_F^\mathsf{c}\right] \sqcup \left[\left(x_1, +\infty\right) \cap D_F^\mathsf{c}\right]$ for which 
            \begin{itemize}
                \item[(a)] 
                    $\phi_A^{-1}\left(\left\{ 0 \right\}\right) \cap \left(-\infty, x_1\right) = \left(-\infty, x_1\right)\ \left[\mu_F\right]$ implies $\mathbb{I}_{\left(-\infty, x_1\right) \cap D_F^\mathsf{c}}\phi_A = 0\ \left[\mu_F\right],$
                \item[(b)]
                    $0 \leq \mu_F\left(\left(x_1, +\infty\right) \cap D_F^\mathsf{c}\right) \leq \mu_F\left(\left(x_1, +\infty\right)\right) = 1 - F\left(x_1\right) = 0$ yields $\mathbb{I}_{\left(x_1, +\infty\right) \cap D_F^\mathsf{c}}\phi_A = 0\ \left[\mu_F\right],$
            \end{itemize} so that we know $I_2\phi_A = 0\ \left[\mu_F\right].$ Also,
            \begin{eqnarray}
                \forall x \in \mathbb{R} \Rightarrow I_1\left(x, F\left(x_1\right)\right) &=& H\left(\frac{F\left(x_1\right) - \mu_F\left(\left(-\infty, x_1\right)\right)}{\mu_F\left(\left\{ x_1 \right\}\right)}\right)\mathbb{I}_{\left\{ x_1 \right\}}\left(x\right) \nonumber \\
                                                                                          &=& H\left(1\right)\mathbb{I}_{\left\{ x_1 \right\}}\left(x\right), \nonumber
            \end{eqnarray} so that substituting these known facts back to identity~(\ref{(19)}) 
            we again obtain $H\left(1\right)\mu_F\left(\left\{ x_1 \right\}\right)\phi_A\left(x_1\right) = 0$ and hence $\phi_A\left(x_1\right) = 0.$

        \item[\textbf{Case (iii):}] $x_1 = +\infty$\textbf{.}

            This remaining case is clear due to~(\ref{(11)}), $\# D_F = m = 0.$
    \end{itemize}

    The above discussion under induction index $k = 0$ informs us that not only $\phi_A^{-1}\left(\left\{ 0 \right\}\right) \cap \left(-\infty, x_1\right) = \left(-\infty, x_1\right)\ \left[\mu_F\right]$ but also so long as $x_1 < +\infty, \phi_A\left(x_1\right) = 0.$ Notice that by~(\ref{(10)}) and~(\ref{(11)}), 
    \begin{eqnarray}
        x_1 < +\infty \Leftrightarrow \Lambda_F^\star \neq \varnothing \Leftrightarrow \# D_F = m > 0. \nonumber
    \end{eqnarray}

    Assume $\Lambda_F^\star \neq \varnothing,$ given $k \in \Lambda_F^\star,$ assume for each $1 \leq n \leq k,$ we have $\phi_A^{-1}\left(\left\{ 0 \right\}\right) \cap \left(x_{n - 1}, x_n\right) = \left(x_{n - 1}, x_n\right)\ \left[\mu_F\right]$ and also,
    \begin{equation}
        \forall 1 \leq n \leq k \Rightarrow \phi_A\left(x_n\right) = 0. \label{(20)}
    \end{equation}

    Let us look at $n = k + 1:$
    \begin{itemize}
        \item[\textbf{Case (i):}] $\mu_F\left(\left(x_k, x_{k + 1}\right)\right) = F\left(x_{k + 1}-\right) - F\left(x_k\right) > 0$\textbf{.} 

            Then $\forall z \in \left(F\left(x_k\right), F\left(x_{k + 1}-\right)\right],$ we have
            $$\forall n \geq k + 1, n \in \Lambda_F^\star \Rightarrow z - \mu_F\left(\left(-\infty, x_n\right)\right) \leq F\left(x_{k + 1}-\right) - F\left(x_n-\right) \leq 0,$$
            \begin{eqnarray}
                \forall x \in \mathbb{R} &\Rightarrow& I_1\left(x, z\right)\phi_A\left(x\right) \nonumber \\
                                         &=& \sum_{n \in \Lambda_F^\star - \left\{ 1, 2, \cdots, k \right\}} H\left(\frac{z - \mu_F\left(\left(-\infty, x_n\right)\right)}{\mu_F\left(\left\{ x_n \right\}\right)}\right)\phi_A\left(x_n\right)\mathbb{I}_{\left\{ x_n \right\}}\left(x\right) \nonumber \\
                                         &=& 0. \nonumber
            \end{eqnarray}

            Now by identity~(\ref{(14)}) again whence $\forall w \in \left(x_k, x_{k + 1}\right) \ni z = F\left(w\right) \in \left(F\left(x_k\right), F\left(x_{k + 1}-\right)\right) \subseteq \left(0, 1\right)$ 
            (since $0 \leq F\left(x_k\right) < F\left(x_{k + 1}-\right) \leq F\left(x_{k + 1}\right) \leq 1$ and the existence of such $w$ is guaranteed by $F\left(x_{k + 1}-\right) - F\left(x_k\right) > 0$ along with definition for left limit of $F$ at $x_{k + 1}$) reduces to 
            \begin{equation}
                \int_{\mathbb{R}} I_2\left(x, F\left(w\right)\right)\phi_A\left(x\right) \mu_F\left(dx\right) = 0. \nonumber
            \end{equation}

            From inductive assumption due to the finiteness of $k,$ 
            \begin{eqnarray}
                \phi_A^{-1}\left(\left\{ 0 \right\}\right) \cap \left(-\infty, x_k\right) &=& \left[\sqcup_{n = 1}^k \phi_A^{-1}\left(\left\{ 0 \right\}\right) \cap \left(x_{n - 1}, x_n\right)\right] \sqcup \left[\sqcup_{n = 1}^{k - 1} \phi_A^{-1}\left(\left\{ 0 \right\}\right) \cap \left\{ x_n \right\}\right] \nonumber \\
                                                                                          &=& \left[\sqcup_{n = 1}^k \left(x_{n - 1}, x_n\right)\right] \sqcup \left[\sqcup_{n = 1}^{k - 1} \left\{ x_n \right\}\right]\ \left[\mu_F\right] \nonumber \\
                                                                                          &=& \left(-\infty, x_k\right), \nonumber
            \end{eqnarray} while on the other hand by fourth bullet of Lemma~\ref{lma::2} for each of the above fixed $w \in \left(x_k, x_{k + 1}\right), \overleftarrow{F}\left(F\left(w\right)\right) \geq w > x_k,$ 
            we have $\left(-\infty, \overleftarrow{F}\left(F\left(w\right)\right)\right] - D_F = \left(-\infty, x_k\right) \sqcup \left(\left(x_k, \overleftarrow{F}\left(F\left(w\right)\right)\right] - D_F\right).$

            Then we further reach 
            \begin{equation}
                \int_{\mathbb{R}} \mathbb{I}_{\left(x_k, \overleftarrow{F}\left(F\left(w\right)\right)\right] - D_F}\left(x\right)\phi_A\left(x\right) \mu_F\left(dx\right) = 0. \nonumber
            \end{equation}

            As before, by the first item of $5$ in Lemma~\ref{lma::2}, as $\left(x_k, x_{k + 1}\right) \subseteq D_F^\mathsf{c},$ 
            $$\mu_F\left(\left\{ w \in G : \left(\overleftarrow{F} \circ F|_G\right)\left(w\right) > w \right\} \cap \left(x_k, x_{k + 1}\right)\right) = 0,$$ 
            for which $\forall w \in \left\{ w \in G : \left(\overleftarrow{F} \circ F|_G\right)\left(w\right) = w \right\} \cap \left(x_k, x_{k + 1}\right)$ it then yields 
            \begin{equation}
                \int_{\mathbb{R}} \mathbb{I}_{\left(x_k, w\right]}\left(x\right)\phi_A\left(x\right) \mu_F\left(dx\right) = 0. \nonumber
            \end{equation}

            The argument to reach $\phi_A^{-1}\left(\left\{ 0 \right\}\right) \cap \left(x_k, x_{k + 1}\right) = \left(x_k, x_{k + 1}\right)\ \left[\mu_F\right]$ goes identical word by word as in our discussion for $k = 0$ 
            after modifying $\left(-\infty, x_1\right) = \left(x_0, x_1\right)$ towards $\left(x_k, x_{k + 1}\right),$ for which the dependencies of that proof we have shown on which the interval applied are only 
            $\left(x_k, x_{k + 1}\right) \subseteq D_F^\mathsf{c}$ along with $F\left(x_{k + 1}-\right) < F\left(x_{k +1}\right) \leq 1$ to make $\left(x_k, x_{k + 1}\right) \cap F^{-1}\left(\left\{ 1 \right\}\right) = \varnothing$ happens.
    \item[\textbf{Case (ii):}] $\mu_F\left(\left(x_k, x_{k + 1}\right)\right) = F\left(x_{k + 1}-\right) - F\left(x_k\right) = 0$\textbf{.}

        This case does not need much discussion to obtain $\phi_A^{-1}\left(\left\{ 0 \right\}\right) \cap \left(x_k, x_{k + 1}\right) = \left(x_k, x_{k + 1}\right)\ \left[\mu_F\right]$ for which sets on both sides are $\mu_F$-null.
    \end{itemize}

    Thus, we claim $\phi_A^{-1}\left(\left\{ 0 \right\}\right) \cap \left(x_k, x_{k + 1}\right) = \left(x_k, x_{k + 1}\right)\ \left[\mu_F\right],$ which further entails, 
    in conjunction with $\phi_A^{-1}\left(\left\{ 0 \right\}\right) \cap \left(-\infty, x_k\right) = \left(-\infty, x_k\right)\ \left[\mu_F\right]$ and $\phi_A\left(x_k\right) = 0,$ 
    $\phi_A^{-1}\left(\left\{ 0 \right\}\right) \cap \left(-\infty, x_{k+1}\right) = \left(-\infty, x_{k + 1}\right)\ \left[\mu_F\right]$ together with $\int_{\mathbb{R}} \mathbb{I}_{\left(-\infty, x_{k + 1}\right)}\phi_A d\mu_F = 0.$

    Then,
    \begin{itemize}
        \item[\textbf{Case (i):}] $F\left(x_{k + 1}\right) < 1$\textbf{.}

            Take $w = x_{k + 1}, z = F\left(x_{k + 1}\right) \in \left(F\left(x_{k + 1}-\right), 1\right) \subseteq \left(0, 1\right)$ (since $F\left(x_{k + 1}-\right) \geq 0$) in~(\ref{(14)}). Because 
            \begin{itemize}
                \item[(a)] $\forall n \in \Lambda_F^\star \ni n \geq k + 2 \Rightarrow F\left(x_{k + 1}\right) - \mu_F\left(\left(-\infty, x_n\right)\right) = F\left(x_{k + 1}\right) - F\left(x_{k + 2}-\right) \leq 0,$
                \item[(b)] $F\left(x_{k + 1}\right) - \mu_F\left(\left(-\infty, x_{k + 1}\right)\right) = F\left(x_{k + 1}\right) - F\left(x_{k + 1}-\right) = \mu_F\left(\left\{ x_{k + 1} \right\}\right) > 0,$
                \item[(c)] inductive assumption~(\ref{(20)}) above,
            \end{itemize}
            then
            \begin{eqnarray}
                \forall x \in \mathbb{R} &\Rightarrow& I_1\left(x, F\left(x_{k + 1}\right)\right)\phi_A\left(x\right) \nonumber \\
                                         &=& H\left(\frac{F\left(x_{k + 1}\right) - \mu_F\left(\left(-\infty, x_{k + 1}\right)\right)}{\mu_F\left(\left\{ x_{k + 1} \right\}\right)}\right)\phi_A\left(x_{k + 1}\right)\mathbb{I}_{\left\{ x_{k + 1} \right\}}\left(x\right) \nonumber \\
                                         &=& H\left(1\right)\phi_A\left(x_{k + 1}\right)\mathbb{I}_{\left\{ x_{k + 1} \right\}}\left(x\right). \nonumber
            \end{eqnarray}

            Again the remaining steps to reach $\phi_A\left(x_{k + 1}\right) = 0$ are simply by substituting $x_1, x_2$ respectively into $x_{k + 1}, x_{k + 2}$ 
            in the development of induction index $k = 0$ above, as the dependencies of those upon the pair $\left(x_k, x_{k + 1}\right)$ are again only 
            $\left(x_k, x_{k + 1}\right) \subseteq D_F^\mathsf{c}$ as well as $\int_{\mathbb{R}} \mathbb{I}_{\left(-\infty, x_{k + 1}\right)}\phi_A d\mu_F = 0.$
        \item[\textbf{Case (ii):}] $F\left(x_{k + 1}\right) = 1$ \textbf{and} $x_{k + 1} < +\infty$\textbf{.} 

            Then we can let $z = F\left(x_{k + 1}\right) = 1$ in the derivations that yield~(\ref{(13)}) 
            as before with the observation that $\left\{ \mu_F\left(\left(-\infty, X\right]\right) \leq 1 \right\} = \Omega,$ 
            eventually obtaining the counterpart of identity~(\ref{(14)}) in this case to be
            \begin{equation}
                \int_{\mathbb{R}} \left[I_1\left(x, F\left(x_{k + 1}\right)\right) + I_2\left(x\right)\right]\phi_A\left(x\right) \mu_F\left(dx\right) = 0, \label{(21)}
            \end{equation}
            for which the definitions of $I_i, i = 1, 2$ are identical to those laid out in (ii) right after~(\ref{(19)}) for induction index $n = 0.$

            The fact that $\forall x \geq x_{k + 1} \Rightarrow F\left(x\right) = 1$ by monotonicity of $F$ shall suggest through~(\ref{(11)}) that $D_F = \left\{ x_n \right\}_{n = 1}^{k + 1}.$ Additionally, with $\phi_A^{-1}\left(\left\{ 0 \right\}\right) \cap \left(-\infty, x_{k + 1}\right) = \left(-\infty, x_{k + 1}\right)\ \left[\mu_F\right]$ and $F\left(x_{k + 1}\right) = 1$ 
            we may invoke same counterpart argument for induction index $k = 0$ before by changing $x_1$ into $x_{k + 1}$ to reach $I_2\phi_A = 0\ \left[\mu_F\right].$

            Also, due to inductive assumption~(\ref{(20)}) one more time
            \begin{eqnarray}
                \forall x \in \mathbb{R} &\Rightarrow& I_1\left(x, F\left(x_{k + 1}\right)\right)\phi_A\left(x\right) \nonumber \\
                                         &=& H\left(\frac{F\left(x_{k + 1}\right) - \mu_F\left(\left(-\infty, x_{k + 1}\right)\right)}{\mu_F\left(\left\{ x_{k + 1} \right\}\right)}\right)\phi_A\left(x_{k + 1}\right)\mathbb{I}_{\left\{ x_{k + 1} \right\}}\left(x\right) \nonumber \\
                                         &=& H\left(1\right)\phi_A\left(x_{k + 1}\right)\mathbb{I}_{\left\{ x_{k + 1} \right\}}\left(x\right), \nonumber
            \end{eqnarray} so that substituting these known facts back to identity~(\ref{(21)}) 
            it yields $H\left(1\right)\mu_F\left\{ x_{k + 1} \right\}\phi_A\left(x_{k + 1}\right) = 0$ again and hence $\phi_A\left(x_{k + 1}\right) = 0.$

        \item[\textbf{Case (iii):}] $x_{k + 1} = +\infty$\textbf{.}

            This scenario is equivalent to $k + 1 \not\in \Lambda_F^\star,$ or further to $D_F = \left\{ x_n \right\}_{n = 1}^k,$ 
            for which there is nothing to induct upon.
    \end{itemize}

    We have therefore established:
    \begin{itemize}
        \item[(i)] If $\Lambda_F^\star = \varnothing,$ equivalently $x_1 = +\infty,$

            $\phi_A^{-1}\left(\left\{ 0 \right\}\right) \cap \mathbb{R} = \mathbb{R}\ \left[\mu_F\right].$ 

        \item[(ii)] If $\Lambda_F^\star \neq \varnothing,$

            $\forall n \in \Lambda_F^\star,$
            \begin{eqnarray}
                \phi_A^{-1}\left(\left\{ 0 \right\}\right) \cap \left(x_{n - 1}, x_n\right) &=& \left(x_{n - 1}, x_n\right)\ \left[\mu_F\right], \nonumber \\
                                                                                \phi_A\left(x_n\right) &=& 0, \nonumber
            \end{eqnarray} which further implies by the countability of $\Lambda_F^\star$ and~(\ref{(10)}) as well as~(\ref{(11)}) that gives $\mathbb{R} = \sqcup_{n \in \Lambda_F^\star} \left(x_{n - 1}, x_n\right) \sqcup \sqcup_{n \in \Lambda_F^\star} \left\{ x_n \right\},$
            \begin{eqnarray}
                \phi_A^{-1}\left(\left\{ 0 \right\}\right) \cap \mathbb{R} &=& \left[\sqcup_{n \in \Lambda_F^\star} \phi_A^{-1}\left(\left\{ 0 \right\}\right) \cap \left(x_{n - 1}, x_n\right)\right] \sqcup \left[\sqcup_{n \in \Lambda_F^\star} \phi_A^{-1}\left(\left\{ 0 \right\}\right) \cap \left\{ x_n \right\}\right] \nonumber \\
                                                                           &=& \left[\sqcup_{n \in \Lambda_F^\star} \left(x_{n - 1}, x_n\right)\right] \sqcup \left[\sqcup_{n \in \Lambda_F^\star} \left\{ x_n \right\}\right]\ \left[\mu_F\right] \nonumber \\
                                                                           &=& \mathbb{R}. \nonumber
            \end{eqnarray}
    \end{itemize}

    Recalling the definition of $\phi_A$ we may obtain, in either cases, for each $A \in \mathcal{F}$ and for $\mu_F$-almost all $x \in \mathbb{R}$ (whose exceptional set may depend on $A$), $P\left(\widetilde{Y} \in A\ |\ X = x\right) = \nu_{\widetilde{Y}}\left(A\right).$ 
    Then by definition of regular conditional distribution we know for each $B \in \sigma\left(\left\{ X \right\}\right),$
    \begin{eqnarray}
        P\left(\left\{ \widetilde{Y} \in A \right\} \cap \left\{ X \in B \right\}\right) &=& \int_B P\left(\widetilde{Y} \in A\ |\ X = x\right) \mu_F\left(dx\right) \nonumber \\
                                                                                         &=& P\left(\left\{ \widetilde{Y} \in A \right\}\right)P\left(\left\{ X \in B \right\}\right), \nonumber
    \end{eqnarray} implying
    $\sigma\left(\left\{ \widetilde{Y} \right\}\right) \independent \sigma\left(\left\{ X \right\}\right),$ and we are done.
\end{proof}

\begin{corollary}
    \label{cor::1} 
    Given two pairs of real-valued random variables $\left(X, Y\right)$ and $\left(U_X, U_Y\right)$ with marginal distribution functions $F_X, F_Y, H_X, H_Y : \mathbb{R} \rightarrow \left[0, 1\right]$
    such that
    \begin{itemize}
        \item[(i)]
            $\left(X, Y\right)$ and $\left(U_X, U_Y\right)$ are independent,
        \item[(ii)]
            $U_X$ and $U_Y$ are independent, and
        \item[(iii)]
        $H_X\left(0\right) = H_Y\left(0\right) = 0, H_X\left(1\right), H_Y\left(1\right) > 0.$
    \end{itemize} Additionally, denote by 
    \begin{eqnarray}
        Z_X &=& \left(1 - U_X\right)F_X\left(X-\right) + U_XF_X\left(X\right) \nonumber
    \end{eqnarray} and 
    \begin{eqnarray}
        Z_Y &=& \left(1 - U_Y\right)F_Y\left(Y-\right) + U_YF_Y\left(Y\right) \nonumber
    \end{eqnarray} the marginal Brockwell transforms of $X$ and $Y,$ respectively.
    Then the following propositions hold:
    \begin{enumerate}
        \item $X, Y$ are independent $\Rightarrow Z_X, Z_Y$ are independent.
        \item If both $F_X$ and $F_Y$ satisfy the $\leq$ well-ordering condition, then
              $$Z_X, Z_Y\ \mathrm{are\ independent}\ \Rightarrow X, Y\ \mathrm{are\ independent}.$$
    \end{enumerate}
\end{corollary}

\begin{proof} 
    For necessity let 
    \begin{align}
        f_F:\mathbb{R}\times\mathbb{R} &\rightarrow \mathbb{R} \nonumber \\ 
                                (x, u) &\mapsto (1-u)F(x-)+uF(x), \nonumber
    \end{align} which is measurable as a composition of measurable components $(x, u)\mapsto (F(x-), F(x), u)$ (due to the monotonicity of $F$ and $x\mapsto F(x-)$), $(y, z, w)\mapsto (1-w)y+wz$.

    For any $\left(u_x, x\right), \left(u_y, y\right) \in \mathbb{R}^2,$
    \begin{eqnarray}
        && P\left(\left\{ \left(U_X, X\right) \leq \left(u_x, x\right) \right\} \cap \left\{ \left(U_Y, Y\right) \leq \left(u_y, y\right) \right\}\right) \nonumber \\ 
        &=& P\left(\left\{ \left(U_X, U_Y\right) \leq \left(u_x, u_y\right) \right\} \cap \left\{ \left(X, Y\right) \leq \left(x, y\right) \right\}\right) \nonumber \\
        &=& P\left(\left\{ \left(U_X, U_Y\right) \leq \left(u_x, u_y\right) \right\}\right)P\left(\left\{ \left(X, Y\right) \leq \left(x, y\right) \right\}\right) \label{(7)} \\
        &=& P\left(\left\{ U_X \leq u_x \right\}\right)P\left(\left\{ U_Y \leq u_y \right\}\right)P\left(\left\{ X \leq x \right\}\right)P\left(\left\{ Y \leq y \right\}\right) \label{(8)} \\
        &=& P\left(\left\{ \left(U_X, X\right) \leq \left(u_x, x\right) \right\}\right)P\left(\left\{ \left(U_Y, Y\right) \leq \left(u_y, y\right) \right\}\right), \label{(9)}
    \end{eqnarray} where
    \begin{itemize}
        \item[(i)] in equality~(\ref{(7)}) and~(\ref{(9)}) we utilize $(U_X, U_Y)\independent (X, Y)$ along with the independence arising from any pair of coordinate projections,
        \item[(ii)] equality~(\ref{(8)}) is due to $U_X\independent U_Y$ and $X\independent Y$,
    \end{itemize} so that we arrive at $(U_X, X)\independent (U_Y, Y)$. Applying $f_{F_X}$ and $f_{F_Y}$ to both sides yield the desired conclusion.

    Towards sufficiency, notice $Z_Y$ serves as a valid candidate for $\widetilde{Y}$ in Theorem~\ref{thm::4}:
    \begin{itemize}
        \item[(i)] $(\mathbb{R}, d)$ is Polish, henceforth $(\mathbb{R}, \sigma\left(\tau_d\right))$ is a Borel space.
        \item[(ii)] $Z_X\independent Z_Y$ is the starting condition with $H_X(0)=0$, $H_X(1)>0$ satisfied by the cdf of $U_X$.
        \item[(iii)] 
            The condition $\left(U_X, U_Y\right)\independent \left(X, Y\right)$ along with $U_X\independent U_Y$ yields $U_X\independent\left(U_Y, X, Y\right)$ since 
            \begin{eqnarray}
                \left(U_X, \left(U_Y, X, Y\right)\right)&\sim& \left(U_X, U_Y\right)\times \left(X, Y\right) \nonumber \\
                                                        &\sim& U_X\times U_Y \times \left(X, Y\right) \nonumber \\
                                                        &\sim& U_X \times \left(U_Y, \left(X, Y\right)\right), \nonumber
            \end{eqnarray} which implies $\left(U_X, \left(U_Y, Y\right)\right)|X$ is independent for which $U_X|X\sim U_X$, and we further apply $f_{F_Y}$ to the second coordinate to obtain 
            $\left(U_X, Z_Y\right)|X\sim U_X \times \left(Z_Y|X\right)$.
    \end{itemize}

    Thus, $X\independent Z_Y$ is established. Now interchanging the role of $X$ and $Y$ notation-wisely in the above conclusion, we also have 
    for any Borel space valued random map $\widetilde{X}:\Omega\rightarrow M$ such that $Z_Y\independent \widetilde{X}$ and $(U_Y, \widetilde{X})|X\sim U_Y|X \times \widetilde{X}|X\sim U_Y\times \widetilde{X}|X$ 
    as $U_Y\independent X$, with $H_Y(0)=0$, $H_Y(1)>0$ then $\widetilde{X}\independent Y$. This time $X$ turns out to be the choice for $\widetilde{X}$, since $X|X$ is simply degenerate, 
    and then we conclude the proof towards $X\independent Y$.
\end{proof}

\begin{corollary}
    \label{cor::2} 
    Given two independent triples of real-valued random variables $\left(X_1, X_2, X_3\right)$ and $\left(U_1, U_2, U_3\right)$ 
    such that $\left(U_1, U_2, U_3\right)$ is uniformly distributed in the unit cube $\left(0, 1\right)^3.$ 
    In addition, let the Brockwell transforms of $X_1$ given $X_3,$ $X_2$ given $X_3$ and $X_3$ be 
    \begin{eqnarray}
        Z_{1 | 3} &=& \left(1 - U_1\right)F_{1 | 3}\left(X_1- | X_3\right) + U_1F_{1 | 3}\left(X_1 | X_3\right), \nonumber \\
        Z_{2 | 3} &=& \left(1 - U_2\right)F_{2 | 3}\left(X_2- | X_3\right) + U_2F_{2 | 3}\left(X_2 | X_3\right) \nonumber
    \end{eqnarray} and
    \begin{eqnarray}
        Z_3 &=& \left(1 - U_3\right)F_3\left(X_3-\right) + U_3F_3\left(X_3\right), \nonumber
    \end{eqnarray} respectively. Also assume the Borel measurability of 
    \begin{eqnarray}
        \left(x_1, x_3\right) &\mapsto& F_{1 | 3}\left(x_1 | x_3\right), \nonumber \\
        \left(x_1, x_3\right) &\mapsto& F_{1 | 3}\left(x_1- | x_3\right), \nonumber \\
        \left(x_2, x_3\right) &\mapsto& F_{2 | 3}\left(x_2 | x_3\right), \nonumber \\
        \left(x_2, x_3\right) &\mapsto& F_{2 | 3}\left(x_2- | x_3\right). \nonumber
    \end{eqnarray}
    \begin{enumerate}
        \item 
            Then, $X_1$ and $X_2$ are conditionally independent given $X_3 \Rightarrow$ the set of transformed random variables $\left\{ Z_{1 | 3}, Z_{2 | 3}, Z_3 \right\}$ is independent.
        \item
            Conversely, if, for $PX^{-1}_3$-almost all $x \in \mathbb{R},$ all the distribution functions 
            $F_{1 | 3}\left(\cdot | x\right), F_{2 | 3}\left(\cdot | x\right)$ and $F_3$ satisfy the $\leq$ well-ordering condition,
            then $\left\{ Z_{1 | 3}, Z_{2 | 3}, Z_3 \right\}$ is independent $\Rightarrow X_1$ and $X_2$ are conditionally independent given $X_3.$
    \end{enumerate}
\end{corollary}

\begin{proof}
    Notice if we apply (i) in the last item of Lemma~\ref{lma::2} to each fixed $x \in \mathbb{R},$ outside a $PX^{-1}_3$-null set if necessary, of the conditional distribution $X_1 | X_3 = x \sim F_{1 | 3}\left(\cdot | x\right)$ with $U_1 \independent \left(X_1 | X_3 = x\right)$ 
    then we obtain $Z_{1 | 3} | \left(X_3 = x\right) \sim \mathrm{uniform}\left(0, 1\right),$ which implies that $Z_{1 | 3} \independent X_3.$ Identical reasoning substituting $X_1$ as $X_2$ yields 
    $Z_{2 | 3} \independent X_3.$

    With the above observation we see that $\left\{ Z_{1 | 3}, Z_{2 | 3}, X_3 \right\}$ is independent if and only if $\left(Z_{1 | 3}, Z_{2 | 3}\right) | X_3$ is independent, i.e.,
    $\left(Z_{1 | 3}, Z_{2 | 3}\right) | X_3 \sim \left(Z_{1 | 3} | X_3\right) \times \left(Z_{2 | 3} | X_3\right) \sim Z_{1 | 3} \times Z_{2 | 3}.$

    Similarly, by virtue of Theorem~\ref{thm::4} conditionally to each $x \in \mathbb{R}$ that does not belong to the assumed $PX^{-1}_3$-null set, 
    that both $F_{1 | 3}\left(\cdot | x\right)$ and $F_{2 | 3}\left(\cdot | x\right)$ satisfy the $\leq$ well-ordering condition shall entail 
    $\left(Z_{1 | 3}, Z_{2 | 3}\right) | \left(X_3 = x\right)$ is independent implies $\left(X_1, X_2\right) | \left(X_3 = x\right)$ is independent 
    while this converse holds regardless.

    It remains for us to show that $\left\{ Z_{1 | 3}, Z_{2 | 3}, X_3 \right\}$ is independent entails $\left\{ Z_{1 | 3}, Z_{2 | 3}, Z_3 \right\}$ is independent,
    and its sufficiency holds under the $\leq$ well-ordering condition on $F_3.$

    First of all, $U_3 \independent \left(X_1, U_1, X_2, U_2, X_3\right)$ as for each $u_3 \in \mathbb{R}, \left(x_1, u_1, x_2, u_2, x_3\right) \in \mathbb{R}^5,$
    \begin{eqnarray}
        && P\left(\left\{ U_3 \leq u_3 \right\} \cap \left\{ \left(X_1, U_1, X_2, U_2, X_3\right) \leq \left(x_1, u_1, x_2, u_2, x_3\right) \right\}\right) \nonumber \\
        &=& P\left(\left\{ \left(U_1, U_2, U_3\right) \leq \left(u_1, u_2, u_3\right) \right\}\right)P\left(\left\{ \left(X_1, X_2, X_3\right) \leq \left(x_1, x_2, x_3\right) \right\}\right) \label{(22)} \\
        &=& P\left(\left\{ U_3 \leq u_3 \right\}\right)P\left(\left\{ \left(U_1, U_2\right) \leq \left(u_1, u_2\right) \right\}\right)P\left(\left\{ \left(X_1, X_2, X_3\right) \leq \left(x_1, x_2, x_3\right) \right\}\right) \nonumber \\ \label{(23)} \\
        &=& P\left(\left\{ U_3 \leq u_3 \right\}\right)P\left(\left\{ \left(X_1, U_1, X_2, U_2, X_3\right) \leq \left(x_1, u_1, x_2, u_2, x_3\right) \right\}\right), \label{(24)}
    \end{eqnarray} in which 
    \begin{itemize}
        \item[(i)] 
            in equality~(\ref{(22)}) and~(\ref{(24)}) we use the condition $\left(U_1, U_2, U_3\right) \independent \left(X_1, X_2, X_3\right)$ 
            along with each coordinate projection that can be applied to any side of the independence equation,
        \item[(ii)]
            equality~(\ref{(23)}) is for $\left(U_1, U_2, U_3\right) \sim \mathrm{uniform}\left(0, 1\right)^3$ so that 
            $\left\{ U_1, U_2, U_3 \right\}$ is also independent and hence $U_3 \independent \left(U_1, U_2\right).$
    \end{itemize}

    Notice that when $\left(x_1, x_3\right) \mapsto F\left(x_1 | x_3\right)$ and $\left(x_1, x_3\right) \mapsto F\left(x_1- | x_3\right)$ are measurable, we obtain the measurability of 
    $\left(x_1, u, x_3\right) \mapsto \left(F\left(x_1- | x_3\right), F\left(x_1 | x_3\right), u\right)$ checking through coordinate function by coordinate function, 
    and then composing it with the continuous $\left(y, z, u\right) \mapsto \left(1 - u\right)y + uz$ 
    we obtain the measurability of $\left(x_1, u, x_3\right) \mapsto f_{F\left(\cdot | x_3\right)}\left(x_1, u\right)$ 
    with $f_F$ defined at the onset of the proof for Corollary~\ref{cor::1}.

    Therefore, due to the Borel measurability assumption $\left(x_1, u_1, x_2, u_2, x_3\right) \mapsto \left(f_{F_{1 | 3}\left(\cdot | x_3\right)}\left(x_1, u_1\right), f_{F_{2 | 3}\left(\cdot | x_3\right)}\left(x_2, u_2\right), x_3\right)$
    is measurable again as each such coordinate function is measurable. 
    Then we apply this measurable vector-valued function to the right of $U_3 \independent \left(X_1, U_1, X_2, U_2, X_3\right)$
    to claim $U_3 \independent \left(Z_{1 | 3}, Z_{2 | 3}, X_3\right).$ 

    \begin{itemize}
        \item[``$\Rightarrow$''] Necessity.
            Assume $\left\{ Z_{1 | 3}, Z_{2 | 3}, X_3 \right\}$ is independent. We have 
            \begin{eqnarray}
                \left(\left(Z_{1 | 3}, Z_{2 | 3}\right), \left(X_3, U_3\right)\right) &\sim& \left(Z_{1 | 3}, Z_{2 | 3}, X_3\right) \times U_3 \nonumber \\
                                                                                      &\sim& \left(Z_{1 | 3}, Z_{2 | 3}\right) \times X_3 \times U_3 \nonumber \\
                                                                                      &\sim& \left(Z_{1 | 3}, Z_{2 | 3}\right) \times \left(X_3, U_3\right) \nonumber \\
                                                                                      &\sim& \left(Z_{1 | 3} \times Z_{2 | 3}\right) \times \left(X_3, U_3\right), \nonumber
            \end{eqnarray} where in the first and third distributional equality we have utilized $\left(Z_{1 | 3}, Z_{2 | 3}, X_3\right) \independent U_3$ 
            while the second and fourth are owing to the condition that $\left\{ Z_{1 | 3}, Z_{2 | 3}, X_3 \right\}$ is independent.

            The above distributional equation allows us to apply $f_{F_3}$ to the second component of it to obtain 
            $\left(Z_{1 | 3}, Z_{2 | 3}, Z_3\right) \sim \left(\left(Z_{1 | 3}, Z_{2 | 3}\right), Z_3\right) \sim Z_{1 | 3} \times Z_{2 | 3} \times Z_3,$ 
            which means that $\left\{ Z_{1 | 3}, Z_{2 | 3}, Z_3\right\}$ is independent.
        \item[``$\Leftarrow$''] Sufficiency.
            Let $\left\{ Z_{1 | 3}, Z_{2 | 3}, Z_3 \right\}$ be independent. 
            Then we immediately know $\left(Z_{1 | 3}, Z_{2 | 3}\right) \independent Z_3.$
            On the other hand by $U_3 \independent \left(Z_{1 | 3}, Z_{2 | 3}, X_3\right)$ we also know 
            $\left(\left(Z_{1 | 3}, Z_{2 | 3}\right), U_3\right) | X_3$ is independent. 
            Next, the marginal distribution function for each of the $U_1, U_2, U_3$'s clearly has $0$ as its fixed point and is positive at $1.$
            Last but not least $\left(\mathbb{R}^2, \rho_2\right)$ is Polish, 
            implying that $\left(\mathbb{R}^2, \sigma\left(\tau_{\rho_2}\right)\right)$ is a Borel space.

            The four conditions stated sequentially above with the assumption that $F_3$ satisfies the $\leq$ well-ordering condition allows us to resort to Theorem~\ref{thm::4} 
            with $\left(Z_{1 | 3}, Z_{2 | 3}\right)$ being $\widetilde{Y}$ there to get $\left(Z_{1 | 3}, Z_{2 | 3}\right) \independent X_3.$
            In the presence of $\left\{ Z_{1 | 3}, Z_{2 | 3}, Z_3 \right\}$ being independent, we also know $Z_{1 | 3} \independent Z_{2 | 3}$ and hence 
            these two claims together conclude the independence of $\left\{ Z_{1 | 3}, Z_{2 | 3}, X_3 \right\}.$
    \end{itemize}
\end{proof}

\begin{remark}
    \label{rmk::6} 
    Note that Corollary~\ref{cor::2} is stronger than Corollary~\ref{cor::1} in the case where 
    $\left(U_1, U_2\right)$ is uniformly distributed in the unit square $\left(0, 1\right)^2$ by considering the following reasoning.
\end{remark}

\begin{proof} 
    Let $X_3 = 0$ and another $U_{X_3} \sim \mathrm{uniform}\left(0, 1\right)$ such that $U_{X_3} \independent \left(U_X, U_Y, X, Y\right),$ 
    which implies
    \begin{itemize}
        \item[(i)] $U_{X_3} \independent \left(U_X, U_Y\right)$ so that $\left(U_X, U_Y, U_{X_3}\right) \sim \left(U_X, U_Y\right) \times U_{X_3} \sim \mathrm{uniform}\left(0, 1\right)^3.$ 
        \item[(ii)]
            \begin{eqnarray}
                \left(\left(U_{X_3}, U_X, U_Y\right), \left(X, Y\right)\right) &\sim& U_{X_3} \times \left(\left(U_X, U_Y\right), \left(X, Y\right)\right) \nonumber \\
                                                                               &\sim& U_{X_3} \times \left(U_X, U_Y\right) \times \left(X, Y\right) \nonumber \\
                                                                               &\sim& \left(U_{X_3}, U_X, U_Y\right) \times \left(X, Y\right), \nonumber
            \end{eqnarray} where the first and the third equality is for $U_{X_3} \independent \left(U_X, U_Y, X, Y\right)$ 
            whereas the second by the condition $\left(U_X, U_Y\right) \independent \left(X, Y\right).$

            Thus, we know $\left(U_{X_3}, U_X, U_Y\right) \independent \left(X, Y\right),$ or equivalently $\left(U_X, U_Y, U_{X_3}\right) \independent \left(X, Y\right).$ 
            On the other hand since $\sigma\left(\left\{ X_3 \right\}\right) = \left\{ \varnothing, \Omega \right\},$ we know $\sigma\left(\left\{ X, Y, X_3 \right\}\right) = \sigma\left(\left\{ X, Y \right\}\right),$ 
            and consequently $\left(U_X, U_Y, U_{X_3}\right) \independent \left(X, Y, X_3\right).$
        \item[(iii)]
            $\forall x \in \mathbb{R}, F_{X_3}\left(x\right) = \mathbb{I}_{\left[0, +\infty\right)}\left(x\right), F_{X_3}\left(x-\right) = \mathbb{I}_{\left(0, +\infty\right)}\left(w\right)$ so that $Z_{X_3} = U_{X_3}$ 
            and $D_{F_{X_3}} = \left\{ 0 \right\}$ which is finite and hence $F_{X_3}$ fulfills the $\leq$ well-ordering condition.
        \item[(iv)]
            In this scenario as $\sigma\left(\left\{ X_3 \right\}\right) = \left\{ \varnothing, \Omega \right\}$ clearly $\left(X, Y\right) \independent X_3$ and hence $X \independent X_3.$ 
            This observation leads to the conclusion that $\forall \left(x, x_3\right) \in \mathbb{R}^2 \Rightarrow F_{X | X_3}\left(x | x_3\right) = F_X\left(x\right), F_{X | X_3}\left(x- | x_3\right) = F_X\left(x-\right),$
            (For each $x_3 \in \mathbb{R}^\star,$ values of $F_{X | X_3}\left(x | x_3\right)$ and $F_{X | X_3}\left(x- | x_3\right)$ may be arbitrary.) 
            which further yields the measurability of both $\left(x, x_3\right) \mapsto F_{X | X_3}\left(x | x_3\right) = F_X\left(x\right)$
            and $\left(x, x_3\right) \mapsto F_{X | X_3}\left(x- | x_3\right) = F_X\left(x-\right)$ due to the monotonicity of $F$ and $x \mapsto F\left(x-\right).$

            Identical discussion applied to $Y \independent X_3$ also gives the measurability of $\left(y, x_3\right) \mapsto F_{Y | X_3}\left(y | x_3\right) = F_Y\left(y\right)$
            as well as $\left(y, x_3\right) \mapsto F_{Y | X_3}\left(y- | x_3\right) = F_Y\left(y-\right).$

            Note that we also get $Z_{X | X_3} = Z_X$ and $Z_{Y | X_3} = Z_Y.$
        \item[(v)]
            In the presence of $\left(U_X, U_Y, U_{X_3}\right) \sim \mathrm{uniform}\left(0, 1\right)^3,$ 
            $\left(U_X, U_Y, U_{X_3}\right) \independent \left(X, Y, X_3\right)$ along with the four measurability assumptions 
            we actually obtain, done in the proof of Corollary~\ref{cor::2}, 
            $U_{X_3} \independent \left(Z_{X | X_3}, Z_{Y | X_3}, X_3\right).$

            Due to $Z_{X_3} = U_{X_3}, Z_{X | X_3} = Z_X, Z_{Y | X_3} = Z_Y$ and $\sigma\left(\left\{ X_3 \right\}\right) = \left\{ \varnothing, \Omega \right\},$ we already have $Z_{X_3} \independent \left(Z_X, Z_Y\right).$
    \end{itemize}

    After observing (i)-(v) listed above, we appeal to Corollary~\ref{cor::2} to reach the conclusion 
    $\left(X, Y\right) \sim \left(X, Y\right) | X_3$ is independent implies $\left\{ Z_X, Z_Y, Z_{X_3} \right\}$ is independent
    whereas the converse direction can be guaranteed by the assumption that $F_X = F_{X | X_3}\left(\cdot | 0\right)$ and $F_Y = F_{Y | X_3}\left(\cdot | 0\right)$ satisfy the $\leq$ well-ordering condition.
    But in the presence of $Z_{X_3} \independent \left(Z_X, Z_Y\right), \left\{ Z_X, Z_Y, Z_{X_3} \right\}$ is independent
    is equivalent to $Z_X \independent Z_Y.$
\end{proof}

\begin{appendices}

\section{Proof of the Completeness of $\left\lVert \cdot \right\lVert$} \label{pfa}

\begin{proof}
    For the completeness of $\lVert \cdot \rVert$, given $\left\{\kappa_n\right\}_{n=1}^\infty\subseteq \mathcal{K}$ Cauchy in $\lVert \cdot \rVert$, since 
    \begin{eqnarray}
        \forall (x, A) \in \mathbb{R}\times \mathcal{F},  |\kappa_n-\kappa_m|(x, A)&\leq& |\kappa_n-\kappa_m|(x, M)\leq \lVert \kappa_n-\kappa_m \rVert, \label{(40)} \\ 
        |\kappa_n-\kappa_m|(x, A)&=&(\kappa_n^++\kappa_m^-)(x, A)+(\kappa_n^-+\kappa_m^+)(x, A), \nonumber \\\label{(41)}
    \end{eqnarray} we know each $\left\{ \kappa_n(x, A)\right\}_{n=1}^\infty$ is Cauchy in $\mathbb{R}$ so that one could define $\kappa(x, A)=\lim_{n\rightarrow \infty}\kappa_n(x, A)\in \mathbb{R}$ in a point-wise manner.

    Then for each $A\in \mathcal{F}$, the $\sigma\left(\tau_\rho\right)$-measurability of each term in $\left\{ \kappa_n(\cdot, A)\right\}_{n=1}^\infty$ shall imply that of $\kappa(\cdot, A)$. 
    Additionally, given $x\in \mathbb{R}$, for each disjoint sequence $\left\{A_m\right\}_{m=1}^\infty$ such that $|\kappa\left(x, \sqcup_{m=1}^\infty A_m\right)|<+\infty$, we have for sufficiently large $n\in \mathbb{N}^\star$, 
    $|\kappa_n\left(x, \sqcup_{m=1}^\infty A_m\right)|<+\infty$ and thus by countable additivity for each $s\in\left\{-, +\right\}$
    $$\kappa_n^s\left(x, \sqcup_{m=1}^\infty A_m\right)=\sum_{m=1}^\infty \kappa_n^s\left(x, A_m\right).$$  
    Letting $n\rightarrow +\infty$ in the above, monotone convergence theorem further implies $\lim_{n\rightarrow \infty}\kappa_n^s\left(x, \sqcup_{m=1}^\infty A_m\right)=\sum_{m=1}^\infty \lim_{n\rightarrow \infty} \kappa_n^s\left(x, A_m\right).$

    As for each $A\in \mathcal{F},$ $\left\{\kappa_n^+\left(x, A\right)\right\}_{n=1}^\infty$, $\left\{\kappa_n^-\left(x, A\right)\right\}_{n=1}^\infty$ are Cauchy in $\mathbb{R}$,
    \begin{eqnarray}
        \kappa\left(x, \sqcup_{m=1}^\infty A_m\right)&=&\lim_{n\rightarrow \infty}\kappa_n\left(x, \sqcup_{m=1}^\infty A_m\right) \nonumber \\
                                             &=&\lim_{n\rightarrow \infty} \kappa_n^+\left(x, \sqcup_{m=1}^\infty A_m\right)-\lim_{n\rightarrow \infty} \kappa_n^-\left(x, \sqcup_{m=1}^\infty A_m\right) \nonumber \\
                                             &=&\sum_{m=1}^\infty \left(\kappa_n^+\left(x, A_m\right)-\kappa_n^-\left(x, A_m\right)\right) \nonumber \\
                                             &=&\sum_{m=1}^\infty \kappa(x, A_m), \nonumber
    \end{eqnarray} so that $\kappa(x, \cdot)$ is a finite signed measure on $\mathcal{F}$.

    Finally, take $A=M$ in~(\ref{(41)}) above and let $m\rightarrow \infty$ for each fixed $n\in \mathbb{N}^\star$, combining~(\ref{(40)}) along with $\kappa^s(x, A)=\lim_{n\rightarrow \infty}\kappa_n^s(x, A)$ for $s\in\left\{-, +\right\}$ we know for each $x\in \mathbb{R}$, 
    $|\kappa_n-\kappa|(x, M)\leq \overline{\lim}_{m\rightarrow \infty}\lVert \kappa_n-\kappa_m \rVert$, which implies that $\lVert \kappa_n - \kappa\rVert \leq \overline{\lim}_{m\rightarrow \infty}\lVert \kappa_n-\kappa_m \rVert\rightarrow 0$ as $n\rightarrow \infty$. 
    The finiteness of $\lVert \kappa \rVert$ follows from the triangle inequality of $\lVert\cdot\rVert$: $$\lVert \kappa \rVert \leq \lVert \kappa -\kappa_{n_0}\rVert + \lVert \kappa_{n_0} - \kappa_{m_0}\rVert + \lVert \kappa_{m_0} \rVert<+\infty,$$ 
    for some fixed $n_0, m_0\in \mathbb{N}^\star$ that is sufficiently large.
\end{proof}

\section{Proof of the Boundedness of $T_{\mu_F}$} \label{pfb}

\begin{proof}
    \begin{eqnarray}
        \lVert T_{\mu_F} \kappa \rVert &=& \sup_{x\in \mathbb{R}}\left(\left(T_{\mu_F}\kappa\right)^++\left(T_{\mu_F}\kappa\right)^-\right)(x, M) \nonumber \\ 
                                       &=& \sup_{x\in \mathbb{R}}\left(T_{\mu_F}\kappa^++T_{\mu_F}\kappa^-\right)(x, M) \nonumber \\
                                       &=& \sup_{x\in \mathbb{R}}\left(T_{\mu_F}|\kappa|\right)(x, M) \nonumber \\
                                       &\leq& 2 \sup_{x\in \mathbb{R}} |\kappa|(x, M)=2\lVert \kappa \rVert. \nonumber
    \end{eqnarray}
\end{proof}

\section{Proof of Lemma~\ref{lma::1}} \label{pflma::1}

\begin{proof} 
    The case of a finite $A$ is clear. When $A$ is countably infinite, let $A = \left\{ y_n \right\}_{n = 1}^\infty \subseteq \mathbb{R}$ with pairwise distinct elements.
    Due to the fact that $A$ is well-ordered by $\leq$ on $\mathbb{R},$ we infer that it must have a minimal element, say $y_{n_1},$ with respect to the relation $\leq$ as a nonempty subset of itself. Then $A - \left\{ y_{n_1} \right\}$ is also countably infinite, so that $\varnothing \neq A - \left\{ y_{n_1} \right\} \subseteq A.$ 
    Now, by well-ordering of $A$ under $\leq,$ $A - \left\{ y_{n_1} \right\}$ must have a minimal element, say $y_{n_2},$ for which $y_{n_1} < y_{n_2}.$ 
    Continuing in this manner by induction, we obtain a permutation of $A = \left\{ y_n \right\}_{n = 1}^\infty$ such that

    $$-\infty < y_{n_1} < y_{n_2} < \cdots < y_{n_i} < y_{n_{i + 1}} < \cdots < +\infty.$$ Relabeling $x_i = x_{n_i}$ yields the desired construction.
\end{proof}

\section{Proof of Lemma~\ref{lma::2}} \label{pflma::2}

\begin{proof}
    For each of the item below, when the argument is \textit{symmetric}, in the sense that it does not invoke the \textit{asymmetric} property of the one-sided continuity for $F$ that it is c\`adl\`ag on $\mathbb{R}$
    with the difference only lying in the \textit{direction} of the order (upper sided interval with infimum versus lower sided interval with supremum), 
    we then only prove one version of the quantile function. But otherwise we give two slightly different proofs.

    For the first item, we only give a proof for $\overrightarrow{F}$ with the other one following symmetrically: Since $\lim_{x \rightarrow -\infty} F\left(x\right) = 0,$ 
    given $y \in \left(0, 1\right)$ there exists an $X_y > 0$ such that $\forall x < -X_y \Rightarrow F\left(x\right) < y$ so that $F^{-1}\left(\left[y, +\infty\right)\right) \subseteq \left[-X_y, +\infty\right).$ 
    Additionally, as $\lim_{x \rightarrow +\infty} F\left(x\right) = 1,$ we know there must exist an $x_0\left(y\right) > 0$ such that $F\left(x_0\right) \geq y$ so that $F^{-1}\left(\left[y, +\infty\right)\right) \neq \varnothing.$ 
    Henceforth, $\overrightarrow{F}\left(y\right) = \inf F^{-1}\left(\left[y, +\infty\right)\right) \in \mathbb{R}.$ Monotonicity of $\overrightarrow{F}$ follows from the fact that $\left[y, +\infty\right)$ is decreasing as $y$ increases, 
    then taking inverse image preserves this ordering and finally infimum operation reverses it.

    Towards the second item, assume there exists a $y \in \left(0, 1\right)$ such that $\overrightarrow{F}\left(y\right) > \overleftarrow{F}\left(y\right).$ Pick any $\overleftarrow{F}\left(y\right) < x < \overrightarrow{F}\left(y\right),$ then 
    $\overleftarrow{F}\left(y\right) = \sup F^{-1}\left(\left(-\infty, y\right]\right)$ shall lead to $F\left(x\right) > y,$ whereas $\overrightarrow{F}\left(y\right) = \inf F^{-1}\left(\left[y, +\infty\right)\right)$ yields $F\left(x\right) < y,$ which is a contradiction.

    If we let 
    \begin{eqnarray}
        A &=& \left\{ y \in \left(0, 1\right) : \overrightarrow{F}\left(y\right) < \overleftarrow{F}\left(y\right) \right\}, \nonumber \\
        B &=& \left\{ y \in \left(0, 1\right) : \mathrm{int}\left(F^{-1}\left(\left\{ y \right\}\right)\right) \neq \varnothing \right\}, \nonumber \\ 
        C &=& \left\{ y \in \left(0, 1\right) : \mathrm{int}\left(F^{-1}\left(\left\{ y \right\}\right)\right) = \left(\overrightarrow{F}\left(y\right), \overleftarrow{F}\left(y\right)\right) \right\}, \nonumber
    \end{eqnarray} then it is not hard to see that $C \subseteq B$ and $C \subseteq A.$ Thus, it suffices to prove $A \subseteq B \subseteq C.$

    \begin{itemize}
        \item[(i)] $A \subseteq B.$

            Fix any $y \in \left(0, 1\right)$ such that $\overrightarrow{F}\left(y\right) < \overleftarrow{F}\left(y\right).$ Then for any $x \in \left(\overrightarrow{F}\left(y\right), \overleftarrow{F}\left(y\right)\right)$ 
            by the definition of $\overleftarrow{F}$ as well as $\overrightarrow{F}$ again we know there exists $x_1 < x < x_2$ such that $F\left(x_1\right) \geq y, F\left(x_2\right) \leq y$ so by monotonicity of $F,$
            $$y \leq F\left(x_1\right) \leq F\left(x\right) \leq F\left(x_2\right) \leq y,$$ i.e., $F\left(x\right) = y$ or $x \in F^{-1}\left(\left\{ y \right\}\right).$ This simply implies that $\left(\overrightarrow{F}\left(y\right), \overleftarrow{F}\left(y\right)\right) \subseteq F^{-1}\left(\left\{ y \right\}\right)$
            and thus $\mathrm{int}\left(F^{-1}\left(\left\{ y \right\}\right)\right) \neq \varnothing.$
        \item[(ii)] $B \subseteq C.$
            Given any $y \in \left(0, 1\right)$ with $\mathrm{int}\left(F^{-1}\left(\left\{ y \right\}\right)\right) \neq \varnothing.$ $\mathrm{int}\left(F^{-1}\left(\left\{ y \right\}\right)\right)$ is bounded in $\mathbb{R}$:
            Otherwise there exists a sequence $\left\{ x_n \right\}_{n = 1}^\infty,$ that either diverges to $+\infty$ or $-\infty,$ such that $\forall n \in \mathbb{N}^\star, F\left(x_n\right) = y$
            but this contradicts $\lim_{x \rightarrow +\infty} F\left(x\right) = 1$ or $\lim_{x \rightarrow -\infty} F\left(x\right) = 0.$

            Then denote $m_y = \inf \left(\mathrm{int}\left(F^{-1}\left(\left\{ y \right\}\right)\right)\right), M_y = \sup \left(\mathrm{int}\left(F^{-1}\left(\left\{ y \right\}\right)\right)\right)$ with $-\infty < m_y \leq M_y < +\infty.$

            For each $x \in \mathrm{int}\left(F^{-1}\left(\left\{ y \right\}\right)\right),$
            due to the openness of $\mathrm{int}\left(F^{-1}\left(\left\{ y \right\}\right)\right),$ there exists an interval $\left(x_1, x_2\right) \subseteq \mathrm{int}\left(F^{-1}\left(\left\{ y \right\}\right)\right)$ such that
            $x \in \left(x_1, x_2\right),$ and henceforth $m_y \leq x_1 < x < x_2 \leq M_y,$ i.e., $x \in \left(m_y, M_y\right).$ 
            Furthermore, $\mathrm{int}\left(F^{-1}\left(\left\{ y \right\}\right)\right) \neq \varnothing$ shall imply $m_y < M_y.$

            On the other hand, $\forall x \in \left(m_y, M_y\right)$ by the definition of the infimum and the supremum again we have there exists $x_1, x_2 \in \mathrm{int}\left(F^{-1}\left(\left\{ y \right\}\right)\right)$ such that $x_1 < x < x_2.$
            Since $F\left(x_1\right) = F\left(x_2\right) = y,$ the monotonicity of $F$ forces $\left[x_1, x_2\right] \subseteq F^{-1}\left(\left\{ y \right\}\right)$ so that $\left(x_1, x_2\right) \subseteq \mathrm{int}\left(F^{-1}\left(\left\{ y \right\}\right)\right),$ 
            and $x \in \left(x_1, x_2\right)$ shall yield $x \in \mathrm{int}\left(F^{-1}\left(\left\{ y \right\}\right)\right),$ and we achieve $\mathrm{int}\left(F^{-1}\left(\left\{ y \right\}\right)\right) = \left(m_y, M_y\right).$ 

            For each $x \in F^{-1}\left(\left[y, +\infty\right)\right),$ if $F\left(x\right) > y$ then as $\mathrm{int}\left(F^{-1}\left(\left\{ y \right\}\right)\right) \neq \varnothing$ there exists an $x_0 \in \mathbb{R}$
            such that $F\left(x_0\right) = y,$ then we must have $x > x_0$ due to monotonicity of $F,$ and hence $x > m_y,$ whereas if $F\left(x\right) = y,$ we claim $x \geq m_y,$ otherwise if $x < m_y,$ 
            then by monotonicity of $F$ together with $\mathrm{int}\left(F^{-1}\left(\left\{ y \right\}\right)\right) \neq \varnothing$ again we know $y = F\left(x\right) \leq F\left(m_y\right) \leq y,$ so that $F\left(m_y\right) = y,$ 
            but this shall imply $\left(x, m_y\right) \subseteq F^{-1}\left(\left\{ y \right\}\right),$ therefore further $\left(x, m_y\right) \subseteq \mathrm{int}\left(F^{-1}\left(\left\{ y \right\}\right)\right),$ 
            contradicting the definition of $m_y.$ We have thus proved $m_y$ is a lower bound for $F^{-1}\left(\left[y, +\infty\right)\right).$ That $m_y$ is the greatest such is clear: For each $\epsilon > 0,$ 
            we can find a point $x_\epsilon \in \mathrm{int}\left(F^{-1}\left(\left\{ y \right\}\right)\right) \subseteq F^{-1}\left(\left\{ y \right\}\right) \subseteq F^{-1}\left(\left[y, +\infty\right)\right)$ such that $x_\epsilon < m_y + \epsilon.$

            Now we know $y = \inf F^{-1}\left(\left[y, +\infty\right)\right) = \overrightarrow{F}\left(y\right)$ and a symmetric argument leads to $M_y = \sup F^{-1}\left(\left[y, +\infty\right)\right) = \overleftarrow{F}\left(y\right).$
    \end{itemize}

    Continuing the proof for the ``and'' part of the second item, for each $y \in E,$ as we have proved 
    $\mathrm{int}\left(F^{-1}\left(\left\{ y \right\}\right)\right) = \left(\overrightarrow{F}\left(y\right), \overleftarrow{F}\left(y\right)\right) \neq \varnothing,$ then fix one rational $r_y \in \left(\overrightarrow{F}\left(y\right), \overleftarrow{F}\left(y\right)\right) \cap \mathbb{Q}.$
    Notice that if $y_1 \neq y_2,$ then $F^{-1}\left(\left\{ y_1 \right\}\right) \cap F^{-1}\left(\left\{ y_2 \right\}\right) = \varnothing,$ and so is their respective interior, i.e., 
    $\left(\overrightarrow{F}\left(y_1\right), \overleftarrow{F}\left(y_1\right)\right) \cap \left(\overrightarrow{F}\left(y_2\right), \overleftarrow{F}\left(y_2\right)\right) = \varnothing,$ which entails $r_{y_1} \neq r_{y_2}$ in our choice. 

    The above explanation tells us the map 
    \begin{align}
        \varphi : E &\rightarrow \mathbb{Q} \nonumber \\
               y &\mapsto r_y \nonumber
    \end{align} is injective, and hence $\# E \leq \# \mathbb{Q} = \aleph_0.$

    For the next item, first of all as proved just now we know $\forall y \in E, \forall x \in \left(\overrightarrow{F}\left(y\right), \overleftarrow{F}\left(y\right)\right) \Rightarrow F\left(x\right) = y,$ 
    then by the right-continuity of $F$ at $\overrightarrow{F}\left(y\right)$ along with the definition of $x \mapsto F\left(x-\right)$ at $\overleftarrow{F}\left(y\right)$ we immediately reach $F\left(\overrightarrow{F}\left(y\right)\right) = F\left(\overleftarrow{F}\left(y\right)-\right) = y.$

    In addition, for the fourth item, $\forall x \in G \Rightarrow x \in F^{-1}\left(\left[F\left(x\right), +\infty\right)\right)$ so that $x \geq \overrightarrow{F}\left(F\left(x\right)\right).$ 
    The disjointness of $\left\{ E_{1, y} \right\}_{y \in E}$ follows from the observation when proving the second that if $y_1 \neq y_2,$ then 
    $\left(\overrightarrow{F}\left(y_1\right), \overleftarrow{F}\left(y_1\right)\right) \cap \left(\overrightarrow{F}\left(y_2\right), \overleftarrow{F}\left(y_2\right)\right) = \varnothing$
    which is still disjoint even if we close either one of the right end point for these two intervals.

    Now $\forall x \in \sqcup_{y \in E} E_{1, y}$ there exists a $y \in E$ such that $x \in E_{1, y}.$ If $F$ is left continuous at $\overleftarrow{F}\left(y\right),$ 
    then from the third item just proved we know $F\left(\overrightarrow{F}\left(y\right)\right) = F\left(\overleftarrow{F}\left(y\right)\right) = y$ and hence $\overrightarrow{F}\left(y\right) < x \leq \overleftarrow{F}\left(y\right)$ shall imply that $F\left(x\right) = y \in \left(0, 1\right)$
    so that $x \in G$ and $\overrightarrow{F}\left(F\left(x\right)\right) < x.$ If $F$ jumps at $\overleftarrow{F}\left(y\right),$ 
    then $x \in \left(\overrightarrow{F}\left(y\right), \overleftarrow{F}\left(y\right)\right) = \mathrm{int}\left(F^{-1}\left(\left\{ y \right\}\right)\right)$ by the definition of $E,$ so we still have $\left(F\left(x\right) = y \in \left(0, 1\right)\right) \Rightarrow \left(x \in G\right)$ and $\overrightarrow{F}\left(F\left(x\right)\right) = \overrightarrow{F}\left(y\right) < x.$

    Conversely, given $x \in \mathbb{R}$ with $F\left(x\right) \in \left(0, 1\right)$ such that $\overrightarrow{F}\left(F\left(x\right)\right) < x,$ then by the definition of the infimum we claim there exists an $x^\prime \in \mathbb{R}, F\left(x^\prime\right) \geq F\left(x\right)$ such that $x^\prime < x$
    but this also implies $F\left(x^\prime\right) \leq F\left(x\right)$ and hence by the monotonicity of $F$ one more time we have $\left(x^\prime, x\right) \subseteq F^{-1}\left(\left\{ F\left(x\right)\right\}\right).$
    Therefore, $\mathrm{int}\left(F^{-1}\left(\left\{ F\left(x\right) \right\}\right)\right) \neq \varnothing$ so that $F\left(x\right) \in E$ with $\mathrm{int}\left(F^{-1}\left(\left\{ F\left(x\right) \right\}\right)\right) = \left(\overrightarrow{F}\left(F\left(x\right)\right), \overleftarrow{F}\left(F\left(x\right)\right)\right).$

    We also know $x \leq \overleftarrow{F}\left(F\left(x\right)\right),$ and if $F$ is left continuous at $\overleftarrow{F}\left(F\left(x\right)\right),$ combined with $\overrightarrow{F}\left(F\left(x\right)\right) < x,$ we immediately know 
    $x \in \left(\overrightarrow{F}\left(F\left(x\right)\right), \overleftarrow{F}\left(F\left(x\right)\right)\right] = E_{1, F\left(x\right)};$ on the other hand, if $F$ jumps at $\overleftarrow{F}\left(F\left(x\right)\right),$ then as $F|_{\left(\overrightarrow{F}\left(F\left(x\right)\right), \overleftarrow{F}\left(F\left(x\right)\right)\right)} = F\left(x\right)$ we have
    $F\left(x\right) = F\left(\overleftarrow{F}\left(F\left(x\right)\right)-\right) < F\left(\overleftarrow{F}\left(F\left(x\right)\right)\right),$ excluding the possibility that $x = \overleftarrow{F}\left(F\left(x\right)\right)$ and hence $x < \overleftarrow{F}\left(F\left(x\right)\right).$
    Therefore, we again have $x \in \left(\overrightarrow{F}\left(F\left(x\right)\right), \overleftarrow{F}\left(F\left(x\right)\right)\right) = E_{1, F\left(x\right)}.$

    In the same manner symmetrically $\forall x \in G \Rightarrow x \in F^{-1}\left(\left(-\infty, F\left(x\right)\right]\right)$ so that $x \leq \overleftarrow{F}\left(F\left(x\right)\right)$ 
    and so is the disjointness of $\left\{ E_{2, y} \right\}_{y \in E}.$ We also see $\forall x \in \sqcup_{y \in E} E_{2, y}$ there exists a $y \in E$ such that $x \in E_{2, y},$ 
    and furthermore $\overrightarrow{F}\left(y\right) \leq x < \overleftarrow{F}\left(y\right)$ 
    with the conclusion of the third item we obtain $F\left(x\right) = y \in \left(0, 1\right),$ resulting in $x \in G$ and $\overleftarrow{F}\left(F\left(x\right)\right) > x.$

    For the reverse direction, given $x \in \mathbb{R}$ with $F\left(x\right) \in \left(0, 1\right)$ such that $\overleftarrow{F}\left(F\left(x\right)\right) > x$
    and by the definition of supremum we have got as before $F\left(x\right) \in E$ with $\mathrm{int}\left(F^{-1}\left(\left\{ F\left(x\right) \right\}\right)\right) = \left(\overrightarrow{F}\left(F\left(x\right)\right), \overleftarrow{F}\left(F\left(x\right)\right)\right) \neq \varnothing.$
    Then $x \geq \overrightarrow{F}\left(F\left(x\right)\right)$ in conjunction with $\overleftarrow{F}\left(F\left(x\right)\right) > x$ shall entail
    $x \in \left[\overrightarrow{F}\left(F\left(x\right)\right), \overleftarrow{F}\left(F\left(x\right)\right)\right) = E_{2, F\left(x\right)},$ which completes the proof of this fourth item.

    For the first subclaim of the fifth item, notice by the previous item,

    $$\left\{ x \in G : \left(\overrightarrow{F} \circ F|_G\right)\left(x\right) < x \right\} = \sqcup_{y \in E} E_{1, y}.$$

    First of all for each $x \in E_{1, y}$ where $F$ is (left) continuous at $\overleftarrow{F}\left(y\right), \overrightarrow{F}\left(y\right) < x \leq \overleftarrow{F}\left(y\right)$ 
    by the third item again we have $y = F\left(\overrightarrow{F}\left(y\right)\right) \leq F\left(x-\right) \leq F\left(x\right) \leq F\left(\overleftarrow{F}\left(y\right)\right) = F\left(\overleftarrow{F}\left(y\right)-\right) = y$ 
    so that $\left(F\left(x-\right) = F\left(x\right)\right) \Rightarrow \left(x \in D_F^\mathsf{c}\right).$ The other case when $F$ jumps at $\overleftarrow{F}\left(y\right)$ is essentially similar as 
    $\overrightarrow{F}\left(y\right) < x < \overleftarrow{F}\left(y\right)$ shall imply $y = F\left(\overrightarrow{F}\left(y\right)\right) \leq F\left(x-\right) \leq F\left(x\right) \leq F\left(\overleftarrow{F}\left(y\right)-\right) = y.$

    \begin{itemize}
        \item[(i)] $F$ is (left) continuous at $\overleftarrow{F}\left(y\right).$

            By the left continuity and by the right continuity of $F$ at $\overleftarrow{F}\left(y\right)$ and $\overrightarrow{F}\left(y\right),$ respectively,
            \begin{eqnarray}
                \mu_F\left(\left(\overrightarrow{F}\left(y\right), \overleftarrow{F}\left(y\right)\right]\right) &=& F\left(\overleftarrow{F}\left(y\right)\right) - F\left(\overrightarrow{F}\left(y\right)\right) \nonumber \\
                                                                                                                 &=& F\left(\overleftarrow{F}\left(y\right)-\right) - F\left(\overrightarrow{F}\left(y\right)+\right) \nonumber \\
                                                                                                                 &=& y - y = 0, \nonumber
            \end{eqnarray} so that $\mu_F\left(E_{1, y}\right) = 0.$
        \item[(ii)] $F$ jumps at $\overleftarrow{F}\left(y\right).$
            \begin{eqnarray}
                \mu_F\left(\left(\overrightarrow{F}\left(y\right), \overleftarrow{F}\left(y\right)\right)\right) &=& F\left(\overleftarrow{F}\left(y\right)-\right) - F\left(\overrightarrow{F}\left(y\right)\right) \nonumber \\
                                                                                                                 &=& F\left(\overleftarrow{F}\left(y\right)-\right) - F\left(\overrightarrow{F}\left(y\right)+\right) \nonumber \\
                                                                                                                 &=& y - y = 0, \nonumber
            \end{eqnarray} and hence $\mu_F\left(E_{1, y}\right) = 0.$
    \end{itemize}

    Besides, given any open interval $\left(a, b\right) \subseteq D_F^\mathsf{c},$
    $$\left\{ x \in G : \left(\overleftarrow{F} \circ F|_G\right)\left(x\right) > x \right\} \cap \left(a, b\right) = \sqcup_{y \in E} \left(E_{2, y} \cap \left(a, b\right)\right),$$
    for which for each $E_{2, y} \cap \left(a, b\right)$ such that it is non-empty, 
    $$E_{2, y} \cap \left(a, b\right) = L_a \sqcup \left(\max \left\{ a, \overrightarrow{F}\left(y\right) \right\}, \min \left\{ b, \overleftarrow{F}\left(y\right) \right\}\right),$$
    where $L_a = \left\{\begin{array}{cc}
                            \varnothing, & \mathrm{if}\ a \geq \overrightarrow{F}\left(y\right), \nonumber \\
                            \left\{ \overrightarrow{F}\left(y\right) \right\}, & \mathrm{if}\ a < \overrightarrow{F}\left(y\right) \nonumber
                        \end{array}\right.$ 
    such that $\mu_F\left(L_a\right) = 0$ since $F\left(\overrightarrow{F}\left(y\right)-\right) = F\left(\overrightarrow{F}\left(y\right)\right)$ provided $a < \overrightarrow{F}\left(y\right)$ 
    due to in this case $E_{2, y} \cap \left(a, b\right) \neq \varnothing$ necessarily implies $b > \overrightarrow{F}\left(y\right)$ so that $\overrightarrow{F}\left(y\right) \in \left(a, b\right) \subseteq D_F^\mathsf{c}.$

    Next as $\left(\max \left\{ a, \overrightarrow{F}\left(y\right) \right\}, \min \left\{ b, \overleftarrow{F}\left(y\right) \right\}\right) \subseteq \left(\overrightarrow{F}\left(y\right), \overleftarrow{F}\left(y\right)\right) = \mathrm{int}\left(F^{-1}\left(\left\{ y \right\}\right)\right)$ we then have
    \begin{eqnarray}
        \mu_F\left(\left(\max \left\{ a, \overrightarrow{F}\left(y\right) \right\}, \min \left\{ b, \overleftarrow{F}\left(y\right) \right\}\right)\right) &=& F\left(\min \left\{ b, \overleftarrow{F}\left(y\right) \right\}-\right) - F\left(\max \left\{ a, \overrightarrow{F}\left(y\right) \right\}\right) \nonumber \\
                                                                                                                                                           &=& F\left(\min \left\{ b, \overleftarrow{F}\left(y\right) \right\}-\right) - F\left(\max \left\{ a, \overrightarrow{F}\left(y\right) \right\}+\right) \nonumber \\
                                                                                                                                                           &=& y - y = 0. \nonumber
    \end{eqnarray} With $\mu_F\left(L_a\right) = 0$ it yields $\mu_F\left(E_{2, y} \cap \left(a, b\right)\right) = 0.$

    Now that we have proved for each $y \in E, \mu_F\left(E_{1, y}\right) = \mu_F\left(E_{2, y} \cap \left(a, b\right)\right) = 0,$ the countability of $E$ established at the end of the second item concludes the main statement.

    For the second sub-item, if $F^{-1}\left(\left\{ 0 \right\}\right) = \varnothing,$ then the conclusion is clear. 
    Otherwise, if $F^{-1}\left(\left\{ 0 \right\}\right) \neq \varnothing,$ then $F^{-1}\left(\left\{ 0 \right\}\right)$ is bounded above, because similar to a previous argument 
    on the contrary we shall obtain a sequence $\left\{ x_n \right\}_{n = 1}^\infty$ that converges to $+\infty$ with 
    $\forall n \in \mathbb{N}^\star \Rightarrow F\left(x_n\right) = 0,$ contradicting $\lim_{x \rightarrow +\infty} F\left(x\right) = 1.$ Let $x_0 = \sup F^{-1}\left(\left\{ 0 \right\}\right),$ we aim to show that
    \begin{eqnarray}
        F^{-1}\left(\left\{ 0 \right\}\right) &=& \left\{\begin{array}{cc}
                                                            \left(-\infty, x_0\right], & \mathrm{if}\ F\ \mathrm{is\ (left)\ continuous\ at}\ x_0, \nonumber \\
                                                            \left(-\infty, x_0\right), & \mathrm{if}\ F\ \mathrm{jumps\ at}\ x_0. \nonumber \end{array}\right. \nonumber \\ \label{(4)}
    \end{eqnarray}

    Since $x_0$ is an upper bound for $F^{-1}\left(\left\{ 0 \right\}\right),$ we have, by the definition of $\leq, F^{-1}\left(\left\{ 0 \right\}\right) \subseteq \left(-\infty, x_0\right].$ 
    When $F$ jumps at $x_0,$ then $F\left(x_0\right) > 0,$
    otherwise by monotonicity of $F$ $\forall x \leq x_0 \Rightarrow F\left(x\right) = 0,$ which implies $F\left(x_0\right) = F\left(x_0-\right) = 0$ a contradiction, 
    so $\left\{ x_0 \right\} \not\subseteq F^{-1}\left(\left\{ 0 \right\}\right),$ i.e., $F^{-1}\left(\left\{ 0 \right\}\right) \subseteq \left(-\infty, x_0\right)$ if $F$ jumps at $x_0.$

    Conversely, suppose $x \not\in F^{-1}\left(\left\{ 0 \right\}\right),$ in other words, $F\left(x\right) > 0,$ then by the definition of supremum, we know 
    $\forall \epsilon > 0, \exists x_\epsilon \in \left(x_0 - \epsilon, x_0\right] \ni F\left(x_\epsilon\right) = 0.$ Now, 
    \begin{itemize}
        \item[(i)] $F$ is (left) continuous at $x_0.$
            
            For now the left continuity of $F$ at $x_0$ implies $F\left(x_0\right) = F\left(x_0-\right) = 0.$ 
            Then by monotonicity of $F$ clearly $x > x_0,$ or $x \not\in \left(-\infty, x_0\right].$
        \item[(ii)] $F$ jumps at $x_0.$

            Upon the current situation, assume $x < x_0,$ by taking the above $\epsilon = x_0 - x > 0$ gives an $x_\epsilon > x$ that satisfies $F\left(x_\epsilon\right) = 0$ 
            but, this enforces $F\left(x\right) = 0$ by monotonicity of $F$ hence a contradiction arises again. Therefore, $x \not\in \left(-\infty, x_0\right).$
    \end{itemize}

    By the proved equality~(\ref{(4)}) it is clearly true that $\forall x \in F^{-1}\left(\left\{ 0 \right\}\right), F$ is left continuous at $x$ 
    as in both cases $F|_{\left(-\infty, x_0\right)} = 0$ for which $\left(-\infty, x\right] \subseteq \left(-\infty, x_0\right)$ if $x < x_0$ 
    and only if $F$ is left continuous at $x_0, x$ may take $x_0.$ 
    This means $F^{-1}\left(\left\{ 0 \right\}\right) \subseteq D_F^\mathsf{c}$ and hence $F^{-1}\left(\left\{ 0 \right\}\right) - D_F = F^{-1}\left(\left\{ 0 \right\}\right).$

    Furthermore, as at the onset of the proof for the converse direction of this containment, 
    due to the definition of supremum we always know $F\left(x_0-\right) = 0,$ and then,
    \begin{eqnarray}
        \mu_F\left(F^{-1}\left(\left\{ 0 \right\}\right)\right) &=& \left\{\begin{array}{lc}
                                                                                F\left(x_0\right) = F\left(x_0-\right) = 0, & \mathrm{if}\ F\ \mathrm{is\ (left)\ continuous\ at}\ x_0, \nonumber \\
                                                                                F\left(x_0-\right) = 0, & \mathrm{if}\ F\ \mathrm{jumps\ at}\ x_0. \nonumber 
                                                                            \end{array}\right. \nonumber
    \end{eqnarray}

    Analogously, if $F^{-1}\left(\left\{ 1 \right\}\right) = \varnothing,$ then there is nothing to prove. 
    Otherwise, if $F^{-1}\left(\left\{ 1 \right\}\right) \neq \varnothing,$ then $F^{-1}\left(\left\{ 1 \right\}\right)$ is bounded below due to $\lim_{x \rightarrow -\infty} F\left(x\right) = 0.$ Let $x_1 = \inf F^{-1}\left(\left\{ 1 \right\}\right),$ we intend to prove that
    $F^{-1}\left(\left\{ 1 \right\}\right) = \left[x_1, +\infty\right).$

    Since $x_1$ is a lower bound for $F^{-1}\left(\left\{ 1 \right\}\right),$ we have, by the definition of $\geq, F^{-1}\left(\left\{ 1 \right\}\right) \subseteq \left[x_1, +\infty\right).$ 
    Conversely, suppose $x \in \left[x_1, +\infty\right),$ in other words, $x \geq x_1,$ then by the definition of the infimum, we know 
    $\forall \epsilon > 0, \exists x_\epsilon \in \left[x_1, x_1 + \epsilon\right) \ni F\left(x_\epsilon\right) = 1.$ The right continuity of $F$ gives us 
    $F\left(x_1\right) = F\left(x_1+\right) = 1$ and so by monotonicity of $F, F\left(x\right) = 1$ so that $x \in F^{-1}\left(\left\{ 1 \right\}\right).$

    Now, 
    \begin{eqnarray}
        && \mu_F\left(F^{-1}\left(\left\{ 1 \right\}\right) - D_F\right) \nonumber \\
        &=& \left\{\begin{array}{lc}
                        \mu_F\left(\left[x_1, +\infty\right)\right) = 1 - F\left(x_1-\right) = 1 - F\left(x_1\right) = 0, & \mathrm{if}\ F\ \mathrm{is\ (left)\ continuous\ at}\ x_1, \nonumber \\
                        \mu_F\left(\left(x_1, +\infty\right)\right) = 1 - F\left(x_1\right) = 0, & \mathrm{if}\ F\ \mathrm{jumps\ at}\ x_1, \nonumber 
                   \end{array}\right. \nonumber
    \end{eqnarray} which completes this second sub-item.

    For the ``In particular'' part of the theorem, by (i) of the last item we know when $F$ is continuous, $F\left(X\right) \in \left(0, 1\right)\ \left[P\right],$ which implies $X \in G\ \left[P\right].$ Note that 
    \begin{eqnarray}
        \left\{ \overrightarrow{F}\left(F\left(X\right)\right) \neq X \right\} \cap \left\{ X \in G \right\} &=& \left\{ \left(\overrightarrow{F} \circ F|_G\right)\left(X\right) < X \right\} \cap \left\{ X \in G \right\} \nonumber \\
                                                                                                             &=& X^{-1}\left(\left\{ x \in G : \left(\overrightarrow{F} \circ F|_G\right)\left(x\right) < x \right\}\right). \nonumber
    \end{eqnarray}

    The continuity of $F,$ in other words, $D_F = \varnothing,$ allows us to take $\left(a, b\right) = \left(-\infty, +\infty\right) \subseteq \varnothing^\mathsf{c}$ specifically in the second sub-item of the conclusion just proved and also as $X \sim F,$ any constructed $P$ 
    satisfies $PX^{-1} = \mu_F.$

    To the sixth item, $F\left(x\right) \geq y$ entails $x \in F^{-1}\left(\left[y, +\infty\right)\right),$ so we achieve $\overrightarrow{F}\left(y\right) = \inf F^{-1}\left(\left(-\infty, y\right]\right) \leq x.$ 
    Conversely, if $F\left(x\right) < y,$ then by monotonicity of $F, \forall x^\prime \in F^{-1}\left(\left[y, +\infty\right)\right) \Rightarrow x^\prime > x$ so that 
    $x$ is a lower bound for $F^{-1}\left(\left[y, +\infty\right)\right)$ and hence $\overrightarrow{F}\left(y\right) \geq x.$ 
    The equality here cannot hold for otherwise $\forall \epsilon > 0, \exists x_\epsilon \in  \left[x, x + \epsilon\right) \ni F\left(x_\epsilon\right) \geq y$
    and by the right continuity of $F$ at $x,$ we infer that $F\left(x\right) \geq y$ contradicting $F\left(x\right) < y.$ Thus, $\overrightarrow{F}\left(y\right) > x.$

    Similarly, $F\left(x\right) \leq y$ shall imply $x \in F^{-1}\left(\left(-\infty, y\right]\right),$ therefore we immediately know $\overleftarrow{F}\left(y\right) = \sup F^{-1}\left(\left(-\infty, y\right]\right) \geq x.$
    When $\overleftarrow{F}\left(y\right) > x,$ by the definition of the supremum we claim that there exists an $x^\prime > x$ such that $F\left(x^\prime\right) \leq y$ and by monotonicity of $F,$ we have $F\left(x\right) \leq y.$
    By this same logic whence $\overleftarrow{F}\left(y\right) = x,$ we get $\forall \epsilon > 0, \exists x_\epsilon \in \left(x - \epsilon, x\right] \ni F\left(x_\epsilon\right) \leq y.$ Then, if $F\left(x\right) \leq y,$ 
    we are good as $F\left(x-\right) \leq F\left(x\right) \leq y.$ On the other hand, if $F\left(x\right) > y,$ then we know each $x_\epsilon$ above must lie in $\left(x - \epsilon, x\right)$ and
    by monotonicity of $F$ again, we have $F\left(x-\right) \leq y.$

    Up to the next item, $\forall y \in \left(F\left(x-\right), F\left(x\right)\right] \cap \left(0, 1\right) \Rightarrow F\left(x-\right) < y \leq F\left(x\right).$ Suppose $x$ is not a lower bound for $F^{-1}\left(\left[y, +\infty\right)\right),$ 
    then there exists an $x^\prime < x$ such that $F\left(x^\prime\right) \geq y$ and this further tells us $F\left(x-\right) \geq F\left(x^\prime\right) \geq y$ which is a contradiction. 
    Now for each $\epsilon > 0, \exists x \in \left[x, x + \epsilon\right) \ni x \in F^{-1}\left(\left[y, +\infty\right)\right),$ which implies $x$ is greatest such, i.e., $\overrightarrow{F}\left(y\right) = x.$

    On the other hand, $\forall y \in \left[F\left(x-\right), F\left(x\right)\right) \cap \left(0, 1\right) \Rightarrow F\left(x-\right) \leq y < F\left(x\right).$ Suppose $x$ is not an upper bound for $F^{-1}\left(\left[-\infty, y\right)\right)$ 
    then there exists an $x^\prime > x$ such that $F\left(x^\prime\right) \leq y$ and by monotonicity of $F, F\left(x\right) \leq F\left(x^\prime\right) \leq y,$ which leads to a contradiction. As $F\left(x-\right) \leq y,$ 
    for each $\epsilon > 0, \exists x_\epsilon \in \left(x - \epsilon, x\right) \ni F\left(x_\epsilon\right) \leq y \Rightarrow x_\epsilon \in F^{-1}\left(\left[-\infty, y\right)\right),$ which implies 
    $\overleftarrow{F}\left(y\right) = \sup F^{-1}\left(\left[-\infty, y\right)\right) = x.$

    Then for the eighth item, for each $y^\prime \in \left(0, 1\right), \epsilon > 0, F\left(\overrightarrow{F}\left(y^\prime\right) - \epsilon\right) \geq y$ 
    would lead to $\overrightarrow{F}\left(y^\prime\right) - \epsilon \in F^{-1}\left(\left[y, +\infty\right)\right)$ so that $\overrightarrow{F}\left(y^\prime\right) - \epsilon \geq \inf F^{-1}\left(\left[y, +\infty\right)\right) = \overrightarrow{F}\left(y\right)$
    contradicting the fact that $\epsilon > 0.$ For the given $\epsilon > 0$ by the definition of infimum we know there exists a
    $x_\epsilon \in \left[\overrightarrow{F}\left(y\right), \overrightarrow{F}\left(y\right) + \epsilon\right)$ such that $F\left(x_\epsilon\right) \geq y$ 
    so by right continuity of $F$ at $\overrightarrow{F}\left(y\right),$ we then infer that $F\left(\overrightarrow{F}\left(y\right)\right) \geq y.$

    Analogously, for the given $\epsilon > 0,$ by the definition of the supremum, there exists an 
    $x_\epsilon \in \left(\overleftarrow{F}\left(y\right) - \epsilon, \overrightarrow{F}\left(y\right)\right]$ such that $F\left(x_\epsilon\right) \leq y$ and the monotonicity of $F$ yields 
    $F\left(\overleftarrow{F}\left(y\right) - \epsilon\right) \leq F\left(x_\epsilon\right) \leq y.$ The ``In particular'' part of the theorem follows by letting $\epsilon \downarrow 0$ in those proved expressions.

    The (i) of last item is Lemma 2, with its immediate consequence, of \cite{brockwell2007universal}.

    Towards (ii), for each $z \in \left(0, 1\right),$
    \begin{eqnarray}
        F_Z\left(z\right) = P\left(\left\{ Z \leq z \right\}\right) &=& E\left(P\left(\left\{ Z \leq z \right\}\ |\ X\right)\right) \nonumber \\
                                                                    &=& E\left(P\left(\left\{ \mu_F\left(\left(-\infty, X\right)\right) + \mu_F\left(\left\{ X \right\}\right)U \leq z \right\}\ |\ X\right)\right) \nonumber \\
                                                                    &=& E\left(P\left(\left\{ \mu_F\left(\left(-\infty, X\right)\right) + \mu_F\left(\left\{ X \right\}\right)U \leq z \right\} \cap \left\{ X \in D_F \right\}\ |\ X\right)\right) \nonumber \\
                                                                    &+& E\left(P\left(\left\{ \mu_F\left(\left(-\infty, X\right]\right) \leq z \right\} \cap \left\{ X \not\in D_F \right\}\ |\ X\right)\right) \nonumber \\
                                            \left((\ref{(1)})\ \mathrm{to}\ (\ref{(3)})\ \mathrm{of}\ 6\ \mathrm{above}\right) \Rightarrow &=& E\left(P\left(\left\{ U \leq \frac{z - \mu_F\left(\left(-\infty, X\right)\right)}{\mu_F\left(\left\{ X \right\}\right)} \right\} \cap \left\{ X \in D_F \right\}\ |\ X\right)\right) \nonumber \\
                                                                    &+& E\left(P\left(\left\{ X \leq \overleftarrow{F}\left(z\right) \right\} \cap \left\{ X \not\in D_F \right\}\ |\ X\right)\right) \nonumber \\
                          \left(U \independent X\right) \Rightarrow &=& E\left(H\left(\frac{z - \mu_F\left(\left(-\infty, X\right)\right)}{\mu_F\left(\left\{ X \right\}\right)}\right)\mathbb{I}_{\left\{ X \in D_F \right\}}\right) + E\left(\mathbb{I}_{\left\{ X \in \left(-\infty, \overleftarrow{F}\left(z\right)\right] - D_F\right\}}\right) \nonumber \\
                                                                    &=& \int_{\mathbb{R}} \left(H\left(\frac{z - \mu_F\left(\left(-\infty, x\right)\right)}{\mu_F\left(\left\{ x \right\}\right)}\right)\mathbb{I}_{D_F}\left(x\right) + \mathbb{I}_{\left(-\infty, \overleftarrow{F}\left(z\right)\right] - D_F}\left(x\right)\right) \mu_F\left(dx\right).  \nonumber \\ \label{(43)} 
    \end{eqnarray}

    With identical conditioning techniques by changing $\leq$ to $<$ throughout, we could also derive
    \begin{eqnarray}
        F_Z\left(z-\right) = P\left(\left\{ Z < z \right\}\right) &=& E\left(P\left(\left\{ Z < z \right\}\ |\ X\right)\right) \nonumber \\
                                                                  &=& E\left(P\left(\left\{ \mu_F\left(\left(-\infty, X\right)\right) + \mu_F\left(\left\{ X \right\}\right)U < z \right\}\ |\ X\right)\right) \nonumber \\
                                                                  &=& E\left(P\left(\left\{ \mu_F\left(\left(-\infty, X\right)\right) + \mu_F\left(\left\{ X \right\}\right)U < z \right\} \cap \left\{ X \in D_F \right\}\ |\ X\right)\right) \nonumber \\
                                                                  &+& E\left(P\left(\left\{ \mu_F\left(\left(-\infty, X\right]\right) < z \right\} \cap \left\{ X \not\in D_F \right\}\ |\ X\right)\right) \nonumber \\
        \left((\ref{(42)})\ \mathrm{of}\ 6\ \mathrm{above}\right) \Rightarrow &=& E\left(P\left(\left\{ U < \frac{z - \mu_F\left(\left(-\infty, X\right)\right)}{\mu_F\left(\left\{ X \right\}\right)} \right\} \cap \left\{ X \in D_F \right\}\ |\ X\right)\right) \nonumber \\
                                                                  &+& E\left(P\left(\left\{ X < \overrightarrow{F}\left(z\right) \right\} \cap \left\{ X \not\in D_F \right\}\ |\ X\right)\right) \nonumber \\
                        \left(U \independent X\right) \Rightarrow &=& E\left(H\left(\left[\frac{z - \mu_F\left(\left(-\infty, X\right)\right)}{\mu_F\left(\left\{ X \right\}\right)}\right]-\right)\mathbb{I}_{\left\{ X \in D_F \right\}}\right) + E\left(\mathbb{I}_{\left\{ X \in \left(-\infty, \overrightarrow{F}\left(z\right)\right) - D_F \right\}}\right) \nonumber \\
                                                                  &=& \int_{\mathbb{R}} \left(H\left(\left[\frac{z - \mu_F\left(\left(-\infty, x\right)\right)}{\mu_F\left(\left\{ x \right\}\right)}\right]-\right)\mathbb{I}_{D_F}\left(x\right) + \mathbb{I}_{\left(-\infty, \overrightarrow{F}\left(z\right)\right) - D_F}\left(x\right)\right) \mu_F\left(dx\right),  \nonumber \\\label{(44)} 
    \end{eqnarray} and by subtracting~(\ref{(44)}) off of~(\ref{(43)}) and noting $2$ above ($\overrightarrow{F}\left(z\right) \leq \overleftarrow{F}\left(z\right)$) we obtain 
    \begin{eqnarray}
        && P\left(\left\{ Z = z \right\}\right) \nonumber \\
        &=& \int_{\mathbb{R}} \left(\left\{H\left(\frac{z - \mu_F\left(\left(-\infty, x\right)\right)}{\mu_F\left(\left\{ x \right\}\right)}\right) - H\left(\left[\frac{z - \mu_F\left(\left(-\infty, x\right)\right)}{\mu_F\left(\left\{ x \right\}\right)}\right]-\right)\right\}\mathbb{I}_{D_F}\left(x\right) + \mathbb{I}_{\left[\overrightarrow{F}\left(z\right), \overleftarrow{F}\left(z\right)\right] - D_F}\left(x\right)\right) \mu_F\left(dx\right). \nonumber\\ \label{(45)}
    \end{eqnarray}

    Now, for the previously fixed $z \in \left(0, 1\right),$
    \begin{itemize}
        \item[(i)] $\overrightarrow{F}\left(z\right) = \overleftarrow{F}\left(z\right).$

            When $F$ is (left)-continuous at $\overrightarrow{F}\left(z\right),$ then $\left[\overrightarrow{F}\left(z\right), \overleftarrow{F}\left(z\right)\right] - D_F = \left[\overrightarrow{F}\left(z\right), \overleftarrow{F}\left(z\right)\right] = \left\{ \overrightarrow{F}\left(z\right) \right\} = \varnothing\ \left[\mu_F\right]$ so that $\mathbb{I}_{\left[\overrightarrow{F}\left(z\right), \overleftarrow{F}\left(z\right)\right] - D_F} = 0\ \left[\mu_F\right].$ 
            Otherwise, when $F$ jumps at $\overrightarrow{F}\left(z\right),$ then $\left[\overrightarrow{F}\left(z\right), \overleftarrow{F}\left(z\right)\right] - D_F = \varnothing$ so that $\mathbb{I}_{\left[\overrightarrow{F}\left(z\right), \overleftarrow{F}\left(z\right)\right] - D_F} = 0.$
        \item[(ii)] $\overrightarrow{F}\left(z\right) < \overleftarrow{F}\left(z\right).$

            Since one can for free remove countably many singletons outside of $D_F$ under $\mu_F,$ in this scenario we see that $\left[\overrightarrow{F}\left(z\right), \overleftarrow{F}\left(z\right)\right] - D_F = \left(\overrightarrow{F}\left(z\right), \overleftarrow{F}\left(z\right)\right)\ \left[\mu_F\right]$ for which by $3$ above 
        \begin{eqnarray}
            \mu_F\left(\left(\overrightarrow{F}\left(z\right), \overleftarrow{F}\left(z\right)\right)\right) &=& F\left(\overleftarrow{F}\left(z\right)-\right) - F\left(\overrightarrow{F}\left(z\right)\right) = z - z = 0, \nonumber
        \end{eqnarray} which again implies $\mathbb{I}_{\left[\overrightarrow{F}\left(z\right), \overleftarrow{F}\left(z\right)\right] - D_F} = 0\ \left[\mu_F\right].$
    \end{itemize}

    Thereby we simplify~(\ref{(45)}) towards
    \begin{eqnarray}
        P\left(\left\{ Z = z \right\}\right) &=& \int_{\mathbb{R}} \left\{H\left(\frac{z - \mu_F\left(\left(-\infty, x\right)\right)}{\mu_F\left(\left\{ x \right\}\right)}\right) - H\left(\left[\frac{z - \mu_F\left(\left(-\infty, x\right)\right)}{\mu_F\left(\left\{ x \right\}\right)}\right]-\right)\right\}\mathbb{I}_{D_F}\left(x\right) \mu_F\left(dx\right), \nonumber\\ \label{(46)} \\
                                             &=& \sum_{x \in D_F} \left\{H\left(\frac{z - \mu_F\left(\left(-\infty, x\right)\right)}{\mu_F\left(\left\{ x \right\}\right)}\right) - H\left(\left[\frac{z - \mu_F\left(\left(-\infty, x\right)\right)}{\mu_F\left(\left\{ x \right\}\right)}\right]-\right)\right\} \mu_F\left(\left\{ x \right\}\right). \nonumber
    \end{eqnarray}

    For the special case of $z = 0,$ it is clear that $F_Z\left(0-\right) = P\left(\left\{ Z < 0 \right\}\right) = 0$ due to $U > 0\ \left[P\right]$ and in the previous development to~(\ref{(43)}) we simply observe that 
    $\left\{ \mu_F\left(\left(-\infty, X\right]\right) \leq 0 \right\} = \left\{ X \in F^{-1}\left(\left\{ 0 \right\}\right) \right\} = \varnothing\ \left[P\right]$ by (ii) of $5$ above and also $\left\{ U \leq -\frac{\mu_F\left(\left(-\infty, X\right)\right)}{\mu_F\left(\left\{ X \right\}\right)} \right\} \cap \left\{ X \in D_F \right\} = \varnothing\ \left[P\right]$ 
    so that finally we have $F_Z\left(0\right) = P\left(\left\{ Z \leq 0 \right\}\right) = 0$ as well and hence $P\left(\left\{ Z = 0 \right\}\right) = 0.$

    On the other hand, for $z\geq 1,$ following the derivation to~(\ref{(43)}) we note that $\left\{ \mu_F\left(\left(-\infty, X\right]\right) \leq 1 \right\} = \Omega$ so that 
    \begin{eqnarray}
        F_Z\left(z\right) = P\left(\left\{ Z \leq z \right\}\right) &=& \int_{\mathbb{R}} \left(H\left(\frac{z - \mu_F\left(\left(-\infty, x\right)\right)}{\mu_F\left(\left\{ x \right\}\right)}\right)\mathbb{I}_{D_F}\left(x\right) + \mathbb{I}_{D_F^\mathsf{c}}\left(x\right)\right) \mu_F\left(dx\right), \nonumber
    \end{eqnarray} whereas for~(\ref{(44)}) by (ii) of $5$ again $\left\{ \mu_F\left(\left(-\infty, X\right]\right) < 1 \right\} \cap \left\{ X \not\in D_F \right\} = \left\{ \mu_F\left(\left(-\infty, X\right]\right) \leq 1 \right\} \cap \left\{ X \not\in D_F \right\} = \left\{ X \not\in D_F \right\}\ \left[\mu_F\right]$ which leads to
    \begin{eqnarray}
        F_Z\left(z-\right) = P\left(\left\{ Z < z \right\}\right) &=& \int_{\mathbb{R}} \left(H\left(\left[\frac{z - \mu_F\left(\left(-\infty, x\right)\right)}{\mu_F\left(\left\{ x \right\}\right)}\right]-\right)\mathbb{I}_{D_F}\left(x\right) + \mathbb{I}_{D_F^\mathsf{c}}\left(x\right)\right) \mu_F\left(dx\right). \nonumber
    \end{eqnarray}

    Combining two expressions above we know that~(\ref{(46)}) holds for $z \geq 1$ as well. Since the condition $H\left(0\right) = 0$ necessarily implies $D_H \subseteq \left(0, +\infty\right), z \in D_{F_Z}$ if and only if there exists an $x \in D_F$ such that $H\left(\frac{z - \mu_F\left(\left(-\infty, x\right)\right)}{\mu_F\left(\left\{ x \right\}\right)}\right) - H\left(\left[\frac{z - \mu_F\left(\left(-\infty, x\right)\right)}{\mu_F\left(\left\{ x \right\}\right)}\right]-\right) > 0,$ 
    or equivalently $\frac{z - \mu_F\left(\left(-\infty, x\right)\right)}{\mu_F\left(\left\{ x \right\}\right)} \in D_H.$ To restate this fact in a simple form we then have the set equality as in~(\ref{(47)}).
\end{proof}

\section{Proof of Lemma~\ref{lma::3}} \label{pflma::3}

\begin{proof}
    By Theorem $2.8.18$ on pp. $152$ of Section $2.8$ along with our note about directional limited property in $6$ above, 
    by taking the singleton $\left\{X\right\}$, for which $\rho$ is directionally limited at, as the countable family that unions to $X$ required there.

    Denote $n=\dim_F(X)<+\infty$, by Corollary 5 in Section $13.3$ on pp. $260$ of \cite{royden1988real}
    or even Theorem $1.21$ along with the comments in Section $1.19$ on pp. $16$ of \cite{rudin1991functional}
    that the following claim holds so long as $(X, \tau)$ is a topological vector space, including the special case of a normed vector space $\left(X, \left\lVert\cdot \right\rVert\right)$, 
    for each Hamel basis $B=\left\{e_i\right\}_{1\leq i \leq n}$ of $X$, the map
    \begin{align}
        \varphi_B: F^n &\rightarrow X \nonumber \\
    (k_1, k_2, \cdots, k_n) &\mapsto \sum_{i=1}^nk_ie_i \nonumber
    \end{align} yields a homeomorphic isomorphism between $\left(F^n, \left\lVert \cdot \right\rVert_p\right)$ and $\left(X, \left\lVert\cdot \right\rVert\right)$ with any choice of $p\in [1, +\infty]$.

    Since $F\in \left\{\mathbb{R}, \mathbb{C}\right\}$ which is separable, its finite Cartesian product $\left(F^n, \left\lVert \cdot \right\rVert_\infty\right)$ is separable for its induced sup metric, 
    and hence so is $\left(X, \left\lVert\cdot \right\rVert\right)$. We thus obtain the desired conclusion for the first bullet.

    For the second bullet, first of all by Example $11.4$ on pp. $168$ and also Example $10.1$ on pp. $158$ 
    of \cite{billingsley1995probability}, we know that 
    $\mathcal{R}_n$ is a semi-ring, which generates the Borel sets on $\left(\mathbb{R}^n, \tau\right)$: $\sigma\left(\mathcal{R}_n\right)=\sigma\left(\tau\right)$.

    Then by the proof of Theorem $12.5$ from pp. $178$--$179$ we know that the defined $\mu_F$ is finitely additive and countably sub-additive on $\mathcal{R}_n$. 
    Through the proof of Theorem $11.3,$ the Carath\'eodory extension theorem, we know $\sigma\left(\tau\right)=\sigma\left(\mathcal{R}_n\right)\subseteq \mathcal{L}_{\mu^\star}$ 
    such that $\mu^\star_F$ is an outer measure on $\mathbb{R}^n$ for which $\mu^\star_F|_{\mathcal{L}_{\mu^\star_F}}$ is a measure and also $\mu^\star_F|_{\mathcal{R}_n}=\mu_F$.

    This discussion confirms the (i) of $3$ in the groups of definitions above, and now it suffices to establish the Borel regularized version of $\mu^\star_F$, $\nu^\star_F$, coincides with itself so that we obtain (ii) of $3.$
    By the monotonicity of $\mu^\star_F$ on $\mathcal{P}\left(\mathbb{R}^n\right)$, 
    for each $E\in \mathcal{P}\left(\mathbb{R}^n\right)$, $\mu^\star_F\left(E\right)$ is a lower bound of $\left\{ \mu^\star_F\left(B\right) : E \subseteq B, B \in \sigma\left(\tau\right) \right\},$ 
    and hence $\mu^\star_F\left(E\right)\leq \nu^\star_F\left(E\right)$.

    On the other hand, for each $\left\{E_n\right\}_{n=1}^\infty \subseteq \mathcal{R}_n, E\subseteq \cup_{n=1}^\infty E_n$, we have
    \begin{eqnarray}
        \mu^\star_F\left(\cup_{n=1}^\infty E_n\right) &\leq& \sum_{n=1}^\infty \mu^\star_F\left(E_n\right) \nonumber \\
                                                      &=& \sum_{n=1}^\infty \mu_F\left(E_n\right), \label{(5)}
    \end{eqnarray} where the first inequality is due to the $\sigma$-subadditivity of $\mu^\star_F$ and second by $\mu^\star_F|_{\mathcal{R}_n}=\mu_F$.

    Observing that $\mathcal{R}_n\subseteq \sigma\left(\mathcal{R}_n\right)=\sigma\left(\tau\right)$ so that $\cup_{n=1}^\infty E_n\in \sigma\left(\tau\right)$, inequality~(\ref{(5)}) above simply tells us
    \begin{eqnarray}
        \nu^\star_F\left(E\right) = \inf \left\{ \mu^\star_F\left(B\right) : E \subseteq B, B \in \sigma\left(\tau\right) \right\} &\leq& \inf \left\{ \mu_F^\star\left(\cup_{n = 1}^\infty E_n\right) : \left\{ E_n \right\}_{n = 1}^\infty \subseteq \mathcal{R}_n, E \subseteq \cup_{n = 1}^\infty E_n \right\} \nonumber \\
                                                                                                                                   &\leq& \inf \left\{ \sum_{n = 1}^\infty \mu_F\left(E_n\right) : \left\{ E_n \right\}_{n = 1}^\infty \subseteq \mathcal{R}_n, E \subseteq \cup_{n = 1}^\infty E_n \right\} \nonumber \\
                                                                                                                                   &=& \mu^\star_F\left(E\right), \nonumber
    \end{eqnarray} implying $\mu^\star_F=\nu^\star_F$ and thus concludes the Borel regularity of $\mu^\star_F$ on $\left(\mathbb{R}^n, \tau\right)$.

    For the ``In particular'' part of the claim, by the discussion above the statement of Theorem $12.5$ for each finite measure $\nu : \sigma\left(\tau\right) \rightarrow \left[0, +\infty\right),$ 
    we shall associate it with 
        \begin{align}
            F_\nu : \mathbb{R}^n &\rightarrow \mathbb{R} \nonumber \\
                    \mathbf{x} &\mapsto \nu \left(\times_{i=1}^n(-\infty, x_i]\right).\nonumber
        \end{align}

    Given any sequence $\left\{\mathbf{x}^{(k)}\right\}_{k=1}^\infty\subseteq \mathbb{R}^n$ along with a point $\mathbf{x}\in \mathbb{R}^n$ such that 
    $\forall i = 1, 2, \cdots, n, x^{(k)}_i\downarrow x_i$, we know $\times_{i=1}^n(-\infty, x_i^{(k)}]\downarrow \times_{i=1}^n(-\infty, x_i]$ as $K\rightarrow \infty$
    and thus as $\nu$ is finite so that it is continuous from above at $\times_{i=1}^n(-\infty, x_i]$ we know $F\left(\mathbf{x}^{\left(k\right)}\right) = \nu\left(\times_{i = 1}^n \left(-\infty, x_i^{\left(k\right)}\right]\right) \downarrow \nu\left(\times_{i = 1}^n \left(-\infty, x_i\right] \right) = F\left(\mathbf{x}\right)$ as $k \rightarrow \infty.$

    Additionally, for each $R\in \mathcal{R}_n$, $R=\times_{i=1}^n (a_i, b_i]$ for some $-\infty<a_i\leq b_i<+\infty$, $i=1, 2, \cdots, n$, then
    \begin{eqnarray}
        0\leq \nu\left(R\right) &=& \int_{\mathbb{R}^n} \mathbb{I}_R d\nu = \int_{\mathbb{R}^n} \prod_{i=1}^n \mathbb{I}_{\left(a_i, b_i\right]} d\nu \nonumber \\
                                &=& \int_{\mathbb{R}^n} \prod_{i = 1}^n \mathbb{I}_{\left(-\infty, b_i\right] - \left(-\infty, a_i\right]} d\nu \nonumber \\
    \left(\left(-\infty, a_i\right] \subseteq \left(-\infty, b_i\right] \right) \Rightarrow &=& \int_{\mathbb{R}^n} \prod_{i = 1}^n \left(\mathbb{I}_{\left(-\infty, b_i\right]} - \mathbb{I}_{\left(-\infty, a_i\right]}\right) d\nu \nonumber \\
                                &=& \int_{\mathbb{R}^n} \sum_{\mathbf{x} \in \times_{i = 1}^n \left\{ a_i, b_i \right\}} \left(-1\right)^{\# \left\{ i : x_i = a_i \right\}} \mathbb{I}_{\times_{i = 1}^n \left(-\infty, x_i\right]} d\nu \nonumber \\
    \left(\mathrm{linearity}, \nu\ \mathrm{is\ finite}\right) &=& \sum_{\mathbf{x} \in \times_{i = 1}^n \left\{ a_i, b_i \right\}} \left(-1\right)^{\# \left\{i : x_i = a_i\right\}} \nu\left(\times_{i = 1}^n \left(-\infty, x_i\right]\right) \nonumber \\
                                &=& \Delta_R F_\nu. \label{(6)}
    \end{eqnarray}

    Now we know $F_\nu$ in fact satisfies conditions (i) and (ii) for the main addendum of this claim, 
    and also by definition in the statement of this second claim along with~(\ref{(6)}) above 
    we know $\mu_{F_\nu} = \nu|_{\mathcal{R}_n}.$ By the proof through Carath\'eodory extension theorem just now, $\mu_{F_\nu}^\star|_{\sigma\left(\tau\right)}$ is also a measure on $\sigma\left(\tau\right)$ such that $\mu_{F_\nu}^\star|_{\mathcal{R}_n} = \mu_{F_\nu}.$ 
    Then we actually obtain $\mu_{F_\nu}^\star|_{\mathcal{R}_n} = \nu|_{\mathcal{R}_n}$ 
    so by the uniqueness in the Carath\'eodory extension theorem we conclude this ``In particular'' part of the proof that $\mu_{F_\nu}^\star|_{\sigma\left(\tau\right)} = \nu,$ 
    for which $\mu_{F_\nu}^\star$ is Borel regular.

    Last but not least, first notice that given any bounded set $E\subseteq \mathbb{R}^n$ we can find a rectangle $R\in \mathcal{R}_n$ such that $E\subseteq R$ 
    then by monotonicity of $\mu_F^\star$ we have $\mu_F^\star(E)\leq \mu_F^\star(R)$. Furthermore, as $\mu_F^\star|_{\mathcal{R}_n}=\mu_F$, this implies $\mu_F^\star(R)=\mu_F(R)=\Delta_RF<+\infty$ 
    and therefore with this chain of inequality we obtain $\mu_F^\star(E)<+\infty$.

    Henceforth, $\mu^\star_F$ is a Borel regular outer measure that assigns finite value to finite diameter subsets of $\mathbb{R}^n$. Additionally, for each $p\in[1, +\infty]$, $\left(\mathbb{R}^n, \left\lVert\cdot \right\rVert_p\right)$ is an $n$-dimensional normed linear space 
    so appealing to the first item proved $\mathfrak{V}_p$ thus defined is a $\mu^\star_F$-Vitali relation.

    With the above and also the $\mu^\star_F|_{\mathcal{P}\left(X\right)\cap E}$-integrability condition on $f$ for each bounded $\mu_F^\star$-measurable $E\in \mathcal{P}\left(\mathbb{R}^n\right)$ 
    allows us to resort to Theorem $2.9.8$ on pp. $156$ in Section $2.9$ of \cite{federer2014geometric}
    to end this proof.
\end{proof}

\end{appendices}


\clearpage

\begin{thebibliography}{}

    \bibitem[Billingsley, 1995]{billingsley1995probability}
    Billingsley, P. (1995).
    \newblock {\em Probability and Measure}.
    \newblock Wiley Series in Probability and Mathematical Statistics. John Wiley \& Sons, New York, NY, United States of America, 3rd edition.
    
    \bibitem[Brockwell, 2007]{brockwell2007universal}
    Brockwell, A. (2007).
    \newblock Universal residuals: A multivariate transformation.
    \newblock {\em Statistics \& Probability Letters}, 77(14):1473--1478.
    
    \bibitem[Cai et~al., 2022]{cai2022distribution}
    Cai, Z., Li, R., and Zhang, Y. (2022).
    \newblock A distribution free conditional independence test with applications to causal discovery.
    \newblock {\em Journal of Machine Learning Research}, 23(85):1--41.
    
    \bibitem[Federer, 1996]{federer2014geometric}
    Federer, H. (1996).
    \newblock {\em Geometric Measure Theory}, volume 153 of {\em Grundlehren der mathematischen Wissenschaften}.
    \newblock Springer, Berlin, Heidelberg, Germany.
    
    \bibitem[Klenke, 2020]{klenke2013probability}
    Klenke, A. (2020).
    \newblock {\em Probability Theory: A Comprehensive Course}.
    \newblock Universitext. Springer, Cham, Switzerland, 3rd edition.
    
    \bibitem[Royden and Fitzpatrick, 2010]{royden1988real}
    Royden, H.~L. and Fitzpatrick, P.~M. (2010).
    \newblock {\em Real Analysis}, volume~32.
    \newblock Pearson Education, 4th edition.
    
    \bibitem[Rudin, 1991]{rudin1991functional}
    Rudin, W. (1991).
    \newblock {\em Functional Analysis}.
    \newblock International Series in Pure and Applied Mathematics. McGraw-Hill, New York, NY, United States of America, 2nd edition.
    
    \bibitem[Sz{\'e}kely et~al., 2007]{szekely2007measuring}
    Sz{\'e}kely, G.~J., Rizzo, M.~L., and Bakirov, N.~K. (2007).
    \newblock Measuring and testing dependence by correlation of distances.
    \newblock {\em The Annals of Statistics}, 35(6):2769--2794.
    
    \bibitem[Titchmarsh, 1939]{titchmarsh1939theory}
    Titchmarsh, E.~C. (1939).
    \newblock {\em The Theory of Functions}.
    \newblock Oxford University Press, London E.C.4, Great Britain, 2nd edition.
    
\end{thebibliography}
\end{document}